\documentclass[15pt, twoside]{article}
\usepackage{amsmath,amsthm,amssymb}
\usepackage{times}
\usepackage{enumerate}
\usepackage{extarrows}
\usepackage{color}
\usepackage{tabularx}
\usepackage{multirow}
\usepackage{graphicx}

\def\titlerunning#1{\gdef\titrun{#1}}
\makeatletter
\def\author#1{\gdef\autrun{\def\and{\unskip, }#1}\gdef\@author{#1}}
\def\address#1{{\def\and{\\\hspace*{18pt}}\renewcommand{\thefootnote}{}%
\footnote {#1}}%
\markboth{\autrun}{\titrun}}
\makeatother
\def\email#1{e-mail: #1}
\def\subjclass#1{{\renewcommand{\thefootnote}{}%
\footnote{\emph{Mathematics Subject Classification (2020):} #1}}}

\newtheorem{theorem}{Theorem}[section]
\newtheorem{lemma}{Lemma}[section]
\newtheorem{definition}{Definition}[section]
\newtheorem{proposition}{Proposition}[section]
\newtheorem{remark}{Remark}[section]
\newtheorem{corollary}{Corollary}[section]
\newtheorem{example}{Example}[section]
\numberwithin{equation}{section}

\newtheorem{TheoremX}{Theorem}
\newtheorem{CorollaryX}{Corollary}

\newcommand{\R}{\mathbb{R}}
\newcommand{\eps}{\varepsilon}
\newcommand{\Proof}{\begin{proof}}
\newcommand{\End}{\end{proof}}

\newcommand{\cS}{\mathcal{S}^{c,-}}
\newcommand{\cSS}{\mathcal{SS}^c}
\newcommand{\Tc}{T^{c,-}_t}

\newcommand{\Bc}{\mathcal{B}^{c,-}}
\newcommand{\chc}{(CH$_c$)\,\,}
\newcommand{\spt}{\text{supp}\,}

\numberwithin{equation}{section}

\newcommand{\PreserveBackslash}[1]{\let\temp=\\#1\let\\=\temp}
\newcolumntype{C}[1]{>{\PreserveBackslash\centering}p{#1}}
\newcolumntype{R}[1]{>{\PreserveBackslash\raggedleft}p{#1}}
\newcolumntype{L}[1]{>{\PreserveBackslash\raggedright}p{#1}}

\newcolumntype{I}{!{\vrule width 1pt}}
\newlength\savedwidth

\begin{document}


\baselineskip=15pt


\titlerunning{Inhomogeneous discounted Hamilton-Jacobi equations}

\title{\Large On the inhomogeneous discounted Hamilton-Jacobi equations}

\author{Liang Jin, Jun Yan \,and\, Kai Zhao}

\date{\today}

\maketitle

\address{Liang Jin : School of Mathematics and Statistics, Nanjing University of Science and Technology, Nanjing 210094, China;
 \email{jl@njust.edu.cn }
		\and Jun Yan: School of Mathematical Sciences, Fudan University, Shanghai 200433, China;
		\email{yanjun@fudan.edu.cn}
		\and
		Kai Zhao:  School of Mathematical Sciences, Tongji University, Key Laboratory of Intelligent Computing and Applications (Ministry of Education),Shang-hai 200433, China;
		\email{zhaokai93@tongji.edu.cn}}

\begin{abstract}
  In this paper, we study the family of inhomogeneous discounted Hamilton-Jacobi equations
  \begin{equation}\label{hjs1}
  \lambda(x)u+h(x,d_x u)=c \quad \tag{$\ast$}
  \end{equation}
  on a closed manifold $M$ with a non-identically vanishing discount factor $\lambda(x)$. There is a critical value $c_0\in[-\infty,\infty)$ such that \eqref{hjs1} admits a viscosity solution if $c>c_0$ and no solution if $c<c_0$. Inspired by the recent development \cite{RWY} on the stability theory of viscosity solutions, we show that the equation admits an asymptotically stable solution if and only if $c>c_0$. In this case, we determine the basin of the stable solution and investigate the long-time behavior of the solution semigroup associated with \eqref{hjs1}. In particular, we obtain a formula relating the lowest convergence rate of the solution semigroup near a stable solution to a minimizing problem concerning the integral of $\lambda$ over Mather measures supporting on the $1$-graph of the stable solution. This formula has two direct consequences: 1) it yields a description of their asymptotic behavior as $c$ tends to infinity; 2) it helps to classify the ergodic Mather measures and locate their distribution in the phase space. The second consequence leads to a dynamical necessary condition for $c=c_0$.
\end{abstract}

%


\section{Introduction}
\setcounter{equation}{0}
\setcounter{footnote}{0}

Let $M$ be a smooth, connected, compact Riemannian manifold without boundary. As usual, $T^\ast M$ denotes the cotangent bundle of $M$ and $|\cdot|_x$ stands for the dual norm induced by the Riemannian metric on $T^\ast_xM$; $T^\ast M\times\R$ denotes the $1$-jet bundle over $M$ with the global coordinates $(x,p,u)$, where $(x,p)$ are the usual Darboux coordinates on $T^\ast M$ and $u\in\R$. We normalize this metric so that diam$(M)=1$ and consider the family of stationary Hamilton-Jacobi equations
\begin{equation}\label{HJs}
\lambda(x)u+h(x,d_xu)=c,\quad x\in M \quad \tag{dS$_c$}
\end{equation}
and the corresponding evolutionary equations
\begin{equation}\label{HJe}\tag{dE$_c$}
\begin{cases}
\partial_t U+\lambda(x)U+h(x,\partial_x U)=c,\quad\, (x,t)\in M\times(0,+\infty)\\
\hspace{7.7em}U(x,0)\,=\varphi(x),\hspace{1.3em}x\in M
\end{cases}
\end{equation}
parametrized by the right-hand constant $c\in\R$. The above two equations are called \textbf{discounted Hamilton-Jacobi equations}, where $h:T^\ast M\rightarrow\R$ and the \textbf{discount factor} $\lambda:M\rightarrow\R$ are $C^\infty$ functions satisfying
\begin{enumerate}
	\item[(H1)] (Convexity) the Hessian $\frac{\partial^2 h}{\partial p^2}$ is positive definite for all $(x,p)\in T^*M$;
	\item[(H2)] (Superlinearity) for every $K\geqslant 0$, there is $A(K)>0$ such that $h(x,p)\geqslant K|p|_x-A(K)$;
	\item[(H3)] (Nondegeneracy) there exists $x_0\in M$ such that $\lambda(x_0)\neq0$.
\end{enumerate}
The first two assumptions indicate that $h$ is a Tonelli Hamiltonian and the last assumption just asserts that the term $\lambda(x)u$ genuinely occurs in the equation. Due to (H3), we have
\[
\text{either}\quad\textbf{(H3)$_+$ $\displaystyle\,\,\max_{x\in M}\lambda(x)>0$}\hspace{7em}\text{or}\quad\textbf{(H3)$_-$$\displaystyle\,\, \min_{x\in M}\lambda(x)<0$}.
\]
In what follows, we use $C(M,\R)$ to denote the Banach space of continuous functions on $M$ equipped with the standard sup-norm $\|\cdot\|_\infty$, the solutions to \eqref{HJs} and \eqref{HJe} are always understood in the \textbf{viscosity sense}.

\subsection{Brief summary of related work}
Discounted Hamilton-Jacobi equations \eqref{HJs} appeared in the classical literature on the theory of viscosity solutions \cite{CL, Crandall_Evans_Lions1984} and homogenization \cite{LPV_Hom} where the discount factor is \textbf{homogeneous}, that is $\lambda(x)\equiv\lambda>0$. The discounted term $\lambda u$ guarantees the validity of the comparison principle and Perron's method for stationary equations, thus implying the existence and uniqueness of the solution, see \cite[Page 22, 33]{Tran} for a modern treatment. It is well-known that the solution semigroup of \eqref{HJe} converges in $C^0$ sense to the stationary solution exponentially for any continuous initial datum. More refined convergence properties of the solution semigroup were discussed in \cite{AS}.

\vspace{0.5em}
Based on the well-posedness of \eqref{HJs}, \cite{LPV_Hom} developed a technique to solve the ergodic problem of classical Hamilton-Jacobi equations: the authors showed that there is a unique constant $c(h)\in\R$ such that the unique solution $u_\lambda$ of \eqref{HJs} with $c=c(h)$ is uniformly bounded and Lipschitz in $\lambda$. By sending $\lambda$ to $0$, any limit point of $\{u_\lambda\}_{\lambda>0}$ gives a solution of $h(x,d_xu)=c(h)$. By adapting weak KAM theory \cite{Fathi_book} to \eqref{HJs}, it was proved in \cite{DFIZ} that $u_\lambda$ actually converges to a specific solution $u_0$. This process is now called the \textbf{vanishing discount limit} and has attracted much research interests over the last ten years, for instance \cite{CCIZ,CFZZ,IJ,IMT1,IMT2,WYZ,Z,DNYZ}.

\vspace{0.5em}
The characteristic systems of homogeneous Hamilton-Jacobi equations define the so-called \textbf{conformally symplectic} vector fields on $T^\ast M$: the Lie derivative of the canonical 2-form $\omega_{std}=dp\wedge dx$ with respect to such vector fields is $\lambda\,\omega_{std}$. The dynamics of the phase flows of such vector fields are quite different from the Hamiltonian dynamics and exhibit a mixture of dissipative and conservative phenomena \cite{AA}. For instance, \cite{MS} investigated the systems from the perspective of Aubry-Mather theory and proved the existence of global attractors; the invariant submanifolds of the system were explored in \cite{AF}; the higher-dimensional Birkhoff attractors of the systems were studied via the symplectic approach in \cite{AHV}.

\vspace{0.5em}
In this paper, we shall mainly focus on the \textbf{inhomogeneous} case, that is, the discount factor is not constant. From a PDE perspective, if the discount factor is positive everywhere, the well-posedness theory and the large-time behavior of the solution semigroup share similarities with the positive, homogeneous case: Let $T_t^{c,-}$ be the solution semigroup associated to \eqref{HJe} and $\mu=\min_{x\in M}\lambda(x)>0$, then the equation \eqref{HJs} has a unique solution $u_-$ such that for any $\varphi \in C(M,\mathbb{R})$ and $t>0$,
$$
\|T_t^{c,-}\varphi- u_-\|_\infty\leqslant e^{-\mu t}\|\varphi- u_-\|_\infty.
$$
This is also why the equations \eqref{HJs}-\eqref{HJe} did not receive much attention earlier. The situation changed as a new trend emerged, generalizing the vanishing discount problem to a \textbf{selection problem},  for instance \cite{GMT,MT,NYZ}.  In general, the equation $h(x,d_xu)=c(h)$ admits multiple solutions, determined by different static classes. One considers the family of equations
$$
\varepsilon\lambda(x)u+h(x,d_xu)=c(h).
$$
The solution structure of the above equation, which is presented below, helps to pick up a special solution $u_\varepsilon$. By taking the analogous limit process $\lim_{\varepsilon\rightarrow0}u_\varepsilon$, the inhomogeneous discount factor is used to choose the preferred solution among all possible solutions.

\vspace{0.5em}
Another motivation to study the inhomogeneous equations \eqref{HJs}-\eqref{HJe} comes from the wave-particle duality and the associated characteristic systems, which are the \textbf{contact Hamiltonian systems} defined on the 1-jet bundle of $M$. Such systems are natural extensions of conformally symplectic systems and have become a focus of research of thermodynamics and contact geometry. Their variational principle was recently established in \cite{WWY1,WWY2} and \cite{CCWY,CCJWY}, enabling the study of \eqref{HJs} via Aubry-Mather theory. \cite{WWY3} did first work in this direction, assuming the discount factor is positive. The attempt to explore \eqref{HJs}-\eqref{HJe} in more general cases forces us to consider discount factor \textbf{changing signs} on $M$. If both \textbf{(H3)$_\pm$} hold, \cite{JYZ} defined the critical value $c_0\in\R$ for \eqref{HJs} such that the equation
\begin{itemize}
  \item admits a solution if and only if $c\geqslant c_0$,
  \item admits at least two solutions if $c>c_0$.
\end{itemize}
In \cite{NW}, the authors continued to investigate the solution structure of such equations and the large time behavior of the solution semigroup. We refer to Section 1.2 and 3.3 for detailed conclusions. The dynamics of action minimizing orbits of the characteristic system for \eqref{HJe} was studied in \cite{JYZ2}, in which some  interesting phenomena was found to show the irreversible features of the contact Hamiltonian flow comparing to Hamiltonian flow.

\vspace{0.5em}

Apart from the mathematical motivations, an interesting economical background of \eqref{HJs} arise in the recent work of Bertucci, Lasry and Lions \cite{BLL}. In this work, the authors studied bubble-asset prices through the viscous Hamilton-Jacobi equation 
\begin{equation}\label{vis-Hjs}\tag{*}
-\nu\Delta u+\frac{1}{2}\,|d_x u|^2=\lambda(x)u,
\end{equation}
where $\lambda(x)$, playing the role of the gap between the return on crypto and fiat money, may change sign as the monetary condition changes. The diffusion $-\nu\Delta u$ makes their model a second order counterpart of the first order inhomogeneous Hamilton-Jacobi equations studied here. Under suitable assumptions, they prove the existence and uniqueness of the positive stationary solution and convergence of the associated evolution. The standard vanishing viscosity process suggests that, as $\nu$ goes to $0$, the structure of the solutions of \eqref{HJs} is closely related to that of \eqref{vis-Hjs}.

\subsection{Preliminary definitions and observations}

In this part, we recall some definitions and observations that are necessary for the formal statement of our main result. To proceed, we use $\{T^{c,-}_t\}_{t\geqslant0}$ to denote the (backward) \textbf{solution semigroup} associated with \eqref{HJe}: $\{T^{c,-}_t\}_{t\geqslant0}$ acts continuously on $C(M,\R)$ such that $(x,t)\mapsto T^{c,-}_t\varphi(x)$ is the \textbf{unique} solution to \eqref{HJe} on $M\times[0,+\infty)$. Thus $\{T^{c,-}_t\}_{t\geqslant0}$ defines a semiflow on $C(M,\R)$. From this perspective, solutions to \eqref{HJs} are identified with the set of stationary points $\mathcal{S}^{c,-}:=\{u_-\in C(M,\R):T^{c,-}_t u_-=u_-, \forall t\geqslant0\}$ of this semiflow. In the recent works \cite{RWY,RWY2}, the authors succeeded in generalizing the ideas of stability theory for equilibria of finite-dimensional
ODE systems to the semiflow $\{T^{c,-}_t\}_{t\geqslant0}$ defined on the infinite-dimensional space $C(M,\R)$. Indeed, they proposed the following interesting

\begin{definition}\cite[Definition 1.1]{RWY}\label{L-stable/A-stable}
$u_-\in\mathcal{S}^{c,-}$ is said \textbf{Lyapunov stable} if for every $\varepsilon>0$, there is some $\delta>0$ such that for any $\varphi\in C(M,\R)$ with $\|\varphi-u_-\|_\infty<\delta$, $\|T^{c,-}_t\varphi-u_-\|_\infty<\varepsilon$ for every $t\geqslant0$. Moreover, $u_-$ is \textbf{asymptotically stable} if it is Lyapunov stable and there is some $\delta>0$ such that for any $\varphi\in C(M,\R)$ with $\|\varphi-u_-\|_\infty<\delta, \lim_{t\rightarrow+\infty}\|T^{c,-}_t\varphi-u_-\|_\infty=0$.
\end{definition}

The second definition concerns the relation between the solvability of \eqref{HJs} and the right-hand constant $c$.
\begin{definition}\cite[Theorem A]{CIP}-\cite[Section 2.1]{JYZ}
Assume $h(x,p)$ satisfies (H1)-(H2). Define
\begin{equation}\label{dcv}
c_0:=\inf_{\psi\in C^\infty(M,\R)}\sup_{x\in M} \left[\lambda(x)\psi(x)+ h(x,d_x\psi(x))\right]\in[-\infty,+\infty).
\end{equation}
\end{definition}

\begin{remark}
The value $c_0$ is unchanged if we replace $C^\infty(M,\R)$ in the formula \eqref{dcv} by $C^2(M,\R)$ or $C^{1,1}(M,\R)$, see for instance \cite{Be,FS}. If $\lambda\equiv 0$, then, by \cite[Theorem A]{CIP}, $c_0\in\R$ is the \textbf{Ma\~{n}\'{e} critical value} and by \cite{Fathi_book}, \eqref{HJs} admits a solution if and only if $c=c_0$. We shall abuse terminology and call $c_0$ the \textbf{critical value} in our setting.
\end{remark}

In view of the definition of $c_0$, a simple but fundamental consequence is
\begin{proposition}\label{prop:1}
Assume (H1)-(H2). Then
\begin{enumerate}[(1)]
  \item If \eqref{HJs} admits a subsolution, then $c\geqslant c_0$.

  \item $c_0\in\R$ if and only if $\lambda(x)$ vanishes somewhere on $M$.

  \item In addition, if (H3) holds, then \eqref{HJs} admits a solution if and only if it admits a subsolution. In particular, if $c>c_0$, then \eqref{HJs} admits a solution.
\end{enumerate}
\end{proposition}

\begin{remark}
For the first statement: if $c=c_0$ and $\lambda(x)$ changes sign on $M$, then it is proved in \cite{JYZ} that \eqref{HJs} always admits subsolution; however, if $c=c_0$ and $\lambda(x)\geqslant0$, then \eqref{HJs} may admit no subsolution. We give another characterization of the solvability of \eqref{HJs} with non-negative discount factor in Appendix \ref{section-Appendix-b}.
\end{remark}

\begin{remark}
As the action of $T^{c,-}_t$ on the subsolution is increasing in $t$ (see Proposition \ref{mono}), we provide a dynamical interpretation of (1) in Appendix \ref{proof:thm4}, concluding that for $c<c_0, \lim_{t\to +\infty} T_t^{c,-} \varphi=-\infty$ for any $\varphi\in C(M,\R)$.
\end{remark}

\begin{remark}
The last conclusion differs from that in the degenerate case, where $\lambda\equiv0$: where for any $c>c_0$, there are many subsolutions but no solution to \eqref{HJs}.
\end{remark}

The formulation of the next observation relies on a counterpart of $\{T^{c,-}_t\}_{t\geqslant0}$. To define it, we use $\breve{T}^{c,-}_t\varphi$ to denote the solution semigroup associated with the Cauchy problem
\begin{equation}\label{HJe'}\tag{dE$'_c$}
\begin{cases}
\partial_t U-\lambda(x)U+h(x,-\partial_x U)=c,\quad\, (x,t)\in M\times(0,+\infty)\\
\hspace{8.6em}U(x,0)\,=\varphi(x),\hspace{1.2em}x\in M.
\end{cases}
\end{equation}
The \textbf{forward solution semigroup} associated with \eqref{HJe} is $T^{c,+}_t\varphi:=-\breve{T}^{c,-}_t(-\varphi)$. Like its twin brother, $\{T^{c,+}_t\}_{t\geqslant0}$ also defines a semiflow on $C(M,\R)$. The set of stationary points of this semiflow are $\mathcal{S}^{c,+}:=\{v^c_+\in C(M,\R):T^{c,+}_t v^c_+=v^c_+, \forall t\geqslant0\}$. An equivalent description is that $v^c_+\in\mathcal{S}^{c,+}$ if and only if  $-v^c_+$ solves the equation
\begin{equation}\label{HJs'}\tag{dS$'_c$}
-\lambda(x)v+h(x,-d_x v)=c,\quad x\in M.
\end{equation}

\begin{remark}\label{+ to -}
It is worth-noting that if $\lambda(x)$ satisfies \textbf{(H3)$_-$}, then $-\lambda(x)$ satisfies \textbf{(H3)$_+$}. This obvious fact allows us to translate the main theorems proved under \textbf{(H3)$_+$} to similar results under \textbf{(H3)$_-$}.
\end{remark}

Our main results are built upon the following two observations. In these observations, we select certain solutions to the static equations \eqref{HJs} and \eqref{HJs'} that govern on the large-time behavior of solution semigroups. Notice that these observations are essentially known from \cite[Theorem 1-2]{NW} with the additional assumption $(\pm)$. We state the first observation as

\begin{proposition}\label{prop:new-NW-JDE}
Assume \eqref{HJs} admits a subsolution. Then
\begin{enumerate}[(1)]
  \item If {\bf (H3)$_+$} holds, there is a \textbf{maximal} element $u_-^c$ in $\mathcal{S}^{c,-}$ such that for any $\varphi\in C(M,\R)$ with $\varphi \geqslant u_-^c,\\ T_t^{c,-}\varphi$  converges to $u_-^c$  uniformly on $M$ as $t\to +\infty$.

  \item If {\bf (H3)$_-$} holds, there is a \textbf{minimal} element $v_+^c$ in $\mathcal{S}^{c,+}$ such that for any $\varphi\in C(M,\R)$ with $\varphi \leqslant v_+^c,\\ T_t^{c,+}\varphi $ converges to $v_+^c$  uniformly on $M$ as $t\to +\infty$.
\end{enumerate}
\end{proposition}

The second observation concerns the large-time behavior of solution semigroups on the set
\[
\mathcal{SS}^c:=\{\psi\in C^\infty(M,\R)\,:\,\lambda(x)\psi(x)+h(x,d_x\psi(x))<c\}
\]
\begin{proposition} \label{prop:2}
Assume (H1)-(H2) and $c>c_0$.
\begin{itemize}
  \item[(1)] If \textbf{(H3)$_+$} holds, $\displaystyle u^c_-=\lim_{t \to +\infty} T^{c,-}_t\psi>\psi$ for any $\psi\in\mathcal{SS}^c$.
  \item[(2)] If \textbf{(H3)$_-$} holds, $\displaystyle v^c_+=\lim_{t \to +\infty} T^{c,+}_t\psi<\psi$ for any $\psi\in\mathcal{SS}^c$.
\end{itemize}
\end{proposition}
Proposition \ref{prop:1}-\ref{prop:2} are proved in Section \ref{proof:pre-observation}.

\vspace{0.5em}
Now we want to emphasize the role played by \textbf{characteristic system} in the investigation of the dynamics of solution semigroups of Hamilton-Jacobi equations. The characteristic system for the Cauchy problem \eqref{HJe} is the contact Hamiltonian system \eqref{ch} associated with $H(x,p,u)=\lambda(x)u+h(x,p)-c$, see Appendix A \eqref{ch} for the explicit formula. We shall use $\Phi^t_{c}$ to denote the phase flow of \eqref{ch}. Assume that $\mathcal{S}^{c,\pm}\neq\emptyset$. Then (H2) implies that each $u_\pm\in\mathcal{S}^{c,\pm}$ is a Lipschitz function on $M$. It was proved in \cite[Theorem 1.1]{WWY3} that $\mathrm{J}^1_{u_-}/\mathrm{J}^1_{u_+}$ is backward/forward invariant under $\Phi^t_{c}$, that is,
\begin{equation}\label{1jet-inv}
\Phi^t_{c}(\mathrm{J}^1_{u_-})\subset\mathrm{J}^1_{u_-}\text{  for}\,\,\,t\leqslant0,\quad\quad \Phi^t_{c}(\mathrm{J}^1_{u_+})\subset\mathrm{J}^1_{u_+}\text{  for}\,\,\,t\geqslant0.
\end{equation}
Let $\mathcal{P}(T^*M \times \mathbb{R})$ be the space of Borel probability measures on $T^*M \times \mathbb{R}$, we adapt \cite[Page 5,(1.4)]{RWY} as follows
\begin{definition}\label{mather}
$\mathfrak{M}^c_{u_\pm}:= \left\{ \mu \in \mathcal{P}(T^*M \times \mathbb{R}):\spt(\mu)\subset \mathrm{J}^1_{u_\pm},\ (\Phi^t_{c})_{\#}  \mu = \mu,\ \forall t \in \mathbb{R} \right\}$, where $(\Phi^t_{c})_{\#}\mu$ denotes the push-forward of $\mu$ through $\Phi^t_{c}$. By \eqref{1jet-inv} and the Krylov-Bogoliubov method, $\mathfrak{M}^c_{u_\pm}\neq\emptyset$. Any $\mu\in \mathfrak{M}^c_{u_\pm}$ is called a \textbf{Mather measure} associated with $u_\pm$. 
The set of Mather measures for \eqref{ch} is defined as 
$$
\mathfrak M^c:=\overline{\operatorname{co}}\!\left(\bigcup_{u_-\in\mathcal S^{c,-}}\mathfrak M^c_{u_-}\right)=\overline{\operatorname{co}}\!\left(\bigcup_{u_+\in\mathcal S^{c,+}}\mathfrak M^c_{u_+}\right),
$$
 where $\operatorname{co}(A)$ denotes the \textbf{closed convex hull} of $A$.
\end{definition}

\subsection{Statement of the main results}
In this part, we formulate the main theorems of this paper. The main theorems are organized into two aims:
\begin{enumerate}[(1)]
  \item Theorem A$-$C: we study the existence of asymptotically stable solutions of \eqref{HJs} and the dynamical behavior of $\Tc$,
  \item Theorem D: we give information about the distribution of Mather measures associated with \eqref{HJs} when $c=c_0$.
\end{enumerate}
More precisely, under the assumptions (H1)-(H2) and \textbf{(H3)$_+$},
\begin{enumerate}[(a)]
    \item we provide conditions equivalent to the existence of an asymptotically stable solution and if such a solution exists, it is automatically unique (Theorem A).
    \item we characterize the basin of attraction for the unique asymptotically stable solution (Theorem B).
    \item we prove exponential convergence of $T^{c,-}_t\varphi$  throughout the basin, identify the lowest convergence rate by Mather measures, prove that this sharp rate is attained by certain initial data, and describe the rate as $c\to+\infty$ (Theorem C).
    \item assuming \eqref{HJs} admits a solution when $c=c_0$, we classify the Mather measures via the integral of discount factor against them and then we relate their locations to specific solutions of \eqref{HJs}.
\end{enumerate}
In what follows, \textbf{(H3)$_+$} is divided into two subclasses:
\begin{itemize}
    \item[($+$)] $\lambda(x) \geqslant 0$ for all $x \in M$ and $\displaystyle \max_{x\in M} \lambda(x) > 0$.
    \item[($\pm$)] There exist $x_1, x_2 \in M$ such that $\lambda(x_1) > 0$ and $\lambda(x_2) < 0$.
\end{itemize}

Our first main theorem offers a clean answer to the existence of an asymptotic solution to \eqref{HJs}.
\begin{TheoremX}\label{thm:1}
\eqref{HJs} admits an asymptotically stable solution if and only if $c>c_0$. In this case, the asymptotically stable solution to \eqref{HJs} is \textbf{unique} and equals $u^c_-$ found in Proposition \ref{prop:2}.
\end{TheoremX}
\begin{remark}
 Here $u^c_-$ has the following representation, established  in \cite[Theorem 4.8]{Z} and \cite[Proposition 3.19]{DNYZ}: for any $t>0$ and $x\in M$,
  $$
 u^c_-(x)=\min_{ \substack{ \gamma\in \text{AC}([-t,0],M)\\ \gamma(0)=x}} \Big\{   e^{-  \int_{-t}^{0} \lambda(\gamma(\tau)) \, d\tau}u^c_-(\gamma(-t))   +  \int_{-t}^{0} e^{-  \int_{s}^{0} \lambda(\gamma(\tau)) \, d\tau} \left( \, l(\gamma(s), \dot{\gamma}(s)) + c \right) ds  \Big\}
 $$
 and there is a curve $\gamma_x : (-\infty, 0] \to M$ with $\gamma_x(0) = x$ such that
\[
u_-^c(x) = \int_{-\infty}^{0} e^{-  \int_{s}^{0} \lambda(\gamma_x(\tau)) \, d\tau} \left( \, l(\gamma_x(s), \dot{\gamma}_x(s)) + c \right) ds, \quad \forall x\in M.
\]
According to this representation formula, the stable solution appears as the value function of an inhomogeneous discounted optimal control problem, which suggests potential application of our results to global problems involving varying discount factors. This is also one of our motivations to study \eqref{HJs}. 
\end{remark}

The next goal is to determine the \textbf{basin of attraction} of the asymptotically stable solution $u^c_-$, that is
$$
\mathcal{B}^{c,-}:= \{\,\,\varphi\in C(M,\R) :\lim_{t\to +\infty} T_t^{c,-}\varphi=u^c_-\,\,\}.
$$
\begin{TheoremX}\label{thm:2}
Assume $c>c_0$ and that $u^c_-,v^c_+$ are the solutions appearing in Proposition \ref{prop:2}.
\begin{itemize}
  \item[(B1)] If ($+$) holds, then $\mathcal{B}^{c,-}=C(M,\R)$. In this case, we say $u^c_-$ is \textbf{globally asymptotically stable}.

  \item[(B2)] If $(\pm)$ holds, then $\mathcal{B}^{c,-}=\{\varphi \in C(M,\R): \varphi>v^c_+\}$.
\end{itemize}
\end{TheoremX}

To our surprise, under the assumption \textbf{(H3)$_+$}, the asymptotically stable solution to \eqref{HJs} always attracts the initial data in its basin \textbf{exponentially fast}. We improve \cite[Theorem 1.2]{RWY} to give a formula for the lowest convergence rate by using Mather measures. Curiously, the formula connects the convergence rate with an \textbf{ergodic optimization} problem.\,\,For $\varphi\in\mathcal{B}^{c,-}$, we define
\begin{equation}\label{def-rate}
R(c,\varphi):=-\limsup_{t \to \infty} \frac{\ln \| T_t^{c,-}\varphi - u^c_- \|_\infty}{t},\quad \quad R(c):=\inf_{\varphi\in\mathcal{B}^{c,-} } R(c,\varphi).
\end{equation}

\begin{TheoremX}\label{thm:3}  	
Assume $c>c_0$. Let $u^c_-$ be the unique asymptotically stable solution to \eqref{HJs}. Then
\begin{enumerate}[(C1)]
  \item  $R(c)=\inf_{\mu\in \mathfrak{M}_{u^c_-} } \int_{T^*M \times \R } \lambda d\mu  >0$.

  \item  
 For any $\varphi\in\mathcal{B}^{c,-}$ satisfying $\varphi(x)<u^c_-(x)$ for every $x\in M$,
  \[
  R(c,\varphi)=-\lim_{t \to +\infty} \frac{\ln \| T_t^{c,-}\varphi - u^c_- \|_\infty}{t} = R(c).
  \]

  \item  $\lim_{c\to+\infty}R(c)= \lambda_+:=\max_{x\in M}\lambda(x)$.
\end{enumerate}
\end{TheoremX}

\begin{remark}\label{lim-mather}
Combining the conclusions (C1) and (C3), for any $\varepsilon>0$, there is $c(\varepsilon)>0$ such that for any $c>c(\varepsilon)$ and $\mu\in\mathfrak{M}_{u^c_-}, \int \lambda d\mu>\lambda_+-\varepsilon$. In particular, if $\mu$ is a Dirac-delta measure concentrated at $(x_0, p_0, u_0)$, then $\lambda(x_0)>\lambda_+-\varepsilon$.
\end{remark}

\begin{remark}
Regarding $R(c)$ as a function in $c\in(c_0,+\infty)$, 
one might be interested in the continuity and monotonicity of $R(c)$. The following two examples show that there is no conclusion in general: let $M=\mathbb{S}^1$,
\begin{enumerate}[(a)]

  \item let $h(x,p)=\frac{1}{2}|p|^2+(\cos x-1)$ be the pendulum Hamiltonian and $\lambda(x)=\sin x$. By taking $\psi\equiv0$ in \eqref{mane}, we have $c_0\leqslant0$. From the definition, $c_0\geqslant\inf_{\psi\in C^\infty(M,\R)}\lambda(0)\psi(0)+h(0,d_x\psi(0))=\inf_{\psi\in C^\infty(M,\R)}\frac{1}{2}|d_x\psi(0)|^2\geqslant0$, which shows that $c_0=0$. For $c>0, \mathfrak{M}_{u^c_-}$ is a singleton consisting of the Dirac-delta measure supported on the fixed point $\left(\arccos\left(\frac{1}{c + 1}\right) ,\ 0,\sqrt{c (c + 2)}\right)$, where $\arccos\left(\frac{1}{c + 1}\right)\in[0,\frac{\pi}{2})$. Thus $R(c)=\sin \Big( \arccos \frac{1}{c + 1} \Big)=\sqrt{1-\frac{1}{(c+1)^2}}$ is continuous and increasing. 
  

  \item  let $h(x,p)=\frac{1}{2}|p|^2+V(x)$ and $\lambda(x)=2+\sin x\geqslant 1$, where 
$$
V(x)=1-\lambda(x)u_*(x)-\frac12|Du_*(x)|^2, \quad u_*(x)=\varepsilon\left(\sin y-\frac12\sin 2y\right), \quad\varepsilon=\frac{1}{100}, \quad y=x+\pi/3.
$$
Then $c_0=-\infty$. For $c=1$, the ergodic measures in $\mathfrak{M}_{u^c_-}$ are three Dirac-delta measures supporting on the fixed points 
$$ 
P_0=\bigg(-\frac{\pi}{3},\,0,\,0\bigg),\quad P_1=\bigg(\frac{\pi}{3},\,0\,,\frac{3\sqrt{3}\,\varepsilon}{4}\bigg),\quad P_2=\bigg(\pi,\,0,\,-\frac{3\sqrt{3}\,\varepsilon}{4}\bigg).
$$
The degenerate fixed point $P_0$ disappears for $c<1$, while $P_1, P_2$ continue to exist. Thus $R(1)=2-\frac{\sqrt3}{2}$ and $\lim_{c\to 1^-}R(c)=2$. In particular, $R$ is discontinuous at $c=1$.  Moreover, since Theorem \ref{thm:3} gives $\lim_{c\to+\infty}R(c)=\max_{x\in \mathbb T}\lambda(x)=3$, we conclude that the function $c\mapsto R(c)$ is neither increasing nor decreasing on $(0,+\infty)$.
\end{enumerate}
\end{remark}

\begin{remark}\label{exp-div}
Under the assumptions of Theorem A$-$C, one can show that $T_t^{c,-}$ converges or diverges exponentially for \textbf{generic} continuous initial data. More precisely, for any $\varphi\in C(M,\R)$ satisfying $ \min_{x\in M}\{\varphi(x)-v_+^c(x)\}<0$, one can show that
$$
\liminf_{t\to+\infty }\,\,\frac{\ln\|T^{c,-}_t\varphi-u^c_-\|_{\infty}}{t}\geqslant - \inf_{\mu\in \mathfrak{M}_{v^c_+} } \int_{T^*M \times \R } \lambda d\mu>0.
$$
A similar conclusion holds for $T_t^{c,+}$ if $\lambda(x)$ satisfies $\textbf{(H3)}_-$.
\end{remark}

Our last result concerns the distribution of Mather measures. For its statement, it is necessary to introduce
\begin{definition}
For $c\geqslant c_0$, assume \eqref{HJs} admits a solution. We define the Mather sets
\begin{equation}\label{mather1}
\mathcal{M}(c):=\overline{\bigcup_{\mu\in\mathfrak{M}^c}\text{supp}(\mu)},\quad\quad \mathcal{M}_0(c):=\overline{\bigcup_{\mu\in\mathfrak{M}_0^{\text{erg,c}}}\text{supp}(\mu)}
\end{equation}
where $\mathfrak{M}_{0}^{\text{erg,c}}$ denotes the set of all \textbf{ergodic} measures $\mu\in\mathfrak{M}^c$ such that $\int \lambda \, d\mu =0$. The following theorem suggests that $\mathfrak{M}_0^{\text{erg,c}}$ is of particular importance when we try to understand \eqref{HJs} in the critical case.
\end{definition}

\begin{TheoremX}\label{thm:4} 
Assume \eqref{HJs} admits a solution. Then
\begin{itemize}
	\item[(i)]  If $(+)$ holds,  then $\mathcal{M}_0(c)\neq\emptyset$  if and only if  $c=c_0$.
		\item[(ii)] If $(\pm)$ holds, then $\mathcal{M}_0(c) \neq \emptyset$ implies that  $c=c_0$.
\end{itemize}
\end{TheoremX}

\begin{remark}
	The converse implication in {\rm (ii)} is false in general.  Let $M=\mathbb T, h(x,p)=\frac{1}{2}p^2+\sin x\cdot p$ and $\lambda(x)=\cos x$, then for the Hamilton-Jacobi equation
$$
\frac12|u'|^2+\sin x\,u'+\cos x\,u=0,
	$$
the function \(u\equiv0\) is the unique viscosity solution. Indeed, every viscosity solution \(u\) is Lipschitz continuous by
coercivity and satisfies the equation almost everywhere. Since $\cos x\,u+\sin x\,u'=(\sin x\,u)',$
integration over \(\mathbb T\) yields
\[
0
=
\int_{\mathbb T}
\left[
 \frac12|u'|^2+(\sin x\,u)'
\right]\,dx
=
\frac12\int_{\mathbb T}|u'|^2\,dx.
\]
Thus \(u\) is constant, and substituting it back into the equation
shows that \(u\equiv0\). The only ergodic Mather measures are
$\mu_+=\delta_{(0,0,0)}, \mu_-=\delta_{(\pi,0,0)}.$
Their averaged discounts are $\int\lambda\,d\mu_+=1,\int\lambda\,d\mu_-=-1$. Applying Theorem \ref{thm:3}, we conclude that $c=c_0$ and $\mathfrak{M}_{0}^{\text{erg,c}}$ is empty. 
\end{remark}

As a by-product of Theorem D, we return to the uniqueness  of the solution of some non-negative discount Hamilton-Jacobi equation appeared in the study of vanishing discount problem. Let $c(h):=\displaystyle \inf_{\psi\in C^\infty(M,\R)}\sup_{x\in M} h(x,d_x\psi(x))$ be the classical Ma\~{n}\'{e} critical value of $h(x,p)$, we consider
	\begin{equation}\label{eq:c(h)}
\lambda (x)u + h(x, d_x u)= c(h) \quad \text{for }\quad x \in M  
	\end{equation}
Assume $\lambda(x)\geqslant 0$. The sufficient condition uniqueness of solutions to equation \eqref{eq:c(h)} have been studied in \cite{Z}. In this paper, we complete the picture by the following corollary proved in Appendix \ref{section-Appendix-b}. Let $l:TM\rightarrow\R$ be the Lagrangian associated with $h$, an-$l$ minimizing measure is a Borel probability measure $\nu$ on $TM$  such that $\nu$ is invariant by the Euler-Lagrange flow defined by $l$ and 
$$
\int_{TM} l(x,v)\,d\nu=-c(h).
$$
We have

\begin{CorollaryX}\label{Cor:B2}
Assume $\lambda(x)\geqslant 0$. The followings are  equivalent
\begin{itemize}
	\item [(1)] the equation \eqref{eq:c(h)} admits a globally asymptotically stable solution.
	\item [(2)] $c(h)>c_0$
	\item [(3)] $\int_{TM}\lambda(x)\,d\nu>0$ for every $l$-minimizing measure $\nu$.
	\item [(4)]the equation \eqref{eq:c(h)} admits the unique solution.
\end{itemize}
\end{CorollaryX}

\begin{remark}\label{rmk-main-thm}
In this long remark, we describe how to translate the main theorems into the case when the assumption \textbf{(H3)$_+$} is replaced by \textbf{(H3)$_-$}. As above, \textbf{(H3)$_-$} is divided into two subclasses:
\begin{itemize}
    \item[($-$)] $\lambda(x) \leqslant 0$ for all $x \in M$ and $\displaystyle \min_{x\in M} \lambda(x) < 0$.
    \item[($\pm$)] There exist $x_1, x_2 \in M$ such that $\lambda(x_1) > 0$ and $\lambda(x_2) < 0$.
\end{itemize}
Notice that for the definition of Lyapunov stability and asymptotically stability, one only needs a semiflow defined on $C(M,\R)$. Thus we can define the asymptotic stability of $v_+\in\mathcal{S}^{c,+}$ \textbf{with respect to $T^{c,+}_t$} in the similar manner as in Definition \ref{L-stable/A-stable}. As $T^{c,+}_t$ is defined via the backward solution semigroup $\breve{T}^{c,-}_t$ of \eqref{HJe'} defined by $\breve{H}$, it is not hard to check that
\begin{enumerate}[(i)]
  \item $v_+\in\mathcal{S}^{c,+}$ if and only if $-v_+$ is a stationary point of $\breve{T}^{c,-}_t$ and $u_-\in\mathcal{S}^{c,-}$ if and only if $-u_-$ is a stationary point of $\breve{T}^{c,+}_t$, where $\breve{T}^{c,+}_t$ denotes the forward solution semigroup of \eqref{HJe'}.

  \item $v_+\in\mathcal{S}^{c,+}$ is asymptotically stable with respect to $T^{c,+}_t$ \textbf{if and only if} $-v_+$ is an asymptotically stable solution to \eqref{HJs'}.

  \item $-u^c_-$ is minimal within stationary points of $\breve{T}^{c,+}_t$ and $-v^c_+$ is the maximal solution to \eqref{HJs'}.
\end{enumerate}
As is indicated in Remark \ref{+ to -}, if the discount factor $\lambda(x)$ of $H(x,p,u)$ satisfies \textbf{(H3)$_-$}, then the discount factor $-\lambda(x)$ of $\breve{H}(x,p,u)$ satisfies \textbf{(H3)$_+$}. Thus we can apply Theorem A$-$C to $\breve{T}^{c,-}_t$ to obtain the same theorems for the equations \eqref{HJe'} and \eqref{HJs'} defined by $\breve{H}$. Due to (i)-(iii) above, we can translate the results for $\breve{T}^{c,-}_t$ into results for $T^{c,+}_t$ as follows:

\begin{enumerate}[(A$'$)]
  \item There exists $v_+\in\mathcal{S}^{c,+}$ that is asymptotically stable with respect to $T^{c,+}_t$ if and only if $c>c_0$. Such an asymptotically stable $v_+$ is \textbf{unique} and equals $v^c_+$.

  \item Assume $c>c_0$. Define $\mathcal{B}^{c,+}:= \{\,\,\varphi\in C(M,\R) :\lim_{t\to +\infty} T_t^{c,+}\varphi=v^c_+\,\,\}$. Then
        \begin{enumerate}[(B1')]
          \item If ($-$) holds, then $\mathcal{B}^{c,+}=C(M,\R)$. In this case, $v^c_+$ is \textbf{globally asymptotically stable}.

          \item If $(\pm)$ holds, then $\mathcal{B}^{c,+}=\{\varphi \in C(M,\R): \varphi< u^c_-\}$.
        \end{enumerate}

  \item Assume $c>c_0$. For $\varphi\in\mathcal{B}^{c,+}$, we define
        \begin{equation}\label{def-rate}
        R(c,\varphi):=\limsup_{t \to \infty} \frac{\ln \| T_t^{c,+}\varphi - v^c_+ \|_\infty}{t},\quad \quad R(c):=\sup_{\varphi\in\mathcal{B}^{c,+} } R(c,\varphi).
        \end{equation}
        Then it follows that
        \begin{enumerate}[(C1')]
          \item  $R(c)=\sup_{\mu\in \mathfrak{M}_{v^c_+} } \int \lambda d\mu <0$.
          \item   For any $\varphi\in\mathcal{B}^{c,+}$ satisfying $\varphi(x)>v^c_+(x)$ for every $x\in M$,
       $
       \displaystyle   R(c,\varphi)=\lim_{t \to +\infty} \frac{\ln \| T_t^{c,+}\varphi - v^c_+ \|_\infty}{t} = R(c).
        $
          \item  $\lim_{c\to+\infty} R(c)= \lambda_-:=\min_{x\in M}\lambda(x)$.
        \end{enumerate}
\end{enumerate}
\end{remark}

\subsection{Organization of the contents}
In Section 2, we establish uniform bounds for the action of solution semigroups on subsolutions to prove Proposition \ref{prop:1}$-$\ref{prop:2} in the introduction. In Section 3, we recall some result on the large-time behavior of solution semigroups to study of the existence of an asymptotically stable solution and prove Theorem A-B stated in the introduction. With the aid of the improved results in Appendix \ref{Appendix:D} as well as the test function method, we prove Theorem C in Section 4. The last section is devoted to a classification of ergodic Mather measures according to their averaging discount in both non-critical and critical cases, this leads to a description of the distribution of Mather sets on the $(x,u)$-plane.

In Appendix A, we shall briefly recall the variational framework of contact Hamiltonian systems and Hamilton-Jacobi equations with emphasis on the solution semigroups and their properties. Appendix B and C offer proofs of Corollary \ref{Cor:B2} and Proposition \ref{prop:1} (1). In the final appendix, we improve \cite[Theorem 1.2]{RWY} by showing that the lowest convergence rate of the solution semigroup is attained by any initial datum near the stable solution but strictly below and separated from it.

\section{Proof of preliminary observations}\label{proof:pre-observation}
This section is devoted to the proof of Proposition \ref{prop:1}$-$\ref{prop:2}. Here, we list some constants used in the proofs:
\begin{itemize}
  \item $\Lambda:=\|\lambda(x)\|_\infty$;

  \item $\mathbf{E}_0:=\max_{x\in M}h(x,0_x)=-\min_{(x,v)\in TM}l(x,v)$, where $0_x$ denotes the zero covector in $T^\ast_xM$;

  \item $\mathbf{e}_0:=\min_{(x,p)\in T^*M}h(x,p)=-\max_{x\in M}l(x,0_x)$, where $0_x$ denotes the zero vector in $T_xM$.
\end{itemize}

\subsection{Proof of Proposition \ref{prop:1}}

\textit{Proof of (1)}:  Assume that \eqref{HJs} admits a subsolution $\varphi(x)\in C(M,\R)$. Then it follows from the equation that
\[
h(x,d_x\varphi(x))\leqslant\Lambda\|\varphi\|_\infty+\lambda(x)\varphi(x)+h(x,d_x\varphi(x))\leqslant c+\Lambda\|\varphi\|_\infty.
\]
Combining (H2), $\varphi$ is a subsolution of $|d_x\varphi|_x=c+\Lambda\|\varphi\|_\infty+A(1)$. By \cite[Proposition 1.14]{Ishii_chapter}, $\varphi\in \mathrm{Lip}(M,\R)$.

\vspace{0.5em}
Now we apply \cite[Lemma 2.2]{DFIZ} or \cite[Theorem 8.5]{FM_noncompact} to $G(x,p)=\lambda(x)\varphi(x)+h(x,p)-c$ to get, for any $\varepsilon>0$, some $\psi_\varepsilon\in C^1(M,\R)$ such that $\|\varphi-\psi_\varepsilon\|_\infty\leqslant\varepsilon$ and $\lambda(x)\varphi(x)+h(x,d_x\psi_\varepsilon(x))\leqslant c+\varepsilon$. Therefore, we obtain for any $x\in M$,
\[
\lambda(x)\psi_\varepsilon(x)+h(x,d_x\psi_\varepsilon(x))\leqslant\Lambda\cdot\|\psi_\varepsilon-\varphi\|_\infty+\lambda(x)\varphi(x)
+h(x,d_x\psi_\varepsilon(x))\leqslant c+(\Lambda+1)\varepsilon,
\]
which implies that $c_0\leqslant c+(\Lambda+1)\varepsilon$. Since $\varepsilon>0$ is arbitrary, this completes the proof.

\vspace{1em}
Before proving (2), for $c\in\R$, set $\psi_c\equiv c\in C^1(M,\R)$.

\vspace{0.5em}
\noindent\textit{Proof of (2)}  $\Rightarrow$: Assume $\min_{x\in M}\lambda(x):=\lambda_0>0$. Then by (H2),
\[
c_0\leqslant\sup_{x\in M}\{\lambda(x)\psi_c(x)+h(x,d_x\psi_c(x))\}=\max_{x\in M}\{-\lambda(x)\cdot c+h(x,0_x)\}\leqslant\lambda_0\cdot c+\mathbf{E}_0,
\]
Sending $c$ to $-\infty$, we obtain that $c_0=-\infty$. The same conclusion holds by the same argument if $\max_{x\in M}\lambda(x)<0$.

\vspace{0.5em}
\noindent\textit{Proof of (2)} $\Leftarrow$: Assume $\lambda(x_0)=0$ for some $x_0\in M$. Then
\begin{align*}
c_0&\,\geqslant\inf_{\psi\in C^1(M,\R)}\,\,\{\lambda(x_0)\psi(x_0)+h(x_0,d_x\psi(x_0))\}\geqslant\mathbf{e}_0>-\infty.\\
c_0&\,\leqslant\sup_{x\in M}\,\,\lambda(x)\psi_0(x)+h(x,d_x\psi_0(x))\leqslant\mathbf{E}_0<+\infty.\qed
\end{align*}

Before proving (3), we need to improve the estimate of $T_t^{c,\pm}$ on subsolutions in \cite[Lemma 2.12]{JYZ}. This estimate is also used in the proofs of main theorems. For the definitions of the action functions $h_{x_0,u_0}(x,t), h^{x_0,u_0}(x,t)$ that appear below, see Appendix \ref{Appendix-A}.

\begin{lemma}\label{lem:estimation}
Assume (H1)-(H3). For $\delta>0$ and $(x,t)\in M\times(\delta,+\infty)$,
\begin{itemize}	
  \item[(1)]  if there is $x_1\in M$ with $\lambda(x_1)>0$, then for any subsolution $v$ of \eqref{HJs},
              $$
	          T_{t}^{c,-}v(x)\leqslant  \overline B_{c,\delta}(x):= h^c_{x_1,\frac{c-\mathbf{e}_0}{\lambda(x_1)}}(x,\delta).
              $$
              Moreover, for any $\varphi\in C(M,\R)$, $T_{t}^{c,-}\varphi $ is bounded above uniformly  in time.

  \item[(2)] if there is $x_2\in M$ with $\lambda(x_2) <0$, then for any subsolution $v$ of \eqref{HJs},
	         $$
	         T_{t}^{c,+}v(x)\geqslant  \underline  B_{c,\delta}(x):= h_c^{x_2,\frac{c-\mathbf{e}_0} {\lambda(x_2)} }(x,\delta).
	         $$
             Moreover, for any $\varphi\in C(M,\R)$, $T_{t}^{c,+}\varphi $ is bounded below uniformly  in time.
\end{itemize}
\end{lemma}

\begin{proof}
We shall only prove the case  $\lambda(x_1)>0$, that is, (1); the proof of (2) is analogous. Note that  $\lambda(x_1)>0 $. By the definition of the subsolution, $h(x_1,p )+ \lambda(x_1)v(x_1) \leqslant c$ for any $p\in D^+v(x_1)$; here $D^+v(x_1)$ denotes the sup-gradients of $v$ at $x_1$. Thus
$$
\lambda(x_1)v(x_1)\leqslant c-\min_{p\in T^\ast_{x_1}M}h(x_1,p)\leqslant c-\mathbf{e}_0.
$$
Hence, for any subsolution $v$ of \eqref{HJs}, we have $v(x_1)\leqslant \frac{c-\mathbf{e}_0}{\lambda(x_1)}$. By \eqref{eq:Tt-+ rep} and (2) of Proposition \ref{Minimality}, for any $t>0$ and $x\in M$,
\[
T_t^{c,-}v(x)=\inf_{y\in M}h^c_{y,v(y)}(x,t) \leqslant  h^c_{x_1,v(x_1)}(x,t)\leqslant h^c_{x_1,\frac{c-\mathbf{e}_0}{\lambda(x_1)}}(x,t).
\]
On the other hand, using \eqref{eq:Implicit variational} and choosing $\gamma_1(\tau)\equiv x_1$ with $\tau\in [0,t]$, we obtain
\begin{equation*}
\begin{split}
h^c_{x_1,v(x_1)}(x_1,t)=&\,v(x_1)+\inf_{\substack{\gamma(t)=x_1\\ \gamma(0)=x_1} }\int_0^t \bigg[l(\gamma(\tau),\dot \gamma(\tau))+c-\lambda(\gamma(\tau))h^c_{x_1,v(x_1)}(\gamma(\tau),\tau))\bigg]\ d\tau\\
\leqslant  &\, v(x_1)+\int_0^t \bigg[l(\gamma_1(\tau),\dot \gamma_1(\tau))+c-\lambda(\gamma_1(\tau))h^c_{x_1,v(x_1)}(\gamma_1(\tau),\tau))\bigg]\ d\tau\\
= &\, v(x_1)+\int_0^t \bigg[l(x_1,0)+c-\lambda(x_1)\cdot h^c_{x_1,v(x_1)}(x_1,\tau)\bigg]\ d \tau \\
= &\, v(x_1)+(l(x_1,0)+c)\cdot t-\lambda(x_1)\cdot \int_0^t h^c_{x_1,v(x_1)}(x_1,\tau) \ d \tau.
\end{split}
\end{equation*}
It follows from the above inequality that for any $t>0$,
\begin{equation}\label{eq:pf-lem2.1-1}
h^c_{x_1,v(x_1)}(x_1,t)\leqslant\frac{l(x_1,0)+c}{\lambda(x_1)}+\Big(v(x_1)-\frac{l(x_1,0)+c}{\lambda(x_1)}\Big)\cdot e^{-\lambda(x_1)t}\leqslant \frac{c-\mathbf{e}_0}{\lambda(x_1)}.
\end{equation}
Thus by (2) and (3) of Proposition \ref{Minimality}, for any $x\in M$ and $t,\delta >0$,
\begin{align*}
T_{t+\delta }^{c,-}v(x)\leqslant&\,h^c_{x_1, v(x_1) }(x,t+\delta) =  \inf_{y\in M}h^c_{y,h^c_{x_1,v(x_1)}(y,t)}(x,\delta)
\leqslant  h^c_{x_1,h^c_{x_1,v(x_1)}(x_1,t)}(x,\delta)
\leqslant  h^c_{x_1,\frac{c-\mathbf{e}_0} {\lambda(x_1)} }(x,\delta)=\overline B_{c,\delta}(x).
\end{align*}

\vspace{0.5em}
For the second conclusion, we apply \eqref{eq:pf-lem2.1-1} to the continuous initial datum $\varphi$ to obtain: for any $t>0$,
\begin{align*}
h^c_{x_1,\varphi (x_1)}(x_1,t)	\leqslant &\, \frac{l(x_1,0)+c}{\lambda(x_1)}+\Big(\varphi (x_1)-\frac{l(x_1,0)+c}{\lambda(x_1)}\Big)\cdot e^{-\lambda(x_1)t}\leqslant D(x_1,c):= \max\Big\{ \varphi(x_1), \frac{l(x_1,0)+c}{\lambda(x_1)}\Big\}.
\end{align*}
Then by (2) and (3) of Proposition \ref{Minimality}, for any $x\in M$ and $t,\delta >0$,
\begin{align*}
		T_{t+\delta}^{c,-}\varphi (x)=&\, \inf_{y\in M}h^c_{y,\varphi (y)}(x,t+\delta) \leqslant   h^c_{x_1,\varphi (x_1)}(x,t+\delta)=  \inf_{y\in M}h^c_{y,h^c_{x_1,\varphi (x_1)}(y,t)}(x,\delta) \\
\leqslant &\,  h^c_{x_1,h^c_{x_1,\varphi (x_1)}(x_1,t)}(x,\delta)
\leqslant  h^c_{x_1,   D(x_1,c) }(x,\delta),
\end{align*}
where the right-hand side is independent of $t$.

\end{proof}

\noindent\textit{Proof of (3)}: By the definition of a viscosity solution, it suffices to show that if \eqref{HJs} admits a subsolution, then $\mathcal{S}^{c,-}\neq\emptyset$.

\vspace{0.5em}
By Proposition \ref{Minimality} (4), $\overline B_{c,\delta}$ and $\underline B_{c,\delta}$ are Lipschitz functions on $M$. It then follows from Lemma \ref{lem:estimation} that
\begin{enumerate}[(i)]
  \item if \textbf{(H3)$_+$} holds, then $T_{t}^{-}v(x)$ is uniformly bounded from above,
  \item if \textbf{(H3)$_-$} holds, then $T_{t}^{+}v(x)$ is uniformly bounded from below.
\end{enumerate}
In case (i), $\displaystyle \lim_{t\to +\infty} T_t^-v=u_-$ exists and belongs to $\mathcal{S}^{c,-}$. In case (ii), $\displaystyle \lim_{t\to +\infty} T_t^+v=v_+$ exists and belongs to $\mathcal{S}^{c,+}$. Thus by Proposition \ref{mono}(4), the limit $ \displaystyle \lim_{t\to +\infty} T_t^-v_+=u_-$ exists and belongs to $\mathcal{S}^{c,-}$.

\subsection{Proof of Proposition \ref{prop:new-NW-JDE}}
We shall focus only on the proof of (1); the conclusion in (2) can be obtained from (1) and the discussion in Remark \ref{+ to -}.

\vspace{0.5em}
We begin by proving the existence of a solution $u^c_-$ to \eqref{HJs} that is maximal \textbf{not only in $\mathcal{S}^{c,-}$ but also among all subsolutions} to \eqref{HJs}. By Lemma \ref{lem:estimation} (1), set $D:=\max_{x\in M}|\overline B_{c,\delta}(x)|\geqslant0$. Then, for all subsolutions $v$ to \eqref{HJs} and all $x\in M, v(x)\leqslant D$. Choosing a continuous initial datum $\varphi>D$, we find from Lemma  \ref{lem:estimation} (1) and Proposition \ref{sg-2} (3) that $\Tc \varphi$ is bounded above uniformly in $t$ and
$$
u^{\varphi}_-(x):= \liminf_{t\to +\infty} \Tc \varphi\in \mathcal{S}^{c,-}.
$$
Since $\varphi$ is larger than any subsolution $v$, for every subsolution $v$ we have
\begin{align*}
u^{\varphi}_-(x)=\liminf_{t\to +\infty} \Tc \varphi \geqslant \liminf_{t\to +\infty} \Tc v\geqslant v,
\end{align*}
where the first inequality is due to Proposition \ref{sg-1} (2) and the second follows from \eqref{eq:t-mono} of Proposition \ref{mono}. In particular, $u^{\varphi}_-(x)$ is the maximal solution in $\mathcal{S}^{c,-}$. It is readily seen that the maximal solution is unique and independent of $\varphi$, we denote it by $u^c_-$.

\vspace{0.5em}
Now we prove the conclusion concerning the convergence of $\Tc$ on any continuous initial datum $\varphi\geqslant u^c_-(x)$. On the one hand, Proposition \ref{sg-1} (2) implies that for any $x\in M$,
\begin{equation}\label{eq:2-1}
\Tc\varphi(x)\geqslant \Tc u^c_-(x)=u^c_-(x).
\end{equation}
On the other hand, combining Lemma \ref{lem:estimation} (1) with Proposition \ref{sg-2} (3), we find that $\Tc\varphi(x)$ is bounded above uniformly in $t$. Thus $\limsup_{t\rightarrow+\infty} T^{c,-}_t\varphi(x)$ is well-defined and is a subsolution of \eqref{HJs}. By the maximality of $u^c_-(x)$ among the subsolutions to \eqref{HJs} and \eqref{eq:2-1}, we conclude the proof with the inequalities
\begin{equation*}
u^c_-(x)\geqslant\limsup_{t\rightarrow+\infty}T^{c,-}_t\varphi(x)\geqslant\liminf_{t\rightarrow\infty}\Tc\varphi(x)\geqslant\liminf_{t\rightarrow\infty}u^c_-(x)=u^c_-(x).
\end{equation*}

\subsection{Proof of Proposition  \ref{prop:2}}
We shall focus on the proof of (1), the conclusion in (2) can be obtained from (1) and the discussion in Remark \ref{+ to -}. One can compare this proof with \cite[Lemma 5.3, Lemma 5.4]{NW}. Our proof relies on two lemmas.

\begin{lemma}\label{strict-sub}
Let $\psi\in\mathcal{SS}^c$ and $u_-$ be a solution to \eqref{HJs}. Then for $\delta\in(0,1]$, the convex combination
\begin{equation}\label{conv-comb}
u_\delta=\delta\psi+(1-\delta)u_-
\end{equation}
is a strict subsolution to \eqref{HJs}.
\end{lemma}

\begin{proof}
The proof of Proposition \ref{prop:1} (1) shows that any subsolution to \eqref{HJs} is Lipschitz. Let $\mathrm{D}$ be the set of differentiable points of $u_-$. Then, $\mathrm{D}$, as a subset of $M$, has full Lebesgue measure, and for every $x\in\mathrm{D}$,
\begin{equation}\label{cov:1}
\lambda(x) u_-(x)+h(x,d_x u_-(x))=c.
\end{equation}
Since $\psi\in\mathcal{SS}^c$, there is $\varepsilon>0$ such that for every $x\in M$,
\begin{equation}\label{cov:2}
\lambda (x)\psi(x)+h(x,d_x\psi(x))\leqslant c-\varepsilon.
\end{equation}
Therefore $w_\delta$ is differentiable on $\mathrm{D}$ with $d_xu_\delta=\delta d_x\psi+(1-\delta) d_x u_-$. Combining \eqref{cov:1}-\eqref{cov:2} and {\bf(H1)}, for each $x\in \mathrm{D}$,
\begin{align*}
&\,\lambda(x)u_\delta+h(x,d_xu_\delta(x))\\
\leqslant &\,\lambda(x)[\delta\psi(x)+(1-\delta)u_-(x)]+\delta h(x,d_x\psi(x))+(1-\delta )h(x,d_x u_-(x))\\
=&\,\delta\big[h (x,d_x\psi(x))+ \lambda(x)\psi(x)\big]+(1-\delta)\big[h(x,d_x u_-(x))+\lambda(x)u_-(x)\big]\\
\leqslant &\,\delta (c-\varepsilon)+(1-\delta)c=c-\delta\varepsilon.
\end{align*}
We invoke \cite[Theorem 8.3.1]{Fathi_book} to complete the proof.
\end{proof}

To describe the second lemma, we recall from the proof of Proposition \ref{prop:1} (3) that, for any subsolution $v$ of \eqref{HJs}, $T^{c,-}_tv$ always converges as $t$ goes to $+\infty$ since it is increasing and uniformly bounded above in $t$.
\begin{lemma}\label{u-max}
Let $\psi\in\mathcal{SS}^c$ and $u^\psi_-=\lim_{t\rightarrow+\infty}T^{c,-}_t\psi$. Then $\{u_-\in\mathcal{S}^{c,-}:u_-\geqslant\psi\}=\{u^\psi_-\}$.
\end{lemma}

\begin{proof}
For each $u_-\in\cS$ satisfying $u_-\geqslant\psi$,
$$
u_-= \lim_{t\to +\infty} \Tc u_-\geqslant  \lim_{t\to +\infty}\Tc \psi=u^\psi_-.
$$
We argue by contradiction. Assume $u^\psi_-\neq u_-$. Then
\begin{equation}\label{eq:solu-notexist-contradiction}
\max_{x\in M}\{u_-(x)- u^\psi_-(x)  \} >0.
\end{equation}
For $\delta\in [0,1]$, we set $u_\delta=\delta\psi+(1-\delta)u_-$. By Lemma \ref{strict-sub}, for $\delta>0, u_\delta$ is a strict subsolution to \eqref{HJs}, and we apply Proposition \ref{mono} to conclude that for $\delta\in (0,1]$ and $t>0$,
\begin{equation}\label{eq:geq-unique-p2}
\quad \Tc u_{\delta}(x)  > u_{\delta}(x), \quad \forall x \in M.
\end{equation}
Now we define $f:[0,1]\to\R$ by $f(\delta):= \min_{x\in M}\{u^\psi_-(x)-u_{\delta}(x)\}$. The function $f$ is continuous since $u_\delta$ is continuous in $\delta$. By Proposition \ref{mono}, $u^\psi_-\geqslant T^{c,-}_1\psi>\psi$; together with \eqref{eq:solu-notexist-contradiction}, this gives
$$
f(0)=- \max_{x\in M}\{u_-(x)-u^\psi_-(x)\} <0 \quad \text{and} \quad f(1)= \min_{x\in M}\{u^\psi_-(x)-\psi(x)\} >0
$$
Thus there exists $\delta_0\in (0,1)$ such that $f(\delta_0)=0$, i.e.
\begin{equation}\label{eq:delta0}
\min_{x\in M}\{u^\psi_-(x)-u_{\delta_0}(x)\}=f(\delta_0)=0.
\end{equation}
It follows from \eqref{eq:delta0} that $u^\psi_-\geqslant u_{\delta_0}$ and
\begin{equation}\label{eq:geq-unique-p1}
\Tc u^\psi_-(x) \geqslant\Tc u_{\delta_0}(x)  ,\quad \forall x \in M.
\end{equation}
Again, by \eqref{eq:delta0} and the compactness of $M, \arg \min_{x\in M} \{u^\psi_-(x)-u_{\delta_0}(x)\}\neq \emptyset$. Then there is $x_0\in M$ such that $u^\psi_-(x_0)=u_{\delta_0}(x_0)$. Since $u_-\in\cS$, we use \eqref{eq:geq-unique-p1} to obtain
\begin{align*}
u_{\delta_0}(x_0)=u^\psi_-(x_0) =\Tc u^\psi_-(x_0)\geqslant\Tc u_{\delta_0}(x_0).
\end{align*}
This contradicts \eqref{eq:geq-unique-p2}.
\end{proof}

\vspace{0.5em}
\noindent\textit{Proof of (1)}: Given $\psi\in\cSS$, it is known that $\lim_{t\rightarrow+\infty}\Tc\psi$ exists and equals some $u_{-}^\psi\in\cS$. Now Lemma \ref{u-max} shows that
\begin{enumerate}[(C)]
  \item For \textbf{any subsolution} $v$ to \eqref{HJs} with $v\geqslant u_-^\psi,\,\,\,\lim_{t\rightarrow+\infty}\Tc v=u_-^\psi$.
\end{enumerate}
Now by \cite[Section 1.4, Page 128]{Ishii_chapter}, for any $u_-\in\cS$, the function $\max\{u_-,u^\psi_-\}$ is a subsolution to \eqref{HJs}. Thus (C) and Proposition \ref{mono} imply that
\[
u^\psi_-=\lim_{t\rightarrow+\infty}\Tc\max\{u_-,u^\psi_-\}\geqslant\max\{u_-,u^\psi_-\}\geqslant u_-.
\]
This implies that $u_-^\psi:=u^c_-$ is maximal in $\cS$, and is independent of the choice of $\psi$.

\section{Proof of Theorem \ref{thm:1} and \ref{thm:2}}
In this section, we shall prove Theorem A and B. For later descriptions, we divide the assumption \textbf{(H3)} into the following trichotomy:
\begin{itemize}
  \item[($+$)] $\lambda(x) \geqslant 0$ for all $x \in M$ and $\displaystyle \max_{x\in M} \lambda(x)>0$,
  \item[($-$)] $\lambda(x) \leqslant 0$ for all $x \in M$ and $\displaystyle \min_{x\in M} \lambda(x)<0$,
  \item[($\pm$)] There exist $x_1, x_2 \in M$ such that $\lambda(x_1) > 0$ and $\lambda(x_2) < 0$,
\end{itemize}
where ($+$) and ($\pm$) combine to form the assumption \textbf{(H3)$_+$}. In the first part, we show that \textbf{(H3)$_+$} and $c>c_0$ are necessary conditions for the existence of an asymptotically stable solution to \eqref{HJs}. Since \textbf{(H3)$_+$} decomposes into ($+$) and ($\pm$), to complete the proof, we analyze the existence and uniqueness of the asymptotically stable solution and discuss its basin of attraction separately in these two cases .

\subsection{Necessary conditions for the existence of an asymptotically stable solution}
As a first step, we prove the conditions that $\lambda(x)$ satisfies \textbf{(H3)$_+$} and that $c>c_0$ are necessary conditions for an asymptotically stable solution to the discounted Hamilton-Jacobi equation to exist. Before going into the proof, we need
\begin{proposition}\cite[Proposition 14]{WY2021}\label{thG}
The following statements are equivalent:
\begin{itemize}
  \item [(1)] \eqref{HJs} admits a solution.
  \item [(2)] There exist $\varphi, \varphi'\in C(M,\R)$ and $t, t'>0$ such that $T^{c,-}_{t_1}\varphi\geqslant\varphi$ and $T^{c,-}_{t_2}\varphi'\leqslant\varphi'$.
\end{itemize}
\end{proposition}

\begin{proof}
(1)$\,\,\Rightarrow\,\,$(2) is obvious: if \eqref{HJs} admits a solution $u_-$, then (2) automatically holds by choosing $\varphi=\varphi'=u_-$ and any $t=t'>0$.

\vspace{0.5em}					
\noindent(2)$\,\,\Rightarrow\,\,$(1): Notice that $T^{c,-}_{t_1}\varphi\geqslant\varphi$ if and only if $T^{c,+}_{t_1}\varphi\leqslant\varphi$. Thus, as sequences of continuous functions on $M$,
\begin{itemize}
  \item $\{T^{c,-}_{nt_1}\varphi\}_{n\geqslant1}$ is increasing in $n$,
  \item $\{T^{c,+}_{nt_1}\varphi\}_{n\geqslant1}$ is decreasing in $n$.
\end{itemize}
We apply Proposition \ref{sg-2} (4) to obtain three cases:

\medskip
\noindent {\bf Case 1}:	$\lim_{n\to+\infty}T^{c,-}_{nt_1}\varphi(x)=\varphi^-_\infty(x)$ uniformly for $x\in M$ and $\varphi^-_\infty$ is Lipschitz on $M$. Then for any $s\in[0,t_1]$,
\[
\lim_{n\to\infty}T^{c,-}_{nt_1+s}\varphi(x)=T^{c,-}_s\varphi^-_\infty(x).
\]
Hence, $T^{c,-}_{t}\varphi(x)$ is bounded on $M\times[0,+\infty)$, and $u_-(x):=\liminf_{t\to+\infty}T^{c,-}_{t}\varphi(x)$ is a solution to \eqref{HJs}.

\medskip				
\noindent {\bf Case 2}:	$\lim_{n\to+\infty}T^{c,+}_{nt_1}\varphi(x)=\varphi^+_\infty(x)$ uniformly for $x\in M$ and $\varphi^+_\infty$ is Lipschitz on $M$. Then for any $s\in[0,t_1]$,
\[
\lim_{n\to\infty}T^{c,+}_{nt_1+s}\varphi(x)=T^{c,+}_s\varphi^+_\infty(x).
\]
Hence, $T^{c,+}_{t}\varphi(x)$ is bounded on $M\times[0,+\infty)$ and the function $u_+(x):=\limsup_{t\to+\infty}T^{c,+}_{t}\varphi(x)$ belongs to $\mathcal{S}^{c,+}$. Then, applying Proposition \ref{mono}\,(4), we conclude that $u'_-(x):=\lim_{t\to +\infty}  \Tc u_+(x)$ is a solution to \eqref{HJs}.
				
\medskip				
\noindent {\bf Case 3}:	Excluding Cases 1 and 2, we have
\begin{equation}\label{eq:appendixF-case3}
\lim_{n\to+\infty}T^{c,-}_{nt_1}\varphi (x)=+\infty,\quad \lim_{n\to+\infty}T^{c,+}_{nt_1}\varphi (x)=-\infty \quad  \text{uniformly for } x\in M .
\end{equation}
Set $K:= \max_{s\in[0,t_2],x\in M} T^{c,-}_{s}\varphi'(x)$. Then for any $t>0$, there is $m\in\mathbb{N}$ such that $t=mt_2+\tau$ with $\tau\in[0,t_2)$ and
\begin{equation}\label{eq:3-1}
T^{c,-}_t\varphi'=T^{c,-}_\tau\,[T_{mt_2}^{c,-}\varphi']\leqslant T_\tau^{c,-}\varphi'\leqslant K.
\end{equation}
The equations \eqref{eq:appendixF-case3} imply that there exist $n_1,n_2\in \mathbb{N}$ such that
\begin{equation}\label{eq:3-2}
T^{c,+}_{n_2 t_1}\varphi\leqslant T^{c,-}_{t_2}\varphi'\leqslant\varphi'< K+1 \leqslant T^{c,-}_{n_1 t_1}\varphi.
\end{equation}
It follows from \eqref{eq:3-1}-\eqref{eq:3-2} that $K+1\leqslant T^{c,-}_{n_1 t_1}\varphi\leqslant T^-_{(n_1+n_2)t_1}T^{c,+}_{n_2 t_1}\varphi\leqslant T^-_{(n_1+n_2)t_1}\varphi'\leqslant K$, leading to a contradiction.
\end{proof}

The necessary part of Theorem A is implied by the following   	
\begin{proposition}\label{Prop:3.2}
Assume (H1)-(H2). If \eqref{HJs} admits an asymptotically stable solution $\mathbf{u}_-\in\cS$, then \textbf{(H3)}$_+$ holds and $c>c_0$.
\end{proposition}

\begin{proof}
We begin by excluding the case $(-)$. Assume $\lambda(x)$ is non-positive. Then, for any $\delta >0$, $\mathbf{u}_-+\delta$ is a subsolution of \eqref{HJs}. Proposition \ref{mono} implies that for any $\delta>0$ and $t\geqslant0, \Tc(\mathbf{u}_-+\delta)\geqslant\mathbf{u}_-+\delta$; equivalently,
\[
\Tc(\mathbf{u}_-+\delta)-\mathbf{u}_-\geqslant\delta
\]
This contradicts the assumption that $\mathbf{u}_-$ is asymptotically stable.

\vspace{0.5em}
Next, notice that by Proposition \ref{prop:1} (1), $\cS\neq\emptyset$ implies $c\geqslant c_0$. If $c=c_0$, since $\mathbf{u}_-$ is asymptotically stable, there is $\delta>0$ such that $\lim_{t\to+\infty}  T^{c_0,-}_t(\mathbf{u}_--\delta)=\lim_{t\to+\infty} T^{c_0,-}_t(\mathbf{u}_-+\delta)=\mathbf{u}_-$. This implies the existence of $\tau>0$ and two continuous initial data $\varphi=\mathbf{u}_--\delta$ and $\varphi'=\mathbf{u}_-+\delta$ such that
$$
T^{c_0,-}_{\tau}\varphi>\varphi, \quad T^{c_0,-}_{\tau}\varphi'<\varphi'.
$$
By the continuity of $\Tc$ with respect to $c$, there is $\varepsilon>0$ such that
$$
T^{c_0-\varepsilon,-}_{\tau}\varphi>\varphi, \quad T^{c_0-\varepsilon,-}_{\tau}\varphi'<\varphi'.
$$
Hence Proposition \ref{thG} implies that \eqref{HJs} with $c=c_0-\varepsilon$ admits at least one solution. This contradicts Proposition \ref{prop:1} (1).
\end{proof}

\subsection{The asymptotically stable solution and its basin in the case ($+$)}
For this and the next part, we always assume (H1)-(H2). This part focuses on the existence and uniqueness of the asymptotically stable solution to \eqref{HJs} in the case $(+)$. Here, $u^c_-\in\cS$ denotes the solution to \eqref{HJs} found in Proposition \ref{prop:2}. Notice that for any $\psi\in\cSS$,
\begin{equation}\label{eq:3-3}
u^c_-=\lim_{t\to +\infty}\Tc\psi>\psi\in\cS.
\end{equation}
To proceed, we need two lemmas.
\begin{lemma}\label{lem:+global}
Assume $c>c_0$ and $(+)$ holds. Then $\lim_{t\to +\infty} T_t^{c,+}\psi=-\infty$ for any $\psi\in\cSS$.
\end{lemma}

\begin{proof}
Since $\psi\in\cSS, T_t^{c,+}\psi$ is strictly decreasing in $t$. We argue by contradiction. Assume the limit
\begin{equation}\label{eq:3-4}
u_+:=\lim_{t\to +\infty} T_t^{c,+}\psi<\psi
\end{equation}
exists; then $u_+\in\mathcal{S}^{c,+}$. For any $\delta\in(0,\frac{1}{2} \min_{x\in M}\{u^c_--u_+\})$ and $t\geqslant0$,
\begin{equation}\label{eq:3-5}
\|\Tc(u_++\delta)-\Tc u_+ \|_\infty \leqslant \|u_++\delta-u_+\|_\infty=\delta,
\end{equation}
where the first inequality is due to $(+)$. By \eqref{eq:3-4}, there is $\tau>0$ such that $T_\tau^{c,+}\psi<u_++\delta$; then for $t\geqslant\tau$,
$$
T^{c,-}_{t-\tau}\psi\leqslant T^{c,-}_{t}[T^{c,+}_\tau\psi]\leqslant T^{c,-}_{t}[u_++\delta]\leqslant T^{c,-}_{t}u^c_-=u^c_-.
$$
As $t$ tends to $+\infty$, we combine the above inequality with \eqref{eq:3-3} to conclude that
\begin{equation}\label{eq:3-6}
\lim_{t\to +\infty} T_t^{c,-}(u_++\delta)=u^c_-.
\end{equation}
This contradicts \eqref{eq:3-5} since by Proposition \ref{mono}, $\{x\in M: T_t^{c,-}u_+(x)=u_+(x) \}\neq \emptyset$, and hence
$$
\lim_{t\rightarrow+\infty}\|T_t^{c,-}(u_++\delta)-T_t^{c,-}u_+ \|_\infty \geqslant \frac{1}{2} \min_{x\in M}\{ u^c_--u_+\}>\delta.
$$
\end{proof}

Our aim is to show
\begin{proposition}\label{necessary +}
Assume ($+$) and $c>c_0$. Then $u^c_-$ is globally asymptotically stable with respect to $\Tc$. In particular, $u^c_-$ is the unique asymptotically stable solution to \eqref{HJs} and $\mathcal{B}^{c,-}=C(M,\R)$.
\end{proposition}

\begin{proof}
For any $\varphi\in C(M,\R)$, Lemma \ref{lem:+global} gives $t_0>0$ such that  $T_{t_0}^{c,+}\psi\leqslant\varphi$. Then for any $t\geqslant t_0$,
$$
T_t^{c,-}\varphi \geqslant T_t^{c,-}[T_{t_0}^{c,+}\psi]\geqslant T_{t-t_0}^{c,-}\psi.
$$
As $t$ tends to $+\infty$, it follows that $\liminf_{t\to +\infty} T_t^{c,-}\varphi\geqslant u^c_-$.
By an argument similar to that in \cite[Theorem 2]{NW}, we have that
$$
\lim_{t\to +\infty}\Tc\varphi =u^c_-.
$$	
This fact shows that $u^c_-$ is globally asymptotically stable and thus completes the proof.
\end{proof}

\subsection{The asymptotically stable solution and its basin in the case ($\pm$)}
In this part, we focus on the description of the asymptotically stable solution to \eqref{HJs} and its basin when $\lambda(x)$ changes sign on $M$. Notice that $(\pm)$ is a subcase of both \textbf{(H3)$_+$} and \textbf{(H3)$_-$}. By Proposition \ref{prop:2}, there are $u^c_-\in\cS$ and $v^c_+\in\mathcal{S}^{c,+}$ that attract the respective orbits of $\Tc$ and $T^{c,+}_t$ starting from $\cSS$. We recall
\begin{lemma}\label{+-}\cite[Theorem 2, Theorem 3]{NW}
Assume $(\pm)$. Then
\begin{enumerate}[(1)]
  \item for $c\geqslant c_0, \lim_{t\rightarrow+\infty}\Tc\varphi=-\infty$\,\, for any continuous function $\varphi$ with $\varphi(x_0)<v^c_+(x_0)$ for some $x_0\in M$;
  \item for $c>c_0, \lim_{t\rightarrow+\infty}\Tc\varphi=u^c_-$\,\, for any continuous function $\varphi>v^c_+$.
\end{enumerate}
\end{lemma}

The above lemma yields
\begin{proposition}\label{+-as}
Assume $c>c_0$ and $(\pm)$. Then
\begin{enumerate}[(a)]
  \item $u^c_-$ is the unique asymptotically stable solution to \eqref{HJs},
  \item $\Bc=\{\varphi\in C(M,\R):\varphi>v^c_+\}$.
\end{enumerate}
\end{proposition}

\begin{proof}
(a): By Proposition \ref{prop:2}, for any $\psi\in\cSS$,
\begin{equation}\label{order}
u^c_->\psi>v^c_+.
\end{equation}
Thus (1) and (3) above imply that $u^c_-$ is asymptotically stable. Again by (3), for any $u_-\in\cS$ \textbf{different from $u^c_-$}, there is $x_0\in M$ such that $u_-(x_0)\leqslant v^c_+(x_0)$. Thus by (2), for any $\delta>0$,
\[
\lim_{t\rightarrow+\infty}\Tc(u_--\delta)=-\infty.
\]
This shows that $u_-$ is not asymptotically stable.

\medskip
\noindent(b): Notice that Lemma \ref{+-} implies that
$$
\{\varphi \in C(M,\R): \varphi >v^c_+\}\subset\mathcal{B}^{c,-}\subset\{\varphi \in C(M,\R):\varphi\geqslant v^c_+\}.
$$
Define $\partial\Bc=\{\varphi\in C(M,\R):\varphi\geqslant v^c_+,\,\,\exists\, x_0\in M,\,\,\varphi(x_0) =v^c_+(x_0)\}$. To finish the proof, we need to show that
\begin{equation}\label{eq:3-7}
\partial\Bc\cap\Bc=\emptyset.
\end{equation}
To prove \eqref{eq:3-7}, we shall show that any $\varphi\in\partial\Bc$ does not belong to $\Bc$. We fix $x_0\in M$ such that $\varphi(x_0)=v^c_+(x_0)$. By \eqref{1jet-inv}, we choose an orbit $\{(x(t),p(t),u(t)):t\in[0,+\infty)\}\subset\mathrm{J}^1_{u_+}$ of \eqref{ch} satisfying $x(0)=x_0$. Then $u(t)=v^c_+(x(t))$ for $t\geqslant0$. It follows that
\begin{equation}\label{eq:3-8}
\Tc\varphi(x(t))=\inf_{y\in M} h^c_{y,\varphi(y)}(x(t),t)\leqslant\,\,h^c_{x_0,\varphi(x_0)}(x(t),t)
=h^c_{x(0), v^c_+(x(0))}(x(t),t)\leqslant u(t)=v^c_+(x(t)),
\end{equation}
where the second inequality follows from Proposition \ref{Minimality} (1). Employing the inequality \eqref{eq:3-8}, we have for $t\geqslant0$,
$$
\|u^c_--\Tc\varphi\|_\infty \geqslant u^c_-(x(t))-\Tc\varphi(x(t))=u^c_-(x(t))-v^c_+(x(t))\geqslant \min_{x\in M}\{ u^c_-(x)-v^c_+(x)\}>0,
$$
where the last inequality follows from \eqref{order}. Thus $\varphi\notin\Bc$, which completes the proof.
\end{proof}

\section{Proof of Theorem \ref{thm:3}}
In this section, for $c>c_0$, we show that the solution semigroup $T^{c,-}_t$ converges exponentially within the basin $\Bc$ of the stable solution $u^c_-$ with an explicit formula for the lowest convergence rate, and this rate is attained by some initial data near $u^c_-$. With the help of an estimate of the action functions, we derive the asymptotic behavior of the convergence rate as $c$ goes to infinity.

\subsection{Positivity of the averaging discount on the stable solution}
In this part, we show that for $c>c_0$, the integral of $\lambda(x)$ with respect to any measure in $\mathfrak{M}_{u^c_-}$ is bounded below by a positive constant. To achieve the goal, we construct a test function $G(x,u,p)$ on the phase space so that $\lambda(x)+\mathfrak{L}_{H,c}G$ is positive everywhere on $\mathrm{J}^1_{u^c_-}$, where $\mathfrak{L}_{H,c}$ denotes the Lie derivative associated to $\Phi^t_{H,c}$. To proceed, we need an auxiliary
\begin{lemma}\label{lem:appendix-1}
Assume $h$ satisfies (H1)-(H2). Then for any $\psi\in C^\infty(M,\R)$,
\begin{equation}\label{eq:4-1}
\sup_{p\in T_x^\ast M}\left\{\lambda(x)\psi(x)+\left\langle d_x\psi,\frac{\partial h}{\partial p}(x,p)\right\rangle_x+h(x,p)-\left\langle p,\frac{\partial h}{\partial p}(x,p)\right\rangle_x\right\}=\lambda(x)\psi(x)+h(x,d_x\psi),\quad x\in M.
\end{equation}
In particular, $c_0=\inf_{\psi\in C^\infty(M,\R) }\sup_{(x,p)\in T^*M}\Big\{ \lambda(x)\psi(x)+\langle d_x\psi,\frac{\partial h}{\partial p}(x,p)\rangle_x+h(x,p)-\langle p,\frac{\partial h}{\partial p}(x,p)\rangle_x\Big\}$.
\end{lemma}

\begin{proof}
Set $\mathrm{v}(x,p)=\frac{\partial h}{\partial p}(x,p)$; then
\begin{equation}\label{leg-tran}
h(x,p)=\langle p,\mathrm{v}(x,p)\rangle_x-l(x,\mathrm{v}(x,p))=\sup_{v\in T_xM}\{\langle p,v\rangle_x-l(x,v)\}.
\end{equation}
Therefore, for a fixed $\psi\in C^\infty(M,\R)$, we compute
\begin{align*}
&\sup_{p\in T^\ast_xM}\{\lambda(x)\psi(x)+\langle d_x\psi,\mathrm{v}(x,p)\rangle_x+[h(x,p)-\langle p,\mathrm{v}(x,p)\rangle_x]\}\\
=&\,\sup_{p\in T^\ast_xM}\{\lambda(x)\psi(x)+\langle d_x\psi,\mathrm{v}(x,p)\rangle_x-l(x,\mathrm{v}(x,p))\}\\
=&\,\lambda(x)\psi(x)+\sup_{v\in T_xM}\{\langle d_x\psi,v\rangle_x-l(x,v)\}=\lambda(x)\psi(x)+h(x,d_x\psi),
\end{align*}
where the first equality follows from \eqref{leg-tran}; since $h$ is Tonelli, the second equality follows from the fact that, when restricting to $T^\ast_xM$, the Legendre transform $p\mapsto\mathrm{v}(x,p)$ is a global diffeomorphism. By taking the supremum over \( x \) on both sides, we arrive at \eqref{eq:4-1}. Taking the infimum over $\psi\in C^\infty(M,\R)$ on both sides completes the proof.
\end{proof}

\begin{lemma}\label{lem:pf-1.1-2}
For any $c>c'>c_0$, there are $\psi\in\mathcal{SS}^{\,c'}$ and a $C^\infty$ function $G$ defined on
$$
\mathcal{D}_\psi:=\{(x,p,u)\in T^\ast M\times\R:u>\psi(x)\}
$$
such that for all $(x,p,u)\in\mathcal{D}_\psi$,
\begin{equation}\label{eq:4-7}
\lambda(x) + \mathfrak{L}_{H,c}G(x,p,u)\geqslant\frac{c-c'}{u-\psi(x)}>0,
\end{equation}
where $\mathfrak{L}_{H,c}$ denotes the Lie derivative with respect to the phase flow $\Phi^t_{H,c}$.
\end{lemma}

\begin{proof}
By definition of the critical value, $\mathcal{SS}^{\,c'}$ is non-empty. We choose $\psi\in\mathcal{SS}^{\,c'}$ and apply \eqref{eq:4-1} to obtain
\begin{equation}\label{eq:4-2}
\sup_{(x,p)\in T^*M} \Big\{ \lambda(x)\psi(x)+\langle d_x\psi,\frac{\partial h}{\partial p}(x,p)\rangle_x+[h(x,p)-\langle p,\frac{\partial h}{\partial p}(x,p)\rangle_x] \Big\}=\lambda(x)\psi(x)+h(x,d_x\psi)<c'.
\end{equation}
We construct $G:\mathcal{D}_\psi\rightarrow\R$ as follows:
\begin{equation}\label{test-fun}
G(x,p,u):= \ln (u-\psi(x))
\end{equation}
Using \eqref{ch}, we compute
\begin{align*}
\Bigl( \lambda(x) + \mathfrak{L}_{H,c}G\Bigr)(x, p, u) =&\, \lambda (x) + \left. \frac{d}{dt} \right|_{t=0} G\bigl( \Phi_{H,c}^t(x, p, u) \bigr)=\lambda (x) + \frac{\partial{G}}{\partial{u}}\cdot \dot{u}+ \frac{\partial{G}}{\partial{x}}\cdot \dot{x}\\
=&\, \lambda (x)+ \frac{1}{u-\psi(x)} \cdot \Bigl( \langle p,\frac{\partial h}{\partial p}(x,p)\rangle_x-h(x,p)-\lambda(x)u+c-\langle d_x\psi,\frac{\partial h}{\partial p}(x,p)\rangle_x\Bigr) \\
=&\, \frac{1}{u-\psi(x)} \cdot \Bigl( c-\lambda(x)\psi(x)-\langle d_x\psi,\frac{\partial h}{\partial p}(x,p)\rangle_x-[h(x,p)-\langle p,\frac{\partial h}{\partial p}(x,p)\rangle_x]  \Bigr)\\
>&\,\frac{1}{u-\psi(x)}\cdot\Bigl( c-\lambda(x)\psi(x)-h(x,d_x\psi) \Bigr)\geqslant\frac{c-c'}{u-\psi(x)},
\end{align*}
where the last inequality follows from \eqref{eq:4-2}.
\end{proof}

The conclusion of (C1) is equivalent to
\begin{proposition}
For any $c>c_0$, we have $\displaystyle\inf_{ \mu\in\mathfrak{M}_{u^c_-}}\int_{T^*M\times\R} \lambda\ d\mu>0$.
\end{proposition}

\begin{proof}
Choose any finite number $c'$ such that $c_0<c'<c$; this is possible also when $c_0=-\infty$. We apply Lemma \ref{lem:pf-1.1-2} to obtain $\psi\in\mathcal{SS}^{c'}$ and $G:\mathcal{D}_\psi\rightarrow\R$ with

\begin{equation}\label{eq:4-3}
\Bigl( \lambda(x)+\mathfrak{L}_{H,c}G \Bigr)(x, p, u) \geqslant  \frac{c-c'}{u-\psi(x)}, \quad \forall(x,p,u)\in\mathcal{D}_\psi.
\end{equation}

With the aid of Proposition \ref{prop:2}, we have
\begin{equation}\label{eq:4-4}
u^c_-= \lim_{t\to +\infty} T_t^{c,-}\psi>\psi.
\end{equation}
Since \(\mathrm{J}^1_{u^c_-}\) is compact subset of \(\mathcal{D}_\psi\) and
\[
\mathfrak{L}_{H,c}G(x, p, u) = \frac{\partial G}{\partial x} \frac{\partial H}{\partial p} - \frac{\partial G}{\partial p} \left( \frac{\partial H}{\partial x} + \frac{\partial H}{\partial u} p \right) + \frac{\partial G}{\partial u} \left( \frac{\partial H}{\partial p} p - H+c \right)
\]
is smooth on \(\mathcal{D}_\psi\), for $t$ close enough to $0$, the function $\frac{1}{t}[G\circ\Phi^t_{H,c}-G]$ is bounded around $\mathrm{J}^1_{u^c_-}$. Thus for any \(\mu \in \mathfrak{M}_{u^c_-}\), we have
\begin{align*}
&\int_{T^*M \times \mathbb{R}} \mathfrak{L}_{H,c}G\,\, d\mu=\int_{T^*M \times \mathbb{R}} \lim_{t\rightarrow0}\frac{G\circ\Phi^t_{H,c}-G}{t}\,\, d\mu\\
=&\,\lim_{t\rightarrow0}\frac{1}{t}\int_{T^*M \times \mathbb{R}}\bigg[G\circ\Phi^t_{H,c}-G\bigg]\,\, d\mu\\
=&\,\lim_{t\rightarrow0}\frac{1}{t}\bigg[\int_{T^*M \times \mathbb{R}}G\circ\Phi^t_{H,c}\,\, d\mu -\int_{T^*M \times \mathbb{R}}G\,\, d\mu\bigg]\\
=&\,\lim_{t\rightarrow0}\frac{1}{t}\bigg[\int_{T^*M \times \mathbb{R}}G\,\, d(\Phi^t_{H,c})_\sharp\mu -\int_{T^*M \times \mathbb{R}}G\,\, d\mu\bigg]=0,
\end{align*}
where the second equality follows from the fact that supp$(\mu)\subset\mathrm{J}^1_{u^c_-}$ as well as the dominated convergence theorem; while the last equality follows from the \(\Phi^t_{H,c}\)-invariance of \(\mu\). To complete the proof, we observe that for any $\mu \in \mathfrak{M}_{u^c_-}$,

\begin{equation*}
\int_{T^*M \times \mathbb{R}} \lambda(x)\,d\mu =\int_{T^*M \times \mathbb{R}}\big[\lambda(x) + \mathfrak{L}_{H,c}G\big]\, d\mu\geqslant\int_{T^*M\times\R}\frac{c-c'}{u-\psi(x)}\,d\mu\geqslant \frac{c-c'}{\max_{x\in M} \{ u_-^c(x)-\psi(x)\}}>0.
\end{equation*}
where the second inequality also uses the fact that $\mu$ is supported on $\mathrm{J}^1_{u^c_-}$.
\end{proof}

\subsection{Exponential convergence of the solution semigroup}
In this part, we prove (C2). Our proof relies on Proposition \ref{prop:JYZ-new}, an improved statement of \cite[Theorem 1.2]{RWY}, which is of independent interest and applies to any contact Hamiltonian $H(x,p,u):T^*M\times\R\to\R$ satisfying the assumptions given at the beginning of Appendix A.2. For the proof of Proposition \ref{prop:JYZ-new}, see Appendix  \ref{Appendix:D}. To start with, we summarize the conclusions in Appendix  \ref{Appendix:D} in the following

\begin{proposition}\label{con-rate}
Assume $c>c_0$ and $\inf_{\mu\in\mathfrak{M}_{u^c_-}}\int_{T^*M\times\R}\lambda\ d\mu>0$. Then there exists $\Delta>0$ such that
\begin{enumerate}[(1)]
  \item for any $\varphi\in C(M,\R)$ with $\|\varphi-u^c_-\|_\infty<\Delta$,
        \begin{equation}\label{exp-cov1}
        \limsup_{t\rightarrow\infty}\frac{\ln\|T^{c,-}_t\varphi-u^c_-\|_\infty}{t}\leqslant-\inf_{ \mu\in\mathfrak{M}_{u^c_-}}\int_{T^*M\times\R} \lambda\ d\mu.
        \end{equation}

  \item for every $\delta\in(0,\Delta)$,
        \begin{equation}\label{exp-cov2}
        \lim_{t\to +\infty}\dfrac{\ln\|T^{c,-}_t(u_-^c-\delta)-u^c_-\|_{\infty}}{t}=-\inf_{ \mu\in\mathfrak{M}_{u^c_-}}\int_{T^*M\times\R} \lambda\ d\mu.
        \end{equation}
\end{enumerate}
\end{proposition}

\begin{proposition}\label{lem:RWY-main}
For $c>c_0$,
\begin{equation}\label{eq:alpha(c)}
R(c)=\inf_{ \mu\in\mathfrak{M}_{u^c_-}}\int_{T^*M\times\R} \lambda\ d\mu.
\end{equation}
\end{proposition}

\begin{proof}
Notice that for any $\varphi\in\Bc, \lim_{t\to +\infty } \Tc \varphi= u^c_-$ holds uniformly on $M$. Thus for the $\Delta>0$ appearing in Proposition \ref{con-rate}, there exists $t_0>0$ such that $\| T^{c,-}_{t}\varphi-u^c_-\|_\infty<\Delta$ for every $t\geqslant t_0$. Applying Proposition \ref{con-rate} (1), we have
$$
\limsup_{t\rightarrow\infty}\frac{\ln\|T^{c,-}_{t+t_0} \varphi-u^c_-\|_\infty}{t} = \limsup_{t\rightarrow\infty}\frac{\ln\|T^{c,-}_t\varphi-u^c_-\|_\infty}{t}\leqslant- \displaystyle\inf_{ \mu\in\mathfrak{M}_{u^c_-}}\int_{T^*M\times\R} \lambda\ d\mu.
$$
By the definition of $R(c)$ and $R(c,\varphi)$,
\begin{equation}\label{eq:pf-lem4.4-1}
 R(c)=\inf_{\varphi\in\mathcal{B}^{c,-}} R(c,\varphi)=\inf_{\varphi\in\mathcal{B}^{c,-}} -\limsup_{t \to \infty} \frac{\ln \|\Tc\varphi - u^c_- \|_\infty}{t} \geqslant \displaystyle\inf_{ \mu\in\mathfrak{M}_{u^c_-}}\int_{T^*M\times\R} \lambda\ d\mu.
\end{equation}
On the other hand, fix $\delta\in(0,\Delta)$ and put $\varphi_0=u_-^c-\delta$. In the sign-changing case, by taking $\delta$ small enough if necessary, we can assume $\varphi_0>v_+^c$. Hence Theorem B shows that $\varphi_0\in\mathcal B^{c,-}$. Then Proposition \ref{con-rate} (2) yields
$$
\lim_{t \to \infty} \frac{\ln \| T_t^{c,-} \varphi_0- u^c_- \|_\infty}{t}= -\displaystyle\inf_{ \mu\in\mathfrak{M}_{u^c_-}}\int_{T^*M\times\R} \lambda\ d\mu.
$$
It follows that
\begin{equation}\label{eq:pf-lem4.4-2}
R(c) \leqslant R(c,\varphi_0)=-\limsup_{t \to \infty} \frac{\ln \|\Tc\varphi_0 - u^c_- \|_\infty}{t}=\displaystyle\inf_{ \mu\in\mathfrak{M}_{u^c_-}}\int_{T^*M\times\R} \lambda\ d\mu.
\end{equation}
Combining \eqref{eq:pf-lem4.4-1} and \eqref{eq:pf-lem4.4-2}, we complete the proof.
\end{proof}

To conclude this part, we notice that the convergence rate only depends on the behavior of the solution semigroup near the stable solution $u^c_-$. Thus one can globalize Proposition \ref{con-rate} (2) as follows.

\begin{proposition}
For any $\varphi\in\mathcal{B}^{c,-}$ satisfying $\varphi<u_-^c$,
\begin{equation}\label{eq:A(c,varphi)= A(c)}
R(c,\varphi)= -\lim_{t \to \infty} \frac{\ln \| \Tc \varphi - u^c_- \|_\infty}{t}=R (c).
\end{equation}
\end{proposition}

\begin{proof}
Let $\varphi\in\Bc$ satisfy $\varphi<u^c_-$. Note that
$$
\varphi \leqslant \varphi_\Delta:= u^c_- -  \min\Big\{ \Delta/2, \min_{x\in M} \{u^c_--\varphi \} \Big\}
$$
where $\Delta$ is the same as the one appearing in Proposition \ref{con-rate} (2). This implies that $\varphi\leqslant\varphi_\Delta< u^c_-$ and $\|\varphi_\Delta - u^c_- \|_\infty\leqslant\Delta$. We use Proposition \ref{mono} to derive that
$$
0<u^c_--\Tc\varphi_\Delta\leqslant u^c_--\Tc\varphi.
$$
We can apply Proposition \ref{con-rate} (2) to $\varphi_\Delta$ to obtain that
\begin{align*}\label{eq:pf-lem4.4-4}
R (c)\leqslant R (c,\varphi) = &\, -\limsup_{t \to \infty} \frac{\ln \|\Tc\varphi - u^c_- \|_\infty}{t}
\leqslant -\liminf_{t \to \infty} \frac{\ln \|\Tc\varphi -u^c_- \|_\infty}{t}\\
\leqslant &\,  -\liminf_{t \to \infty} \frac{\ln \|\Tc\varphi_\Delta - u^c_-\|_\infty}{t}= -\lim_{t \to \infty} \frac{\ln \| T_t^{c,-} \varphi_\Delta - u^c_- \|_\infty}{t}=R (c,\varphi_\Delta)= R (c),
\end{align*}
which finishes the proof.
 \end{proof}

\subsection{Asymptotic convergence rate as $c$ tends to infinity}
In the last part of this section, we study the limit $\lim_{c\rightarrow+\infty}R(c)$. This is done by relating an upper bound for $u^c_-$ to the parameter $c$. Throughout this part, we assume \eqref{HJs} admits a subsolution, so $c\geqslant c_0$, and \textbf{(H3)$_+$} holds. We recall that by Lemma \ref{lem:estimation}, for any subsolution $v$ of \eqref{HJs} and $\delta>0,$,
\begin{equation}\label{eq:4-5}
T_{t}^{c,-}v(x)\leqslant  \overline B_{c,\delta}(x):= h^c_{x_1,\frac{c-\mathbf{e}_0}{\lambda(x_1)}}(x,\delta),\quad (x,t)\in M\times(\delta,+\infty).
\end{equation}
Set $\lambda_+:= \max_{x\in M}\lambda(x)$. Due to the compactness of $M$, there is $x_1\in M$ such that $\lambda(x_1)=\lambda_+$.

\begin{lemma}\label{lem:eq-B-Lip}
For any  \( c_1,  c_2 \in \R \) with $c_1\leqslant c_2$ and $x\in M$, any $t>0$,
\begin{align*}
| \overline B_{c_2,t}(x) -  \overline B_{c_1,t}(x)|\leqslant (c_2-c_1)\Big[\frac{1}{\lambda_+}+ t \Big] e^{\Lambda t},
\end{align*}
where $\Lambda:= \max_{x\in M}| \lambda(x)|$.
\end{lemma}

\begin{proof}
Note that  \( c_1 \leqslant c_2 \). By Proposition \ref{Minimality} (2), for any $x\in M$ and $t>0$
$$
h_{x_1,  \frac{c_1-\mathbf{e}_0} {\lambda_+}}^{c_1}(x, t) \leqslant   h_{x_1,  \frac{c_2-\mathbf{e}_0} {\lambda_+}}^{c_1}(x, t) \leqslant h_{x_1,  \frac{c_2-\mathbf{e}_0} {\lambda_+}}^{c_2}(x, t).
$$
Let \( \gamma_1: [0,t] \to M\) be a minimizer of \( h_{x_1, \frac{c_1-\mathbf{e}_0} {\lambda_+}}^{c_1}(x, t) \). Then  for any \( s \in (0, t] \), we get
\begin{equation}\label{eq:4-6}
h_{x_1, \frac{c_1-\mathbf{e}_0} {\lambda_+}}^{c_1}(\gamma_1(s), s) \leqslant h_{x_1, \frac{c_2-\mathbf{e}_0} {\lambda_+} }^{c_2}(\gamma_1(s), s).
\end{equation}
We employ Proposition \ref{Implicit variational} to obtain
\begin{align*}
	h_{x_1, \frac{c_1-\mathbf{e}_0} {\lambda_+}}^{c_1}(\gamma_1(s), s)=  \frac{c_1-\mathbf{e}_0} {\lambda_+}+ \int_0^s \Big[  L(\gamma_1(\tau),\dot{\gamma}_1( \tau ), h_{x_1, \frac{c_1-\mathbf{e}_0} {\lambda_+}}^{c_1}(\gamma_1(\tau), \tau) )+c_1\Big] d\tau, \\
		h_{x_1, \frac{c_2-\mathbf{e}_0} {\lambda_+}}^{c_2}(\gamma_1(s), s)\leqslant  \frac{c_2-\mathbf{e}_0} {\lambda_+}+ \int_0^s \Big[  L(\gamma_1(\tau),\dot{\gamma}_1( \tau ), h_{x_1, \frac{c_2-\mathbf{e}_0} {\lambda_+}}^{c_2}(\gamma_1(\tau), \tau) )+c_2\Big] d\tau .
\end{align*}
The above relations lead to
\[
\begin{aligned}
& \quad  h_{x_1,  \frac{c_2-\mathbf{e}_0} {\lambda_+}}^{c_2}(\gamma_1(s), s) - h_{x_1,  \frac{c_1-\mathbf{e}_0} {\lambda_+} }^{c_1}(\gamma_1(s), s) \\
 &\leqslant \frac{c_2-c_1}{\lambda_+} + (c_2 - c_1)s + \int_0^s \left[ L(\gamma_1, h_{x_1, \frac{c_2-\mathbf{e}_0} {\lambda_+}}^{c_2}(\gamma_1(\tau), \tau), \dot{\gamma}_1) - L(\gamma_1, h_{x_1, \frac{c_1-\mathbf{e}_0} {\lambda_+}}^{c_1}(\gamma_1(\tau), \tau), \dot{\gamma}_1) \right] d\tau \\
&\leqslant \frac{c_2-c_1}{\lambda_+} + (c_2 - c_1)t + \int_0^s \Lambda\cdot \left( h_{x_1, \frac{c_2-\mathbf{e}_0} {\lambda_+} }^{c_2}(\gamma_1(\tau), \tau) - h_{x_1, \frac{c_1-\mathbf{e}_0} {\lambda_+}}^{c_1}(\gamma_1(\tau), \tau) \right) d\tau.
\end{aligned}
\]
Set $F(s) := h_{x_1, \frac{c_2-\mathbf{e}_0} {\lambda_+}}^{c_2}(\gamma_1(s), s) - h_{x_1, \frac{c_1-\mathbf{e}_0} {\lambda_+}}^{c_1}(\gamma_1(s), s)$. Then \eqref{eq:4-6} implies that \( F(s) \geqslant 0 \) for any \( s \in (0, t] \). Hence,
\[
0\leqslant F(s) \leqslant \frac{c_2-c_1}{\lambda_+} + (c_2 - c_1)t + \int_0^s \Lambda \cdot F(\tau) d\tau, \quad \forall s\in (0,t].
\]
Gronwall's inequality yields
\[
F(s) \leqslant \Big[\frac{c_2-c_1}{\lambda_+}+ (c_2 - c_1) t \Big] e^{\Lambda s}, \quad s \in [0, t].
\]
Thus by definition \eqref{eq:4-5},
\begin{align*}
|\overline B_{c_2,t}(x) -  \overline B_{c_1,t}(x)|=\,  h_{x_1, \frac{c_2-\mathbf{e}_0} {\lambda_+}}^{c_2}(x, t) - h_{x_1,\frac{c_1-\mathbf{e}_0} {\lambda_+}}^{c_1}(x, t) \leqslant \Big[\frac{c_2-c_1}{\lambda_+}+ (c_2 - c_1) t \Big] e^{\Lambda t},
\end{align*}
which is the desired estimate.
\end{proof}

To relate the stable solution $u^c_-$ to the parameter $c$, we use the above estimate to obtain an upper bound for \( \Tc v \) independent of \( t \). As a fixed point of $\Tc, u_-^c$ admits the same upper bound.
\begin{lemma}\label{lem:Step4-bounded}
Fix a finite reference level $b>c_0$. For $c>b$, any subsolution $v$ of \eqref{HJs}, and $t\geqslant\delta>0$,
$$
T_t^{c,-} v(x) \leqslant \bigg(\frac{1}{\lambda_+ }+ \delta\bigg) e^{\Lambda   \delta} \cdot (c-b) +   \overline B_{b,\delta}(x), \quad \forall x\in M.
$$
In particular, the right-hand side of the above inequality also gives an upper bound for $u^c_-$.
\end{lemma}

\begin{proof}
Lemma \ref{lem:eq-B-Lip} and Lemma \ref{lem:estimation} give, for $t>\delta>0$,
\begin{align*}
T_t^{c,-} v (x) \leqslant \overline B_{c,\delta}(x)\leqslant \bigg(\frac{1}{\lambda_+}+\delta\bigg) e^{\Lambda   \delta}\cdot(c-b) +\overline B_{b,\delta}(x), \quad \forall x\in M.
\end{align*}
Since $u_-^c$ is a subsolution of \eqref{HJs}, for any $\delta>0$
$$
u_-^c(x)=T_{2\delta}^{c,-}u_-^c(x) \leqslant  \bigg(\frac{1}{\lambda_+}+\delta\bigg) e^{\Lambda   \delta}\cdot(c-b) +\overline B_{b,\delta}(x), \quad \forall x\in M.
$$
This completes the proof.
\end{proof}

We combining the above Lemma \ref{lem:appendix-1}, Lemma \ref{lem:pf-1.1-2}, Lemma \ref{lem:Step4-bounded} to have
\begin{lemma}\label{asy-c}
For any $\varepsilon>0$, there exists $C(\varepsilon)$ such that for any $c>C$,  there exist a $C^\infty$ function $G(x,p,u)$ such that
$$
\Bigl( \lambda(x) + \mathfrak{L}_{H,c}G(x,p,u) \Bigr) \Big|_{\mathrm{J}^1_{u^c_-}} \geqslant  \lambda_+-\varepsilon.
$$
\end{lemma}

\begin{proof}
Fix once and for all a finite level $b>c_0$ and choose a smooth strict subsolution $\psi\in\mathcal{SS}^{b}$. For every $c>b$, Lemma \ref{lem:pf-1.1-2}, applied with $c'=b$, gives the test function $G(x,p,u)=\ln(u-\psi(x))$ and

\begin{align*}
&\,  \Bigl(\lambda(x) + \mathfrak{L}_{H,c}G(x,p,u) \Bigr) \Big|_{\mathrm J^1_{u_-^c}}  \geqslant \frac{c-b}{u^c_-(x)-\psi(x)} >0.
\end{align*}
Lemma \ref{lem:Step4-bounded} implies, uniformly for $x\in M$,
$$
0<u^c_-(x)-\psi(x) \leqslant   \bigg(\frac{1}{\lambda_+}+\delta\bigg) e^{\Lambda   \delta}\cdot(c-b) +  \overline B_{b,\delta}(x) - \psi(x).
$$
Since $\overline B_{b,\delta}-\psi$ is a fixed continuous function on the compact manifold $M$,
\begin{equation}\label{eq:4-8}
\liminf_{c\to +\infty}\ \min_{x\in M}\frac{c-b}{u^c_-(x)-\psi(x)} \geqslant \left( \frac{1}{\lambda_+}+ \delta\right)^{-1} e^{-\Lambda  \delta}.
\end{equation}
In view of the facts that $\delta$ is arbitrary in the above estimate and $\lim_{\delta\to 0} ( \frac{1}{\lambda_+ }+ \delta)^{-1} e^{-\Lambda  \delta} = \lambda_+$, 
we have
$$
\liminf_{c\to +\infty}\ \min_{x\in M}\frac{c-b}{u^c_-(x)-\psi(x)} \geqslant \lambda_+,
$$
and this is an equivalent formulation of the conclusion.
\end{proof}			
	
Finally, by Lemma \ref{asy-c} and the definition of $R (c)$, for any $\varepsilon>0$ and $c>C(\varepsilon)$,
\begin{align*}
\lambda_+ \geqslant	R(c):=\inf_{ \mu\in\mathfrak{M}_{u^c_-}}	 \int_{T^*M\times\R} \lambda \   d\mu
= &\, \inf_{ \mu\in\mathfrak{M}_{u^c_-}}	 \int_{T^*M\times\R}\Bigl(\lambda(x) + \mathfrak{L}_{H,c}G(x,p,u) \Bigr) \   d\mu\\
\geqslant &\,\min_{(x,u,p)\in\mathrm{J}^1_{u^c_-}}\Bigl(\lambda(x) + \mathfrak{L}_{H,c}G(x,p,u) \Bigr)\geqslant \lambda_+-\varepsilon.
\end{align*}
It follows that $\lim_{c\to \infty} R(c)= \lambda_+$.

\section{Proof of Theorem \ref{thm:4}}
In this section, we study the distribution of ergodic Mather measures for \eqref{HJs} in the phase space $T^\ast M\times\R$ via the front projection $\pi:T^\ast M\times\R\rightarrow M\times\R;\,(x,p,u)\mapsto(x,u)$. We shall always assume that $c\geqslant c_0$ and that the Hamilton-Jacobi equation \eqref{HJs} admits at least one solution (otherwise $\mathfrak{M}^c$ is empty and there is nothing to prove). Recall that
$$
\mathcal{M}(c):=\overline{\bigcup_{\mu\in\mathfrak{M}^c}\text{supp}(\mu)},\quad\quad\mathcal{M}_0(c):=\overline{\bigcup_{\mu\in\mathfrak{M}_0^{\text{erg,c}}}\text{supp}(\mu)},
$$
and $\mathfrak{M}^{\text{erg,c}}_0$ denotes the set of ergodic measures such that $\int \lambda \ d\mu=0$. To start with, we note the following connection between solution semigroups and invariant sets. For any $(x_0,p_0,u_0)\in T^\ast M\times\R$ and $t\in\R$, we use $(x(t),p(t),u(t))=\Phi_c^t(x_0,p_0,u_0)$ to denote the $\Phi_c^t$-orbit passing through $(x_0,p_0,u_0)$.

\vspace{0.5em}
Theorem \ref{thm:4} is a direct corollary of the following theorem, in which we give more information about the distribution of ergodic Mather measures associated with \eqref{HJs}.

\begin{theorem}\label{prop-thm:4}
	Assume \eqref{HJs} admits a solution. Then $c\geqslant c_0$ and 
\begin{enumerate}
  \item[(D1)] If $(+)$ holds,\,\, then for any ergodic measure $\mu\in\mathfrak{M}^c$ with $\int \lambda d \mu>0, \text{supp}(\mu)\subset\mathrm{J}^1_{u^c_-}$. Moreover,
      \begin{itemize}
        \item[(NC)] for $c>c_0,\,\, \mathcal{M}_0(c)=\emptyset,\,  \, \mathcal{M}(c)\subset\mathrm{J}^1_{u^c_-}$.
        \item[(C)] for $c=c_0,\,\,\mathcal{M}_0(c_0)\cap\mathrm{J}^1_{u^{c_0}_-}\neq\emptyset\quad\text{and}\quad\mathcal{M}_0(c_0)\subset \{(x,p,u)\in T^*M: u\leqslant u_-^{c_0}(x)\}$.
      \end{itemize}

  	\item[(D1')] If $(-)$ holds,\,\, then for any ergodic measure $\mu\in\mathfrak{M}^c$ with $\int \lambda d \mu<0, \text{supp}(\mu)\subset\mathrm{J}^1_{v^c_+}$. Moreover,
  \begin{itemize}
  \item[(NC)] for $c>c_0, \,\, \mathcal{M}_0(c)=\emptyset,\,  \,  \mathcal{M}(c)\subset\mathrm{J}^1_{v^c_+}$.
  \item[(C)] for $c=c_0,\,\,\mathcal{M}_0(c_0)\cap\mathrm{J}^1_{v^{c_0}_+}\neq\emptyset\quad\text{and}\quad\mathcal{M}_0(c_0)\subset \{(x,p,u)\in T^*M: u\geqslant v_+^{c_0}(x)\}$.
  \end{itemize}

  \item[(D2)]If ($\pm$) holds, \,\, then for any ergodic measure $\mu\in\mathfrak{M}^c$ with  $\int \lambda d \mu>0\,(\text{resp.}<0),\,\text{supp}(\mu)\subset\mathrm{J}^1_{u^{c}_-}($resp. $\mathrm{J}^1_{v^{c}_+})$. Moreover,
      \begin{itemize}
        \item[(NC)] for $c>c_0,\,\,  \mathcal{M}_0(c)=\emptyset,\,\mathcal{M}(c)\subset\mathrm{J}^1_{u^c_-}\sqcup\mathrm{J}^1_{v^c_+}$.
      \end{itemize}
\end{enumerate}
\end{theorem}
 
\begin{remark}
$c\geqslant c_0$ is implied by the assumption and Proposition \ref{prop:1}.
\end{remark}

To prove the above theorem, we begin with 

\begin{lemma}\label{lem-thmD-lem+}
Given $\varphi\in C(M,\R)$ and a point $(x_0,p_0,u_0)\in T^\ast M\times\R$, the following statements hold.
\begin{enumerate}[(1)]
  \item If $(x_0,u_0)\in \{(x,u)\in M\times \R: u\geqslant \varphi(x)\}$, then for any $t\geqslant0, (x(t),u(t))\subset\{(x,u)\in M\times \R: u\geqslant T_t^{c,-}\varphi(x)\}$.

  \item If $(x_0,u_0)\in \{(x,u)\in M\times \R: u\leqslant \varphi(x)\}$, then for any $t>0, (x(-t),u(-t))\subset \{(x,u)\in M\times \R: u\leqslant T_t^{c,+}\varphi(x)\}$.
\end{enumerate}
\end{lemma}

\begin{proof}
By symmetry, we focus only on (1). The proof is direct by noting that for any $t\geqslant0$,
\begin{equation*}
T_t^{c,-}\varphi(x(t))=\inf_{y\in M} h^c_{y, \varphi(y) }(x(t),t) \leqslant h^c_{x_0,\varphi(x_0)}(x(t),t)\leqslant h^c_{x_0,u_0}(x(t),t)\leqslant u(t),
\end{equation*}
where the last inequality follows from Proposition \ref{Minimality} (1). This completes the proof.
\end{proof}

 Ergodic Mather measures play a central role in the statement of Theorem \ref{prop-thm:4} and Theorem \ref{thm:4}. 
\begin{lemma}\label{lem:ergodic-associated}
Assume \eqref{HJs} admits a solution. Then every ergodic Mather measure $\mu\in\mathfrak{M}^c$ belongs to $\mathfrak{M}^c_{u_-}$ for some $u_-\in\mathcal{S}^{c,-}$.
\end{lemma}

\begin{proof}

Recall that every Mather measure is supported on the Ma\~{n}\'{e} set of the contact Hamiltonian system
$$
  \mathcal{N}^c:=\bigcup_{u_-\in\mathcal S^{c,-}}
  \mathcal{N}^{c}_{u_-},\quad\text{where} \,\,\,\mathcal{N}^c_{u_-}= \bigcap_{t\leqslant 0}\Phi^t_{c}(\mathrm{J}^1_{u_-}).
$$
We choose a $\mu$-generic point $z_0=(x_0,u_0,p_0)\in$ supp $\mu$, then $\Phi^t_{c}(z_0)$ is contained in $\mathcal{N}^c$ for every $t\in\R$, and by the above definition, there is some $u_-\in\mathcal S^{c,-}$ such that 
$$
 \Phi^t_{c}(z_0)\subset\mathcal{N}^c_{u_-},\quad\forall \, t\in\R.
$$
It is well known that $\Phi^t_{c}(z_0)$ dense in supp $\mu$.  Since $\mathcal{N}^c_{u_-}$ is closed, it follows that supp $\mu\subset\mathcal{N}^c_{u_-}$ and $\mu\in\mathfrak{M}^c_{u_-}$.

\end{proof}

\subsection{Proof of (D1)}
Assume $(+)$ holds. For the non-critical case {\bf  (NC)}, the proof is almost obvious: if $c>c_0$, then \eqref{HJs} admits a unique solution $u_-^c$ and  $\mathfrak{M}^c=\mathfrak{M}_{u^c_-}$. This proves the first conclusion and shows that $\mathcal{M}(c)\subset\mathrm{J}^1_{u^c_-}$. By Theorem \ref{thm:3}, for any Mather measure $\mu\in\mathfrak{M}_{u^c_-}$,
$$
\int_{T^*M\times \R} \lambda d\mu\geqslant a(c)>0.
$$
Thus $\mathfrak{M}^c_0=\emptyset$ and $\mathcal{M}_0(c)=\emptyset$, which completes the proof.

\vspace{1em}
\noindent\textit{Proof of {\bf(C)}}:\quad Since the discount factor $\lambda(x)$ is non-negative, for any $\mu\in\mathfrak{M}_{u_-^{c_0}}, \int_{T^*M\times \R}\lambda \, d\mu\geqslant0$. Due to the weak compactness of $\mathfrak{M}_{u_-^{c_0}}$, the infimum in
\begin{equation}\label{eq:thmD-step1-claim1}
a(c_0)=\inf_{ \mu\in\mathfrak{M}_{u^{c_0}_-}}\int_{T^*M\times\R} \lambda\ d\mu
\end{equation}
is attained. If $a(c_0)>0$, then by Theorem \ref{RWY-1.2} and Proposition \ref{Prop:3.2}, $u^c_-$ is asymptotically stable and $c>c_0$, which leads to a contradiction. Thus $a(c_0)=0$ and any \textbf{extreme point} of the compact convex set $a^{-1}(0)$ is an ergodic measure in $\mathfrak{M}_{u^{c_0}_-}$.
As a consequence, $\mathcal{M}_0(c_0)\cap \mathrm{J}^1_{u^{c_0}_-}\neq\emptyset$. We use the maximality of $u^{c_0}_-$ in Proposition \ref{prop:new-NW-JDE} to conclude that
\begin{equation}\label{eq:proof-D1-1}
\mathcal{M}_0(c_0)\subset\mathcal{M}(c_0)\subset \{(x,p,u)\in T^*M\times\R: u\leqslant u_-^{c_0}(x)\}.
\end{equation}
To complete the proof, it suffices to show that
\begin{itemize}
  \item[($\diamondsuit$)] for any ergodic measure $\mu\in \mathfrak{M}^{c_0}\setminus\mathfrak{M}_{u^{c_0}_-}$, we have $\int_{T^*M \times \R} \lambda d\mu=0$.
\end{itemize}

We argue by contradiction and assume that there is an ergodic measure $\mu_\ast\in \mathfrak{M}^{c_0}\setminus\mathfrak{M}_{u^{c_0}_-}$ satisfying $\int_{T^*M \times \R} \lambda d\mu_\ast=2a>0$. By Lemma \ref{lem:ergodic-associated}, we choose $u_-\in\mathcal{S}^{c_0,-}$ such that $\mu_\ast\in\mathfrak{M}_{u_-}$ and set $v_+= \lim_{t\to +\infty} T_t^{c_0,+}u_-\in\mathcal{S}^{c_0,+}$. By the maximality of $u^{c_0}_-$ and Proposition \ref{mono}, we have
\begin{equation}\label{eq:proof-D1-2}
v_+\leqslant u_-\leqslant u^{c_0}_-\quad\text{and}\quad\mu_\ast\in \mathfrak{M}_{v_+}.
\end{equation}
Since $\mu_\ast\notin\mathfrak{M}_{u^{c_0}_-}$, it follows that
\begin{equation}\label{eq:proof-D1-3}
\mu_\ast(\{(x,p,u)\in T^\ast M\times\R: u<u^{c_0}_-(x)\})>0.
\end{equation}
For $\varepsilon>0$, \eqref{eq:proof-D1-3} allows us to find a calibrated curve $\gamma_+:[0,+\infty) \to M$ of $v_+$ and an open neighborhood $\mathcal{U}_\ast$ of $(x_0,p_0,u_0)\in\spt(\mu_\ast)$ such that
\begin{enumerate}[(I)]
  \item $x_0=\gamma_+(0),\quad p_0=d_x v_+(\gamma_+(0)),\quad u_0=v_+(\gamma_+(0))$,
  \item $\mu_\ast(\mathcal{U}_\ast)>0$,
  \item $u^{c_0}_-(\gamma_+(0))\geqslant v_+(\gamma_+(0))+2\varepsilon$ and $u^{c_0}_-(x)\geqslant v_+(x)+\varepsilon$ for any $(x,p,u)\in\mathcal{U}_\ast$.
\end{enumerate}
For any $t\geqslant0$ and $\Delta t>0$, we have
\begin{align*}
u^{c_0}_{-}(\gamma_+(t+ \Delta t))-u^{c_0}_{-}(\gamma_+(t))&\,\leqslant \int_{t}^{t+ \Delta t} l(\gamma_+(\tau), \dot{\gamma}_+(\tau)) + c_0 - \lambda(\gamma_+(\tau))\cdot u^{c_0}_{-}(\gamma_+(\tau))\  d\tau, \\
v_{+}(\gamma_+(t+ \Delta t)) - v_{+}(\gamma_+(t))&\,=\int_{t}^{t+\Delta t} l(\gamma_+(\tau), \dot{\gamma}_+(\tau)) + c_0 - \lambda(\gamma_+(\tau))\cdot v_{+}(\gamma_+(\tau))\ d\tau.
\end{align*}
Let $w(t)= u^{c_0}_{-}(\gamma_+(t))- v_{+}(\gamma_+(t))$. Then \eqref{eq:proof-D1-2} implies that $w(t)\geqslant0$ and is Lipschitz in $t$. Subtracting the two inequalities gives
$$
w(t+\Delta t)-w(t)\leqslant-\int_t^{t+\Delta t}\lambda(\gamma_+(\tau))\cdot w(\tau) d\tau
$$
Dividing both sides by $\Delta t$ and letting $\Delta t\to 0^+$, we obtain that for $t\in [0,+\infty)$,
$$
w'(t)\leqslant -\lambda(\gamma_+(t)) \cdot w(t).
$$
Integrating the above inequality yields for all $t>0$,
$$
w(t)\leqslant w(0) e^{-\int_0^t \lambda (\gamma_+(\tau))d\tau}.
$$
By Birkhoff's ergodic theorem, there is $\mathcal{T}(a)>0$ such that for all $t\geqslant\mathcal{T}(a)$, we have that $\int_0^{t}\lambda(\gamma_+(\tau))\ d\tau>at$, and
\begin{equation}\label{eq:proof-D1-4}
\lim_{t\to +\infty}w(t)\leqslant\lim_{t\to +\infty}w(0)e^{-at}=0.
\end{equation}
However, by setting $(x(t),p(t),u(t))=\Phi^t_c(x_0,p_0,u_0)$, it follows that $x(t)=\gamma_+(t)$ for all $t\geqslant0$. Applying (II) above, there is a sequence $\{t_n\}_{n\geqslant1}$ with $\lim_{n\rightarrow\infty}t_n=+\infty$ such that $(x(t_n),p(t_n),u(t_n))\in\mathcal{U}_\ast$ and
$$
\lim_{n\to +\infty} w(t_n)=\lim_{t\to +\infty} u^{c_0}_{-}(x(t_n))- v_{+}(x(t_n))\geqslant\varepsilon>0,
$$
which is inconsistent with \eqref{eq:proof-D1-4}.

\subsection{Proof of (D2)}

Assume $(\pm)$ holds.  We begin to prove the first conclusion in {\bf (D2)}, that is,   

\begin{itemize}
\item for any ergodic measure $\mu\in\mathfrak{M}^c$ with $\int \lambda d \mu>0$, then $\text{supp}(\mu)\subset\mathrm{J}^1_{u^{c}_-}$.	
\end{itemize}
The proof for the case of $\int_{T^*M\times\mathbb R}\lambda\,d\mu<0$ goes in a completely similar way.

\vspace{0.5em}

Let $\mu\in\mathfrak M^{c}$ be ergodic. By Lemma \ref{lem:ergodic-associated}, there exists
$u_-\in\mathcal S^{c,-}$ such that
$\mu\in\mathfrak M_{u_-}^{c}$.
Let $v_+:=\lim_{t\to+\infty}T_t^{c,+}u_-.$
By Proposition \ref{prop:new-NW-JDE} and Proposition \ref{mono}, we have
\begin{equation}\label{eq:D2-critical-order}
 v_+^{c} \leqslant v_+\leqslant u_-\leqslant u_-^{c}.
\end{equation}
We choose an arbitrary Birkhoff generic recurrent point $z_0=(x_0,p_0,u_-(x_0))\in\operatorname{supp}(\mu)$, and write 
$$
z(t)=(x(t),p(t),u_-(x(t))):=\Phi^{t}_{c}(z_0).
$$
Since $\mu\in\mathfrak M_{u_-}^{c}$ and its support is invariant,
the curve $x(\cdot)$ is a complete calibrated curve of $u_-$.
In particular, for every $t_1<t_2$,
\begin{equation}\label{eq:D2-u-calibration}
u_-(x(t_2))-u_-(x(t_1))
 =
\int_{t_1}^{t_2}
\left[
l(x(s),\dot x(s))
-\lambda(x(s))u_-(x(s))+c
\right]ds .
\end{equation}
Since $u^c_-\in\mathcal{S}^{c,-}$, we apply Definition \ref{def:weakKAM} to have 
\begin{equation}\label{eq:D2-u-dom}
u^c_-(x(t_2))-u^c_-(x(t_1))
 \leqslant
\int_{t_1}^{t_2}
\left[
l(x(s),\dot x(s))
-\lambda(x(s))u^c_-(x(s))+c
\right]ds .
\end{equation}
Define
$w(t):=u_-^{c}(x(t))-u_-(x(t))\geqslant 0.$ Subtracting  \eqref{eq:D2-u-calibration} from \eqref{eq:D2-u-calibration}, we obtain 
$$
w'(t)\leqslant-\lambda(x(t))w(t)
$$
 for almost every $t$.
Consequently, for $t\geqslant0$, 
\begin{equation}\label{eq:D2-positive-gap}
w(t)\leqslant
w(0)\exp\left\{-\int_0^t\lambda(x(s))\,ds\right\}.
\end{equation}
By Birkhoff's ergodic theorem,
\[
\lim_{t\to+\infty}
\frac1t\int_0^t\lambda(x(s))\,ds
=\int_{T^*M\times\mathbb R}\lambda\,d\mu>0.
\]
It follows from \eqref{eq:D2-positive-gap} that $\lim_{t\to +\infty}w(t)=0$.

\vspace{0.5em}
On the other hand, the recurrence of $\Phi^{t}_{c}(z_0)$ gives a sequence $t_n\to+\infty$
such that $z(t_n)\to z_0$. Hence, by continuity of $u_-$ and $u^c_-$, $w(0)=\lim_{n\to\infty}w(t_n)=0.$
Since the Birkhoff generic recurrent points form a set of full
$\mu$-measure and $u_-^{c}-u_-$ is continuous, we conclude that
$$
u_-^{c}(x)=u_-(x),\quad\text{for any }x\in \operatorname{supp}(\mu).
$$
Therefore, every orbit in $\operatorname{supp}(\mu)$ is also calibrated by $u_-^c$. Thus $\operatorname{supp}(\mu)\subset \mathrm J^1_{u_-^{c}}.$
This completes the proof.

\vspace{1em}
Next, we show that  for $c>c_0,\,\,  \mathcal{M}_0(c)=\emptyset,\,\mathcal{M}(c)\subset\mathrm{J}^1_{u^c_-}\sqcup\mathrm{J}^1_{v^c_+}$. By Proposition \ref{prop:2}, $v_+^c<u_-^c$ on $M$. We claim that
\textit{for any ergodic measure $\mu\in\mathfrak{M}^c$,
$$
\pi(\spt(\mu))\subset\{(x,u)\in M\times\R: u\geqslant u^c_-(x)\}\sqcup\{(x,u)\in M\times\R:u\leqslant v^c_+(x)\},
$$
where the right-hand side is a disjoint union of two closed sets}.

\vspace{0.5em}
\noindent\textit{Proof of the claim:}  By \eqref{eq:D2-critical-order} and
$\mu\in\mathfrak M^c_{u_-}$, the support of $\mu$ lies in the closed
slab $v_+^c(x)\leqslant u\leqslant u_-^c(x)$. It suffices   to show that
$$
\mu(\{(x,p,u)\in T^\ast M\times\R: v^c_+(x)<u<u^{c}_-(x)\})=0.
$$
We argue by contradiction  If the open strip has positive $\mu$-measure, choose a recurrent point  $(x_0, p_0, u_0)\in\spt(\mu)$ such that $v^c_+(x_0)<u_0<u^{c}_-(x_0)$. Thus for some $\varepsilon_0>0$, we can find an open neighborhood $\mathcal{U}_0$ of $(x_0,p_0,u_0)$ such that
\begin{enumerate}[(I)]
  \item $\mu(\mathcal{U}_0)>0$,
  \item $u^{c}_-(x)-\varepsilon_0\geqslant u$ for any $(x,p,u)\in\mathcal{U}_0$.
\end{enumerate}
Setting $(x(t),p(t),u(t))=\Phi^t_c(x_0,p_0,u_0)$ and  using recurrence of the chosen point, we find a sequence $\{t_n\}_{n\geqslant1}$ with $\lim_{n\rightarrow\infty}t_n=+\infty$ such that $(x(t_n),p(t_n),u(t_n))\in\mathcal{U}_0$. Since $u_0>v^c_+(x_0)$, we can choose $\varphi\in C(M,\R)$ with $\varphi(x)>v^c_+(x)$ and $u_0>\varphi(x_0)$. By Lemma \ref{lem-thmD-lem+}, we deduce that for each $n\geqslant1$,
$$
u^{c}_-(x(t_n))-\varepsilon_0\geqslant u(t_n)\geqslant T_{t_n}^{c,-}\varphi(x(t_n)).
$$
This contradicts Lemma \ref{+-}, which states that $\lim_{n\rightarrow\infty}| T_{t_n}^{c,-}\varphi(x(t_n))-u^c_-(x(t_n))|=0$.\qed

\vspace{0.5em}
\noindent Now we use the ergodicity of $\mu$ again to deduce that $\spt(\mu)$ is connected (since there is a $\Phi^t_c$-orbit dense in the support of $\mu$) and hence can belong only to one of these two closed sets. If $\pi(\spt(\mu))\subset\{(x,u)\in M\times\R: u\geqslant u^c_-(x)\}$, by Lemma \ref{lem:ergodic-associated}, there is $u_-\in\cS$ such that $\mu\in\mathfrak{M}^c_{u_-}$.  By maximality, $u_-\leqslant u_-^c$, while the assumption and $\mu\in\mathfrak M^c_{u_-}$ give $u_-=u_-^c$ on $\operatorname{supp}(\mu)$. Every orbit in this support is $u_-$-calibrated and hence also $u_-^c$-calibrated. Therefore $\mu\in\mathfrak M^c_{u_-^c}$. Likewise, if $\pi(\spt(\mu))\subset\{(x,u)\in M\times\R: u\leqslant v^c_+(x)\}$, then  $u_-=v_+^c$ on $ \operatorname{supp} (\mu)$ by \eqref{eq:D2-critical-order}; the same calibration argument gives $\mu\in\mathfrak M^c_{v_+^c}$. Thus every ergodic measure $\mu\in\mathfrak M^c$ belongs to $\mathfrak M^c_{u_-^c}\cup\mathfrak M^c_{v_+^c}$.
Applying Theorem \ref{thm:3} and Remark \ref{rmk-main-thm}, we have
$$
\int_{T^*M\times \R} \lambda d\mu >0, \quad \forall \mu \in  \mathfrak{M}_{u^c_-} \quad \text{and} \quad  \int_{T^*M\times \R} \lambda d\mu <0, \quad \forall \mu \in  \mathfrak{M}_{v^c_+}.
$$
This implies that $\int_{T^*M\times\R}\lambda\,d\mu\neq 0$ for any ergodic measure $\mu\in \mathfrak{M}^c$. Thus $\mathcal{M}_0(c)=\emptyset$ and $\mathcal M(c)\subset\mathrm J^1_{u_-^c}\sqcup\mathrm J^1_{v_+^c}$.

\appendix
\section{Preliminaries}\label{Appendix-A}
In this appendix, we collect the definitions and properties of crucial objects and tools used in the main body of this paper. The content is divided into three parts: the first part is devoted to some definitions related to the theory of viscosity solutions; in the second part, we introduce the variational setting for contact Hamiltonian systems developed in \cite{WWY1}-\cite{WWY3} with an emphasis on the action functions; in the final part, we focus on the basic properties of the solution semigroups for the associated evolutionary Hamilton-Jacobi equations.

\subsection{Viscosity solutions of Hamilton-Jacobi equations}
We begin by repeating the definition of generalized differentials of a continuous function. In the following definition, we use $\exp^{-1}_x:M\rightarrow T_xM$ to denote the local inverse of the exponential map $\exp_x$ based at $x$ and associated with the auxiliary Riemannian metric on $M$; this inverse exists by the inverse map theorem.

\begin{definition}\cite[Definition 1.3]{Tran}
Let $u:M\rightarrow\R$ be a continuous function. For any $x\in M$ and open neighborhood $U\subset M$ of $x$, the convex sets
\begin{equation}\label{sub-sup diff}
\begin{split}
D^-u(x)&=\bigg\{p\in T^\ast_xM\,:\,\liminf_{d(x,y)\rightarrow0_+}\frac{u(y)-u(x)-\langle p,\exp^{-1}_x(y)\rangle}{d(x,y)}\geqslant0\,\,\text{for any }y\in U\bigg\},\\
D^+u(x)&=\bigg\{p\in T^\ast_xM\,:\,\limsup_{d(x,y)\rightarrow0_+}\frac{u(y)-u(x)-\langle p,\exp^{-1}_x(y)\rangle}{d(x,y)}\leqslant0\,\,\text{for any }y\in U\bigg\}.
\end{split}
\end{equation}
are called the (Fr\'{e}chet) subdifferential and superdifferential of $u$ at $x$, respectively.

\vspace{0.5em}
Let us further assume that $u$ is Lipschitz. By Rademacher's theorem, $u$ is differentiable almost everywhere with respect to the volume measure. In particular, for each $x\in M$, the set
\begin{equation}
D^\ast u(x)=\bigg\{p\in T^\ast_xM\,:\exists\,\, x_k\in M\,\,\text{such that}\,\,\lim_{k\rightarrow\infty}x_k=x, u\text{ is differentiable at }x_k\text{ and }p=\lim_{x_k\rightarrow x}d_{x_k}u\bigg\}
\end{equation}
is non-empty and compact. The $1$-graph of the function $u$ is defined by
\begin{align*}
\mathrm{J}^1_{u} :=&\,\left\{ (x, p, u) : p\in D^\ast u(x),\,\, u = u(x) \right\}=\mathrm{cl}\left( \left\{ (x, p, u) : x \text{ is a point of differentiability of } u,\,\, p =d_x u,\,\, u = u(x) \right\} \right),
\end{align*}
where $\mathrm{cl}(B)$ denotes the closure of $B \subset T^*M \times \mathbb{R}$. It is easily seen that $\mathrm{J}^1_{u}$ is compact.
\end{definition}

Let $H:T^\ast M\times\R\rightarrow\R; (x,p,u)\mapsto H(x,p,u)$ be a continuous function, called contact Hamiltonian, and let $c\in\R$. In this part, we work in slightly greater generality and consider the static Hamilton-Jacobi equations of the form
\begin{equation}\label{hjs}\tag{S$_c$}
H(x,d_xu,u)=c
\end{equation}
and the evolutionary Hamilton-Jacobi equations of the form
\begin{equation}\label{hje}\tag{E$_c$}
\begin{cases}
\partial_t U+H(x,\partial_x U,U)=c,\quad\, (x,t)\in M\times(0,+\infty)\\
\hspace{5.3em}U(x,0)\,=\varphi(x),\hspace{1.2em}x\in M.
\end{cases}
\end{equation}
Using generalized differentials, the definition of the viscosity solution to Hamilton-Jacobi equations can be given as
\begin{definition}\cite[Definition 1.1, 1.5]{Tran}
Let $u:M\rightarrow\R$ be a continuous function. Then $u$ is called a
\begin{itemize}
  \item viscosity subsolution to \eqref{hjs} if $H(x,p,u(x))\leqslant c$ for any $x\in M$ and $p\in D^+u(x)$.
  \item viscosity supersolution to \eqref{hjs} if $H(x,p,u(x))\geqslant c$ for any $x\in M$ and $p\in D^-u(x)$.
  \item viscosty solution of \eqref{hjs} if it is both a viscosity subsolution and a viscosity supersolution of \eqref{hjs}.
\end{itemize}
Furthermore, a viscosity subsolution $u$ of \eqref{hjs} is said to be \textbf{strict} at $x$ if $H(x,p,u(x))<c$ for any $p\in D^+u(x)$.

\vspace{0.5em}
Let $U:[0,+\infty)\times M\rightarrow\R$ be a continuous function. The definition of a viscosity subsolution extends to the evolutionary case by asserting that $U$ is a subsolution of \eqref{hje} if and only if $U(0,\cdot)\leqslant\varphi$ and for any $(x,t)\in M\times(0,+\infty)$,
$$
p_t+H(x,p_x,U(t,x))\leqslant c\quad\text{for each}\,\,\,(p_t,p_x)\in D^+U(t,x).
$$
Replacing $D^+$ by $D^-$ and reversing the inequalities above, we can define the supersolution and solution in the same manner as in the static case.
\end{definition}
Note that the solvability of \eqref{hjs}, called \textbf{generalized ergodic problem}, has receive research interests recently, see \cite{JMT} for more results as well as how to treat this problem by PDE techniques.

A direct consequence of the definition of a viscosity subsolution is that a Lipschitz continuous subsolution $v$ to \eqref{hjs} is also an almost-everywhere subsolution, i.e. $H(x,d_xv,v(x))\leqslant c$ for a.e. $x\in M$. A useful fact is
\begin{lemma}\label{sub-criteria}
Assume $H:T^\ast M\times\R\rightarrow\R$ is continuous and convex in $p$. Then a Lipschitz function $v:M\rightarrow\R$ is a subsolution to \eqref{hjs} if and only if $H(x,d_xv,v(x))\leqslant c$ holds almost everywhere.
\end{lemma}

Two more properties concerning the smooth approximation of subsolutions and the regularity of solutions are needed.

\begin{lemma}\label{smooth-app}\cite[Lemma 2.2]{DFIZ}
Assume that $H:T^\ast M\rightarrow\R; (x,p)\mapsto H(x,p)$ is continuous and convex in $p$ and that $v$ is a Lipschitz subsolution to \eqref{hjs}. Then for all $\varepsilon>0$, there is a smooth function $v_\eps:M\rightarrow\R$ such that $\|v-v_\eps\|_\infty<\eps$ and $H(x,d_xv_\eps)\leqslant c+\eps$.
\end{lemma}

\begin{lemma}\label{regularity}\cite[Theorem 5.3.7]{CS}
Assume $H:T^\ast M\times\R\rightarrow\R$ is locally Lipschitz continuous and strictly convex in $p$. Then all Lipschitz solutions to \eqref{hjs} are locally semi-concave. In particular, for any such solution $u, D^+u(x)=coD^\ast u(x)\neq\emptyset$.
\end{lemma}

To conclude this part, let us recall that the \textbf{characteristic system} for the Cauchy problem \eqref{hje} is the contact Hamiltonian system associated with $H-c$, written in the coordinates $(x,p,u)$ as
\begin{equation}\label{ch}\tag{CH$_c$}
\begin{cases}
\dot{x} = \frac{\partial H}{\partial p}(x, p, u) \\
\dot{p} = -\frac{\partial H}{\partial x}(x, p, u) - \frac{\partial H}{\partial u}(x, p, u) p \\
\dot{u} = \frac{\partial H}{\partial p}(x, p, u) \cdot p - H(x, p, u)+c
\end{cases}
\end{equation}
If we set $\breve{H}(x,p,u):=H(x,-p,-u)$, the characteristic system for the Cauchy problem
\begin{equation}\label{reverse ham}\tag{E$'_c$}
\begin{cases}
\partial_t u+\breve{H}(x,\partial_x u,u)=c,\quad (x,t)\in M\times[0,+\infty),\\
u(0,x)=-\varphi(x),
\end{cases}
\end{equation}
is obtained by replacing $H$ in \eqref{ch} with $\breve{H}$. A simple change of coordinate $(X,U,P)=(x,-u,-p)$ together with \textbf{a time inversion} transforms this new system into \eqref{ch}. Thus the characteristic systems for \eqref{hje} and \eqref{reverse ham} share the same dynamics, especially the same set of invariant measures.

Weak KAM theory provides a connection between the Aubry-Mather theory of Lagrangian systems and the Crandall-Lions theory of viscosity solutions of the Hamilton-Jacobi equation. For a comprehensive and up-to-date  introduction to weak KAM theory, we recommend the excellent monograph \cite{Zbook}. Notably, according to \cite{WWY2}, viscosity solutions and backward weak KAM solutions are essentially equivalent. However, the distinctive advantage of the weak KAM framework lies in its ability to connect the solutions of the equation with the dynamics of contact Hamiltonian systems through a geometric tool: the \textbf{calibrated curves}.

To establish this connection, we first introduce the notion of backward (resp. forward) weak KAM solutions.

\begin{definition}\label{def:weakKAM}
Let $u_-\in C(M,\mathbb{R})$ (resp. $u_+\in C(M,\mathbb{R})$). We say that $u_-$ (resp. $u_+$) is a backward (resp. forward) weak KAM solution of \eqref{hjs} if it satisfies the following two conditions:

\begin{itemize}
\item [(1)] For every continuous, piecewise-$C^1$ curve $\gamma:[t_1,t_2]\rightarrow M$, we have
\begin{align}\label{do}
u_\pm(\gamma(t_2))-u_\pm(\gamma(t_1))\leqslant\int_{t_1}^{t_2}\bigl[L\big(\gamma(s),\dot{\gamma}(s),u_\pm(\gamma(s))\big)+c\bigr] \ ds;
\end{align}

\item [(2)] For every $x\in M$, there exists a $C^1$ curve $\gamma:(-\infty,0]\rightarrow M$ (resp. $\gamma:[0,+\infty)\rightarrow M$) with $\gamma(0)=x$ such that for any $t_1,t_2\in(-\infty,0]$ (resp. $t_1,t_2\in[0,+\infty)$),
\begin{align}\label{cali1}
u_\pm(\gamma(t_2))-u_\pm(\gamma(t_1))=\int^{t_2}_{t_1}\bigl[L\big(\gamma(s),\dot{\gamma}(s),u_\pm(\gamma(s))\big)+c\bigr] \ ds.
\end{align}
Here, curves $\gamma:[a,b]\to M$ ($-\infty\leqslant a<b\leqslant+\infty$) satisfying \eqref{cali1} are called $(u_\pm, L, c)$-{\bf calibrated curves} on $M$.
\end{itemize}
\end{definition}

Moreover, it follows directly from the definition that for any $u_-\in \mathcal{S}^{c,-}$ and $u_+\in \mathcal{S}^{c,+} $, $T^{c,\pm}_t u_\pm =u_\pm$ for every $t>0$.
By Proposition \ref{Implicit variational}, the above equality implies that for any $x\in\mathbb{R}$, there is an orbit
\[
(x(t),p(t),u(t)),\quad t\in(-\infty,0] , \quad  (resp. \ t \in[0,+\infty))
\]
of the characteristic system \eqref{ch} with the contact Hamiltonian $H$ such that $x(0)=x,\,p(0)\in D^\ast u_\pm(x),\,u(t)=u_\pm(x(t))$. Here, $D^{\ast}u_\pm(x)$ denotes the reachable gradients of the solution $u_\pm$ at $x\in M$. Moreover, for any $t<0$ (resp. $t>0$) , $x(t)\in\mathcal{D}u_\pm $ with $u'_\pm(x(t))=p(t)$. For such an orbit and any $s\leqslant t\leqslant 0$ (resp.$ 0\leqslant s \leqslant t$), it follows that
\[
u_\pm(x(t))-u_\pm(x(s))=\int^{t}_{s}\bigl[L(x(\tau), \dot x(\tau),u_\pm(x(\tau)))+c\bigr]\ d\tau.
\]
Due to this identity, $(x(t),p(t),u(t))$, with $t\in(-\infty,0] $(resp. $t\in[0,+\infty)$, is said to be calibrated by $u_-$(resp. $u_+$)  or briefly called calibrated orbit in $T^*M\times \R$.

\subsection{Variational setting for Tonelli contact Hamiltonian systems}
We shall use $TM$ to denote the tangent bundle of $M$. As usual, a point of $T M$ will be denoted by $(x,\dot{x})$, where $x\in M$ and $\dot{x}\in T_x M$. A smooth function $H:T^{\ast}M\times\R\rightarrow\R, (x,p,u)\mapsto H(x,p,u)$ is said to be a Tonelli contact Hamiltonian if $H$ satisfies (H1)-(H2) (carried over to the case of contact Hamiltonian in an obvious way) and is uniformly Lipschitz in $u$, i.e., there is $\Lambda>0$ such that for any $u,u'\in\R$,
\[
|H(x,p,u)-H(x,p,u')|\leqslant\Lambda|u-u'|.
\]
The Lagrangian $L:TM\times\R\rightarrow\R$ associated with $H$ is defined as
\[
L(x,\dot{x},u)=\sup_{p \in T_x^*M}\{p \cdot \dot{x}-H(x,p,u)\}.
\]
Notice that \eqref{hjs} is a family of equations parametrized by $c$. We introduce $L^c(x,\dot{x},u):=L(x,\dot{x},u)+c$. The variational principle of \chc is encoded in the following implicitly defined action functions.
\begin{proposition}\cite[Theorem 2.1, 2.2]{WWY3}\label{Implicit variational}
Given $x_0\in M$ and $u_0,c\in \R$, there exist continuous functions $h^c_{x_0,u_0}(x,t)$ and $h_c^{x_0,u_0}(x,t)$ defined on $M\times (0,+\infty)$ by
\begin{equation}\label{eq:Implicit variational}
\begin{split}
h^c_{x_0,u_0}(x,t)=&\inf_{\substack{\gamma(t)=x\\ \gamma(0)=x_0}}\Big\{u_0+\int_0^t L^c(\gamma(\tau), \dot \gamma(\tau),h^c_{x_0,u_0}(\gamma(\tau) ,\tau )  )\ d\tau\Big\},\\
h_c^{x_0,u_0}(x,t)=&\sup_{\substack{\gamma(t)=x_0\\ \gamma(0)=x } }\Big\{u_0-\int_0^t L^c(\gamma(\tau), \dot \gamma(\tau),h_c^{x_0,u_0}(\gamma(\tau) ,t-\tau )  )\ d\tau\Big\},
\end{split}
\end{equation}
where the infimum and supremum are taken among Lipschitz continuous curves $\gamma:[0,t]\rightarrow M$ and are achieved. We call $h^c_{x_0,u_0}(x,t)$ the backward action function and $h_c^{x_0,u_0}(x,t)$ the forward action function.
\end{proposition}

\begin{remark}
Let $\gamma:[0,t]\rightarrow M$ achieve the infimum (resp. supremum) in \eqref{eq:Implicit variational} and
\[
x(s):=\gamma(s), \quad p(s):=\frac{\partial L}{\partial \dot x}(x(s),\dot x(s),u(s)), \quad u(s):=h^c_{x_0,u_0}(x(s),s)\,(\text{resp.}\,\,h_c^{x_0,u_0}(x(s),t-s)).
\]
Then $(x(s),p(s),u(s))$ is $C^1$ and satisfies \chc with $x(0)=x_0,x(t)=x\,($resp. $x(0)=x,x(t)=x_0)$ and $\lim_{s\to 0^+ }u(s)=u_0\,($resp. $\lim_{s\to t^-}u(s)=u_0)$.
\end{remark}

We collect the properties of the above action functions that are used in this paper into the following
\begin{proposition}{\cite{WWY2}}\label{Minimality}
For each $c\in\R$, the action function $h^c_{x_0,u_0}(x,t)\,\,($resp.\,\,$h^{x_0,u_0}_c(x,t)) $ satisfies
\begin{enumerate}
  \item[(1)] \textbf{(Minimality)} Given $x_0,x\in M $  and $u_0\in \R $, and $t>0$, let $S^{x,t}_{x_0,u_0}$ be the set of solutions $(x(s),p(s),u(s))$ of \chc on $[0,t]$ with $x(0)=x_0,x(t)=x,u(0)=u_0\,\,($resp. $x(0)=x, x(t)=x_0, u(t)=u_0)$. Then, for any $(x,t)\in M\times(0,+\infty)$,
    \begin{equation}\label{eq:Minimality}
    \begin{split}
    h^c_{x_0,u_0}(x,t)=&\inf \{u(t):(x(s),p(s),u(s))\in S^{x,t}_{x_0,u_0}\}.\\
    (\text{resp.}\,\,h_c^{x_0,u_0}(x,t)=&\sup \{u(0):(x(s),p(s),u(s))\in S^{x,t}_{x_0,u_0}\}.)
    \end{split}
    \end{equation}
    As a result, $h^c_{x_0,u_0}(x,t)=u\Leftrightarrow h_c^{x,u}(x_0,t)=u_0$.

  \item[(2)]\textbf{(Monotonicity)} Given $x_0\in M, u_0\in \R, u_1< u_2$, and $c_1<c_2$, for all $t>0$ and $x\in M$,
    \begin{align*}
    	h^c_{x_0,u_1}(x,t)< h^c_{x_0,u_2}(x,t),\quad h_c^{x_0,u_1}(x,t)< h_c^{x_0,u_2}(x,t),\\
      h^{c_2}_{x_0,u_0}(x,t)>h^{c_1}_{x_0,u_0}(x,t), \quad h_{c_2}^{x_0,u_0}(x,t)<h_{c_1}^{x_0,u_0}(x,t).
    \end{align*}

  \item[(3)]\textbf{(Markov property)} Given $x_0 \in M,u_0 \in \R $, for any $s,t>0$ and all $x\in M$,
	\begin{equation}\label{markov}
    \begin{split}
	&h^c_{x_0,u_0}(x, t+s)=\inf_{y\in M}h^c_{y,h^c_{x_0,u_0}(y,t)}(x,s),\\
    &h_c^{x_0,u_0}(x, t+s)=\sup_{y\in M}h_c^{y,h_c^{x_0,u_0}(y,t)}(x,s).
    \end{split}
	\end{equation}

  \item[(4)]\textbf{(Lipschitz continuity)} The function $(x_0,u_0,x,t)\mapsto h^c_{x_0,u_0}(x,t)\,\,($resp. $h_c^{x_0,u_0}(x,t))$ is locally Lipschitz continuous on $M\times \R\times M\times (0,+\infty) $.
\end{enumerate}
\end{proposition}

\subsection{Solution semigroups for evolutionary Hamilton-Jacobi equations}
Based on the backward\,/\,forward action functions defined above, we define, for each $c\in\R$, two families of nonlinear operators $\{T^{c,\pm}_t\}_{t\geqslant 0}$. For each $\varphi\in C(M,\R)$ and $(x,t)\in M\times(0,+\infty)$,
\begin{equation}\label{eq:Tt-+ rep}
\begin{split}
T^{c,-}_t\varphi(x):=\inf_{y\in M}h^c_{y,\varphi(y)}(x,t),\\
T^{c,+}_t\varphi(x):=\sup_{y\in M}h_c^{y,\varphi(y)}(x,t).
\end{split}
\end{equation}
By setting $T^{c,\pm}_0\varphi=\varphi$, one can show that

\begin{proposition}\label{sg-1}
For $t\geqslant0$, the functional operators $T^{c,\pm}_t$ defined by \eqref{eq:Tt-+ rep} map $C(M,\R)$ to itself and satisfy the following properties:
\begin{enumerate}[(1)]
  \item for any $t,s\geqslant0$, $T^{c,\pm}_{t+s}=T^{c,\pm}_{t}\circ T^{c,\pm}_{s}$, that is, $\{T^{c,\pm}_t\}_{t\geq0}$ form two semigroups.

  \item for any $t\geqslant0$ and any two initial data $\varphi$ and $\psi$ with $\varphi\leqslant\psi$, one has $T^{c,\pm}_t\varphi\leqslant T^{c,\pm}_t\psi$.

  \item for any $t\geqslant0$ and any two initial data $\varphi$ and $\psi$, one has $\|T^{c,\pm}_t\varphi-T^{c,\pm}_t\psi\|_\infty\leqslant e^{\Lambda t}\|\varphi-\psi\|_\infty$.

  \item $(x,t)\mapsto T^{c,-}_t\varphi(x)$ is the unique solution to \eqref{hje} and $(x,t)\mapsto-T^{c,+}_t\varphi(x)$ is the unique solution to \eqref{reverse ham}.

  \item for any $t\geqslant0$ and $\varphi\in C(M,\R)$, one has $T^{c,-}_{t}\circ T^{c,+}_{t}\varphi\geqslant\varphi$ and $T^{c,+}_{t}\circ T^{c,-}_{t}\varphi\leqslant\varphi$.
\end{enumerate}
By (2) and (3), $\{T^{c,-}_t\}_{t\geqslant 0}$ and $\{T^{c,+}_t\}_{t\geqslant 0}$ are called the backward and forward solution semigroups to \eqref{HJe} respectively.
\end{proposition}

The next proposition builds a connection between the definitions in the first part and the $t$-monotonicity of the solution semigroups.
\begin{proposition}\label{mono}
A function $v\in C(M,\R)$ is a subsolution to \eqref{HJs} if and only if for any $x\in M$ and $t\geqslant0$,
\begin{equation}\label{eq:t-mono}
v(x)\leqslant\,\,T^{c,-}_t v(x)\quad(\text{resp.}\quad v(x)\geqslant\,\,T^{c,+}_t v(x)).
\end{equation}
Furthermore, we have
\begin{enumerate}[(1)]
  \item if $v$ is strict at every $x\in M$, then \eqref{eq:t-mono} holds \textbf{strictly}, that is for any $x\in M$ and $t>0$,
        \[
        v(x)<\,\,T^{c,-}_t v(x)\quad\text{and}\quad v(x)>\,\,T^{c,+}_t v(x).
        \]

  \item if\,\,$T^{c,-}_t v$ \big(resp. $T^{c,+}_t v$\big) is bounded above (resp. bounded below), then the uniform limit $\lim_{t\rightarrow\infty}T^{c,-}_t v\,\,\,(\text{resp. }\lim_{t\rightarrow\infty}T^{c,+}_t v)$ exists and belongs to $\mathcal{S}^{c,-}($ resp. $\mathcal{S}^{c,+})$, where $\mathcal{S}^{c,-}($ resp. $\mathcal{S}^{c,+})$ is the set of all fixed points of $\{T^{c,-}_t\}_{t\geqslant 0}($ resp. $\{T^{c,+}_t\}_{t\geqslant 0})$.

  \item $\cS$ coincides with the set of viscosity solutions to \eqref{hjs} and $v_+\in\mathcal{S}^{c,+}$ if and only if $-v_+$ is a solution to
      \begin{equation}\label{reverse eq}\tag{S$'_c$}
      \breve{H}(x,d_x u,u)=c,\quad x\in M.
      \end{equation}

  \item for each $u_-\in\mathcal{S}^{c,-}\,($resp. $v_+\in\mathcal{S}^{c,+})$, the uniform limit $\lim_{t\to+\infty}T^{c,+}_tu_-:=v_+ \,($resp. $\lim_{t\to+\infty}T^{c,-}_tv_+:=u_-)$ exists and belongs to $\mathcal{S}^{c,+}\,($resp. $\mathcal{S}^{c,-}$). In this case, for any $t\geqslant0$, the set $\{x\in M: T^{c,+}u_-(x)=v_+(x)\}\,($resp. $\{x\in M: T^{c,-}v_+(x)=u_-(x)\})$ is not empty.

\end{enumerate}
\end{proposition}

Concerning the large-time behavior of solution semigroups for general initial data, we have

\begin{proposition}\label{sg-2}
For any $\varphi\in C(M,\R)$ and $t_0>0$,
\begin{enumerate}[(1)]
  \item the function $(x,t)\mapsto T^{c,\pm}_{t}\varphi(x)$ is locally Lipschitz on $M\times(0,+\infty)$.

  \item if the family of functions on $M$, $\{T^{c,-}_{t}\varphi\}_{t\geqslant t_0}\,($resp. $\{T^{c,+}_{t}\varphi\}_{t\geqslant t_0})$ is uniformly bounded, then $\{T^{c,-}_{t}\varphi\}_{t\geqslant t_0}\,($resp. $\{T^{c,+}_{t}\varphi\}_{t\geqslant t_0})$ is uniformly Lipschitz in $x$.

  \item under the assumptions of (2), then $\liminf_{t\rightarrow+\infty}T^{c,-}_{t}\varphi\in\cS\,\,($resp.  $\limsup_{t\rightarrow+\infty}T^{c,+}_{t}\varphi\in\mathcal{S}^{c,+})$. In both cases, the limit functions are subsolutions to \eqref{hjs}.

  \item if $\varphi\leqslant T^{c,-}_{t_0}\varphi\,($resp. $\varphi\geqslant T^{c,+}_{t_0}\varphi)$, then either $\lim_{n\to+\infty}T^{c,-}_{nt_0}\varphi(x)=+\infty\quad  (\text{resp.}\  \lim_{n\to+\infty}T^{c,+}_{nt_0}\varphi(x)=-\infty)$ uniformly on $M$, or $\lim_{n\to+\infty}T^{c,-}_{nt_0}\varphi(x)/T^{c,+}_{nt_0}\varphi(x)=\varphi_\infty(x)$ uniformly on $M$, where $\varphi_\infty(x)\in\mathrm{Lip}(M,\R)$.
\end{enumerate}
\end{proposition}

\section{Proof of Corollary \ref{Cor:B2}}\label{section-Appendix-b}
In this part, we consider the solvability of \eqref{HJs} with $\lambda(x)\geqslant0$ in the critical case $c=c_0$, and then provide a proof of Corollary \ref{Cor:B2}. We begin by giving another definition of the critical value depending on $h$ and $\lambda$. This leads to a necessary and sufficient condition for the existence of a subsolution. Throughout this section, we assume (H1)-(H2) and that the discount factor satisfies $\lambda(x)\geqslant 0$, with $\lambda^{-1}(0)\neq\emptyset$ and $\max_{x\in M}\lambda(x)>0$. We define

\begin{equation}\label{dcv2}
c_\lambda(h): = \inf_{u \in C^\infty(M,\R)} \sup_{x \in \lambda^{-1}(0)} h(x, d_x u(x))
\end{equation}

\begin{lemma}\label{eq:dcv}
$c_\lambda(h)=c_0$.
\end{lemma}

\begin{proof}
By the definition \eqref{dcv}, we have
\begin{equation}
c_0=\inf_{u \in C^\infty(M,\R)}\bigg(\sup_{x \in M}\big[\lambda(x) u(x) + h(x, d_x u(x))\big]\bigg)\geqslant \inf_{u \in C^\infty(M,\R)}\bigg(\sup_{x \in \lambda^{-1}(0)} h(x, d_x u(x))\bigg) = c_\lambda(h)
\end{equation}
To prove the reverse inequality, for any $\varepsilon>0$, there is $g\in C^\infty(M,\R)$ such that
$$
\sup_{x \in \lambda^{-1}(0)} h(x, d_x g(x)) < c_\lambda(h) + \frac{\varepsilon}{2}.
$$
Due to the smoothness of \( h \) and \( g \), there is $\delta > 0$ such that
$$
 \sup_{x\in\lambda^{-1}([0, \delta])} h(x, d_x g(x)) < c_\lambda(h) + \varepsilon.
 $$
Set $E:=\sup_{x \in M} h(x, d_x g(x)), u_g:=g(x)-\|g\|_\infty-\frac{E+|c_\lambda(h)|}{\delta}\leqslant0$. Then
\begin{align*}
c_0\leqslant &\, \sup_{x \in M} \{ \lambda(x)u_g(x) + h(x, d_x u_g(x)) \} \\
\leqslant &\, \sup\left\{ \sup_{x \in \lambda^{-1}([0, \delta])}\{ \lambda(x) u_g(x) + h\left(x, d_x u_g(x)\right)\}, \sup_{x\in M}\{ \delta u_g(x) + h\left(x, d_x u_g(x)\right)\} \right\}\\
\leqslant &\, \sup\left\{ \sup_{x \in \lambda^{-1}([0, \delta])} h\left(x, d_x u_g(x)\right), -|c_\lambda(h)| \right\}  \leqslant c_\lambda(h) + \varepsilon.
\end{align*}
Since $\varepsilon$ is arbitrary, $c_0\leqslant c_\lambda(h)$. This completes the proof.
\end{proof}

We recall the definition of the Ma\~{n}\'{e} critical value for $h$ as
\begin{equation}\label{mane}
c(h)=\inf_{u\in C^\infty(M,\R)}\sup_{x \in M} h(x, d_x u(x))=\inf_{\substack{u\in C^\infty(M,\R)\\ u\leqslant0}}\sup_{x \in M} h(x, d_x u(x)).
\end{equation}
From \eqref{dcv2}, it is evident that for any $a\in\R$,
\begin{equation}\label{eq:C-1}
c_0=c_\lambda(h)=\inf_{u\in C^\infty(M,\R)} \sup_{x \in \lambda^{-1}(0)} a\lambda(x)+h(x, d_x u(x))\leqslant\inf_{u\in C^\infty(M,\R)} \sup_{x \in M} a\lambda(x)+h(x, d_x u(x))=c(h+a\lambda),
\end{equation}
where $h+a\lambda$ is the Tonelli Hamiltonian $h(x,p)+a\lambda(x)$. Moreover, we have

\begin{proposition}\label{prop:B1}
With the definitions \eqref{dcv2} and \eqref{mane},
\begin{enumerate}[(1)]
  \item For every $a\in\R, c_\lambda(h)\leqslant c(h+a\lambda)$ and $c_\lambda(h)=\lim_{a \to -\infty} c(h + a\lambda)$,
  \item The equation $(dS_{c_0})$ has a subsolution if and only if there is $a\in\R$ such that \(c_\lambda(h) = c(h + a\lambda)\).
\end{enumerate}
\end{proposition}

\begin{proof}
\noindent (1) Notice that $c(h+a\lambda)$ is monotone increasing in \(a\), so \eqref{eq:C-1} implies that $\lim_{a \to -\infty} c(h + a\lambda)$ exists. To show the limit is $c_\lambda(h)$, for any $\varepsilon>0$, we follow the proof of Lemma \ref{eq:dcv} to choose a smooth function $u_g$ such that
$$
\sup_{x \in M} \{ \lambda(x)u_g(x) + h(x, d_x u_g(x))\}\leqslant c_\lambda(h)+\varepsilon.
$$
Hence for $a\leqslant-\|u_g\|_\infty$, we have
\begin{equation}
c(h+a\lambda)\leqslant\sup_{x \in M} \{ \lambda(x)a + h(x, d_x u_g(x))\}\leqslant\sup_{x \in M} \{ \lambda(x)u_g(x) + h(x, d_x u_g(x))\}\leqslant c_\lambda(h)+\varepsilon,
\end{equation}
which finishes the proof.

\vspace{0.5em}
\noindent (2) \(\Leftarrow\) Due to the monotonicity of $c(h+a\lambda)$ in $a$, we can assume that for some $a\leqslant0$, \(c_\lambda(h) = c(h + a\lambda)\). Then there is a Lipschitz solution $\varphi$ to the equation
\[
a\lambda(x) + h(x, d_x \varphi(x)) = c(h + a\lambda) = c_\lambda(h)
\]
Set \(u_\varphi = \varphi - \max \varphi + a\in C(M,\R)\). The inequality
\[
\lambda(x) u_\varphi(x) + h(x, d_x u_\varphi(x)) \leqslant \lambda(x) a + h(x, d_x \varphi(x)) = c_\lambda(h)=c_0
\]
implies that $u_\varphi$ is a subsolution to $(dS_{c_0})$.

\vspace{0.5em}
\(\Rightarrow\) If $(dS_{c_0})$ has a continuous subsolution \(\varphi\), then the standard viscosity solution theory implies that $\varphi$ is Lipschitz. By smooth approximation, for all \(n \in \mathbb{N}\), there exists \(\varphi_n \in C^\infty(M)\) such that $\|\varphi_n - \varphi\|_\infty < \frac{1}{n}$ and
\begin{equation*}
\lambda(x) \varphi_n(x) + h(x, d_x \varphi_n(x)) < c_\lambda(h) + \frac{1}{n}.
\end{equation*}
Set \(a = \min_{x \in M} \varphi(x) - 1\). Then, by the above inequality,
\[
c(h + a\lambda) \leq\sup_{x\in M}\{a\lambda(x) + h(x, d_x \varphi_n(x))\} \leq\sup_{x\in M}\{ \lambda(x) \varphi_n(x) + h(x, d_x \varphi_n(x))\}< c_\lambda(h) + \frac{1}{n}
\]
By the arbitrariness of \(n\), \(c(h + a\lambda) \leq c_\lambda(h)\). Combining this with \eqref{eq:C-1}, we obtain \(c(h + a\lambda) = c_\lambda(h)\).	
\end{proof}

A direct consequence of the above proposition is the following 
\begin{corollary}\label{non-sol}
$(dS_{c_0})$ does not admit any subsolution if and only if $c(h+a\lambda)>c_\lambda(h)$ for every $a\in\R$.
\end{corollary}

To conclude this section, we give an example where $(dS_{c_0})$ does not admit any subsolution.
\begin{example}
Let $M=\mathbb{T}$ and $h(x,p) = \frac{1}{2}p^2+1-\cos x$ and let $\lambda(x)=(1-\cos x)^2\geqslant0$. Then $\lambda^{-1}(0)$ is a singleton and $c_\lambda(h)=0$. However, for any $a\in\mathbb{R}$,
$$
c(h+a\lambda) =\max\limits_{x\in S^1} (1-\cos x) + a(1-\cos x)^2 =\max\limits_{x\in S^1} (1-\cos x)\left(1 + a(1-\cos x)\right)>0=c_\lambda(h).
$$
Hence by Corollary \ref{non-sol}, $(dS_{c_0})$ does not admit any subsolution.
\end{example}

 \begin{lemma}\label{sol}
Let $l:TM\rightarrow\R$ be the Lagrangian associated with $h$, and fix $a>0$. Then $c(h)=c(h-a\lambda)$ if and only if there exists an $l$-minimizing measure
$\nu$ such that $\int_{TM}\lambda\,d\nu=0$.
\end{lemma}

\begin{proof}
Let $\mathcal H$ denote the class of closed probability measures on
$TM$ with finite first moment, namely
\[
\int_{TM} |v|_x\,d\nu(x,v)<\infty\quad\text{and}\quad \int_{TM}\langle d f(x),v\rangle\,d\nu(x,v)=0,
\qquad \forall f\in C^1(M).
\]
The closed-measure characterization of the Ma\~n\'{e} critical value
\cite[Section 4.8]{Fathi_book} gives
\[
-c(h-a\lambda)=\min_{\mu\in\mathcal H}\int(l+a\lambda)\,d\nu.
\]
Notice that, for every $a\in\R$, the minimum above is achieved by some probability measures that is invariant under the Euler-Lagrange flow defined by the Tonelli Lagrangian $l(x,v)+a\lambda(x)$.    

\vspace{1em}
$\Rightarrow$: Assume first that the two critical values coincide. Let $\mu_a$ minimizes
$\int (l+a\lambda)\, d\nu$ over $\mathcal H$. Since $\lambda\geqslant0$,
\[
-c(h)\leqslant\int l\,d\nu_a
\leqslant\int(l+a\lambda)\,d\nu_a
=-c(h-a\lambda)=-c(h).
\]
Both inequalities are equalities.  Since $a>0$, this implies
$\int\lambda\,d\nu_a=0$, and $\nu_a$ minimizes $l$ as well.

\vspace{0.5em}
$\Leftarrow$: Assume there is an $l$-minimizing measure $\nu_0\in\mathcal H$ satisfies
$\int\lambda\,d\nu_0=0$, then
\[
-c(h-a\lambda)=\min_{\mathcal H}\int(l+a\lambda)\,d\nu
\leqslant\int(l+a\lambda)\,d\nu_0=-c(h).
\]
The reverse inequality follows from the definition \eqref{mane} and the fact $h-a\lambda\leqslant h$. Therefore $c(h-a\lambda)=c(h)$.
\end{proof}

%

\medskip
\noindent{{\bf Proof of Corollary \ref{Cor:B2}:}} 
The equivalence \((1)\) and \((2)\) follows from Theorems
\ref{thm:1} and \ref{thm:2}. 

\vspace{0.5em}
$(2)\Rightarrow(3):$ If there is a \(l\)-minimizing measure \(\nu\) with \(\int_{TM}\lambda\,d\nu=0\), then, by Lemma \ref{sol}, for each 
    $a>0, c(h-a\lambda)=c(h)$. According to Lemma \ref{eq:dcv} and Proposition \ref{prop:B1}(1) gives $$
    c_0=c_\lambda(h)=\lim_{a\to+\infty}c(h-a\lambda)=c(h).
    $$

$(3)\Rightarrow(2):$ If $\int_{TM}\lambda\,d\nu>0$ for every \(l\)-minimizing measure \(\nu\), then Lemma \ref{sol} implies that \(c(h-\lambda)<c(h)\). By Lemma \ref{eq:dcv} and the monotonicity of $a\mapsto c(h-a\lambda)$, it follows that 
$$
c_0=c_\lambda(h)\leqslant c(h-\lambda)<c(h).
$$ 
\vspace{0.5em}
 $(1)\Rightarrow(4)$ is obvious: Let $u$ be a globally asymptotically stable solution of \eqref{eq:c(h)}. Since any other viscosity solution \(v\) is a fixed point of \(T_t^{c(h),-}\) for every \(t\geqslant0\), it follows that 
\[
v=\lim_{t\to +\infty} T_t^{c(h),-}v= u.
\]

\((4)\Rightarrow(2)\): We argue by contradiction: Assume 
\(c(h)=c_0\). By Proposition \ref{prop:B1} and Theorem D (1), the equation 
\begin{equation}\label{eq:cri}
\lambda(x)u+h(x,d_x u)=c(h)=c_0
\end{equation}
admit a solution \(u_-\) and an ergodic Mather measure 
$\mu_\ast\in\frak{M}^{c_0}_{u_-}$ such that $\int_{T^\ast M\times\R}\lambda\,d\mu_\ast=0$. Since \(\lambda\geqslant0\), this implies 
\begin{equation}\label{eq:null}
\operatorname{supp}\mu_\ast\subset\lambda^{-1}(0).
\end{equation}


Choose a Birkhoff generic point $z_*=(x_*,p_*,u_-(x_*))\in\operatorname{supp}\mu_*$
and let \(\gamma_*:(-\infty,0]\to M\) be the base projection of its
backward orbit. The curve \(\gamma_*\) is \(u_-\)-calibrated, and by the
invariance of \(\operatorname{supp}\mu_*\) under the contact Hamiltonian flow, we have $\lambda(\gamma_*(s))=0$ for any $s\leqslant0.$
Therefore, for every \(t>0\),
\begin{equation}\label{eq:zero-calibration}
u_-(x_*)-u_-(\gamma_*(-t))
=\int_{-t}^{0}
 \bigl[l(\gamma_*(s),\dot\gamma_*(s))-\lambda(\gamma_*(s))u_-(\gamma_*(s))+c_0\bigr]\,ds=\int_{-t}^{0}
 \bigl[l(\gamma_*(s),\dot\gamma_*(s))+c_0\bigr]\,ds.
\end{equation}
Notice that, for every \(\delta>0\), the function
\(u_--\delta\) is a viscosity subsolution of
\eqref{eq:c(h)}. Hence, by Proposition \ref{mono},
\begin{equation}\label{order}
u_--\delta\leqslant T_t^{c_0,-}(u_--\delta)
\leqslant T_t^{c_0,-}u_-=u_-.
\end{equation}
In particular, \(T_t^{c_0,-}(u_--\delta)\) is increasing and uniformly
bounded from above. Proposition \ref{mono} (2) shows that the uniform limit
\[
u_\delta:=
\lim_{t\to+\infty}T_t^{c_0,-}(u_--\delta)
\in\mathcal S^{c_0,-}.
\]
exists and is a solution to \eqref{eq:cri}.

\vspace{0.5em}
By the definition of backward solution semigroup and \eqref{eq:zero-calibration}, we have 
\begin{align*}
T_t^{c_0,-}(u_--\delta)(x_*)
\leqslant &
[u_-(\gamma_*(-t))-\delta]
+\int_{-t}^{0}
 \bigl[l(\gamma_*(s),\dot\gamma_*(s))-\lambda(\gamma_*(s))u_-(\gamma_*(s))+c_0\bigr]\,ds\\
 =&[u_-(\gamma_*(-t))-\delta]
+\int_{-t}^{0}
 \bigl[l(\gamma_*(s),\dot\gamma_*(s))+c_0\bigr]\,ds=u_-(x_*)-\delta.
\end{align*}
Combining this with \eqref{order}, we obtain
\[
T_t^{c_0,-}(u_--\delta)(x_*)=u_-(x_*)-\delta
\qquad\text{for every \(t\geqslant0\)}.
\]
 Thus $u_\delta(x_*)=u_-(x_*)-\delta<u(x_*)$, yielding that \(u_\delta\neq u_-\). As \(\delta>0\) is arbitrary, \eqref{eq:cri} admits a continuum of distinct solutions. This contradicts the uniqueness of $u_-$.


\section{Total divergence of the solution semigroups}\label{proof:thm4}
We assume throughout this section that \textbf{(H3)$_+$} holds and
\eqref{HJs} has no viscosity subsolution. We show that an equivalent statement of the non-existence of subsolutions to \eqref{HJs} is the total divergence of both solution semigroups. Notice that if \eqref{HJs} does not admit a subsolutions, then by Proposition \ref{prop:1}, $c\leqslant c_0$. Before going into the proof, we need

\begin{lemma}\label{lem:uniform-divergence}
   If $\phi\in C(M,\R)$ is a viscosity supersolution of \eqref{HJs}, then
\[
\lim_{t\to +\infty} T_t^{c,-}\phi=-\infty
\quad\text{uniformly on }M.
\]
\end{lemma}

\begin{proof}
We claim that  
$$
 \lim_{t\to +\infty} T_t^{c,-}\phi(x) =-\infty, \quad \forall \, x\in M,
 $$  
 \textit{Proof of the claim:} Note that the family $T_t^{c,-}\phi$ is pointwise non-increasing in $t$ by Proposition \ref{mono}. We argue by contradiction: assume that $\{T_t^{c,-}\phi\}$ is bounded on some $x_0$, i.e., there exists $K_0\geqslant 0$ and $x_0\in M$ such that
$|T_t^{c,-}\phi(x_0)|\leqslant K_0$ for any $t>0$. It follows that, for each $x \in M$ and $t > 1 $,  
\begin{align*}
  \phi(x)\geqslant  T_t^{c,-} \phi(x) &\geqslant (T_1^{c,+} \circ T_{t+1}^{c,-} ) \phi(x) = \sup_{y \in M} h^{y,T_{t+1}^{c,-} \phi(y)}(x,1) 
     \geqslant h^{x_0,T_{t+1}^{c,-} \phi(x_0)}(x,1) \\
     & \geqslant h^{x_0,-K_0}(x,1) 
    \geqslant m_0:=\min _{z_1,z_2 \in M} h^{z_1,-K_0}(z_2,1)
    > - \infty.
\end{align*}
 Hence, $\{T_t^{c,-}\phi\}_{t\geqslant 0}$ is uniformly bounded by  $\displaystyle K_1:=\max\{\| \phi \|_\infty,|m_0| \}$. 
 
 Now we show the equi-Lipschitz continuity of $\{T_t^{c,-} \phi\}_{t\geqslant2}$ in $x\in M$. As
\begin{align*}
    | T_t^{c,-} \phi(x) - T_t^{c,-} \phi(y) | & = | \inf _{z \in M} h_{z,T_{t-1}^{c,-} \phi(z)} (x,1) -  \inf _{z \in M} h_{z,T_{t-1}^{c,-} \phi(z)} (y,1) | \\
    & \leqslant  \sup _{z \in M} | h_{z,T_{t-1}^{c,-} \phi(y)} (x,1) - h_{z,T_{t-1}^{c,-} \phi(y)} (y,1)|.
\end{align*}
Since $h_{.,.}(.,1)$ is uniformly Lipschitz on $M \times [-K_1,K_1] \times M$, it follows that $\{T_t^{c,-}\phi\}_{t\geqslant 2}$ is uniformly Lipschitz in $x$. Hence, the limit $\lim_{t\to +\infty}T_t^{c,-} \phi(x)$ exists and is a viscosity solution of \eqref{HJs}. This is a contradiction, completing the proof of the claim.

\vspace{0.5em}
The convergence is uniform. Indeed, for $K>0$ the open sets
$U_t:=\{x:T_t^{c,-}\phi(x)<-K\}$ increase in $t$ and cover $M$. The compactness of $M$, together with the monotonicity of $U_t$ in $t$, provides time $T(K)$ after which $U_t=M$.
\end{proof}

\begin{proposition}\label{below c0}
If \eqref{HJs} does not admit any subsolution, then for any $\varphi\in C(M,\R)$,
\[
\lim_{t\to +\infty} T_t^{c,-} \varphi=-\infty\quad \text{uniformly on}\,\,\,M.
\]
This conclusion is stated in Remark \ref{thm:4}.
\end{proposition}

\begin{proof}
For $c_1>c_0\geqslant c$, it is clear that $u^{c_1}_-$ is a supersolution to \eqref{HJs}. We choose $\psi\in \mathcal{S}\mathcal{S}^{c_1}$. For any $\delta>1$,
\begin{equation}\label{eq:appendix-C-pf-1}
\varphi^\delta:=\delta u_-^{c_1}-(\delta-1)\psi.
\end{equation}
By the convexity assumption (H1), $h(x,d_xu_-^{c_1})\leqslant\frac{1}{\delta}h(x, d_x\varphi^\delta)+(1-\frac{1}{\delta})h(x, d_x\psi)$. Hence,
\begin{equation}\label{eq:E2}
\begin{split}
\lambda(x)\varphi^\delta+  h(x, d_x\varphi^\delta(x))\geqslant\,\delta [ \lambda(x) u_-^{c_1}+ h(x, d_x  u_-^{c_1} ) ]- (\delta-1) [ \lambda(x)\psi + h(x, d_x \psi)]>c_1>c.
\end{split}
\end{equation}
Since $u_-^{c_1}$ is semiconcave and $\psi\in C^\infty$, then \eqref{eq:appendix-C-pf-1} implies that $\varphi^\delta$ is also semiconcave. Thus, by \eqref{eq:E2}, we find that $\varphi^\delta$ is a supersolution to \eqref{HJs}. It follows from Lemma \ref {lem:uniform-divergence} that for any $\delta>1$,
$$
\lim_{t\to +\infty}T_t^{c,-}[\varphi^\delta]=-\infty.
$$
Since $u_-^{c_1}> \psi$ on the compact manifold $M$,
$\varphi^\delta\rightarrow+\infty$ uniformly as $\delta\to\infty$.  Given
$\varphi$, choose $\delta$ so that $\varphi\leqslant\varphi ^\delta$.
Order preservation and Lemma \ref{lem:uniform-divergence} yield
\[
\lim_{t\to +\infty} T_t^{c,-}\varphi \leqslant  \lim_{t\to +\infty}T_t^{c,-}\varphi^\delta=-\infty.
\]
\end{proof}

\section{Proof of Proposition \ref{prop:JYZ-new}}\label{Appendix:D}

\subsection{The sharp rate from below for general contact Hamiltonians}
\label{app:general-contact-below}

Let $H:T^*M\times\mathbb R\to\mathbb R$ be a $C^3$ contact Hamiltonian satisfying the assumptions of Appendix A.2, that is (H1)-(H2) and $|\partial_u H|\leqslant \Lambda$. Let $L$ be the Lagrangian associated with $H$. We consider the Hamilton-Jacobi equations \eqref{hjs} and \eqref{hje} with $c=0$ and write $T_t^-$ for the backward semigroup to the equation (E$_0$). In this appendix, we fix a solution $u_-$ of (S$_0$). 

\vspace{0.5em}
We assume
\begin{equation}\label{B1}
 a:=\inf_{\mu\in\mathfrak M_{u_-}}
 \int_{T^*M\times\mathbb R}\partial_u H(x,u,p)\,d\mu>0.
\end{equation}
Here $\mathfrak M_{u_-}$ is the set of Mather measures associated with $u_-$. In \cite{RWY}, the authors prove the asymptotic stability of $u_-$ under the assumption \eqref{B1} by establishing the following lower bound for the exponential convergence of $T^-_t$ near $u_-$.

\begin{theorem}\cite[Theorem 1.2]{RWY}\label{RWY-1.2}
There exists $\Delta>0$ such that, for every
$\varphi\in C(M,\mathbb R)$ with
$\|\varphi-u_-\|_\infty<\Delta$,
\begin{equation}\label{exp-cov}
 \limsup_{t\to\infty}\frac1t
 \ln\|T_t^-\varphi-u_-\|_\infty\leq-a.
\end{equation}
\end{theorem}

\begin{remark}
We refer to \cite{NY} for a PDE treatment of the stability issue, in which the authors employ a dual version of \eqref{B1}. 
\end{remark}

The aim of this appendix is to show that the above lower bound can be attained by certain initial data and the convergence rate of such initial data is well-defined. More precisely, we have    
\begin{proposition}\label{prop:JYZ-new}
Under \eqref{B1}, let $\Delta$ be as in Theorem
\ref{RWY-1.2}. For every $\varphi\in C(M,\mathbb R)$ satisfying
\[
 \|\varphi-u_-\|_\infty<\Delta,
 \qquad \varphi(x)<u_-(x)\quad\text{for every }x\in M,
\]
one has
\begin{equation}\label{eq:B-2}
 \lim_{t\to\infty}\frac1t
 \ln\|T_t^-\varphi-u_-\|_\infty=-a.
\end{equation}
\end{proposition}
\begin{lemma}\label{lem5}
For every $\delta\in(0,\Delta)$, set $u_\delta=u_--\delta$.
Then
\begin{equation}\label{eq:general-contact-lower-rate}
 \liminf_{t\to\infty}\frac1t
 \ln\|T_t^-u_\delta-u_-\|_\infty\geqslant  -a.
\end{equation}
\end{lemma}

\begin{proof}
Note that the set $\frak{M}_{u_-}$ is weakly compact, then the infinum in \eqref{B1} is attained by an ergodic measure $\mu_*$. We fix a $\mu_*$-generic point $z_*$. Let $\gamma_*:(-\infty,0]\to M$ be the $u_-$-calibrated curve with $x_*=\gamma_*(0)$. Set 
$$
b(s):=\partial_u L\bigl(\gamma_*(s),\dot\gamma_*(s),
 u_-(\gamma_*(s))\bigr),
$$
The relation
$\partial_u L=-\partial_u H$ and Birkhoff's ergodic theorem gives
\begin{equation}\label{eq:general-contact-birkhoff}
 \lim_{t\to\infty}\frac1t\int_{-t}^0 b(s)\,ds=-a,
 \qquad
.
\end{equation}
We define $U(x,t)=T_t^-u_\delta(x)$ and $E(t)=\|U(\cdot,t)-u_-\|_\infty$. Proposition \ref{sg-1}(2) implies that  
$$
U(x,t)=T_t^-u_\delta(\cdot  )\leqslant u_-\quad \text{for any}\,\,r>0,
$$ 
and Theorem \ref{RWY-1.2} implies $\lim_{t\to +\infty} T_t^-u_\delta(x  )=u_-(x)$ uniform on $M$. Then $\lim_{t\to +\infty}E(t)=0$, and it follows that 
$$
B:=\sup_{t\geqslant0}E(t)<\infty.
$$  
Since the closure of
$\{(\gamma_*(s),\dot\gamma_*(s),u_-(\gamma_*(s))):s\leqslant 0\}$
is compact. By continuity of $\partial_u L$, there is a bounded
nonnegative nondecreasing function $\omega$ on $[0,B]$, with
$\lim_{\rho \to 0} \omega(\rho)=0$, such that
\begin{equation}\label{eq:general-contact-modulus}
 \left|\partial_u L\bigl(\gamma_*(s),\dot\gamma_*(s),
 u_-(\gamma_*(s))-q\bigr)-b(s)\right|
 \leqslant \omega(\rho)
\end{equation}
whenever $s\leqslant 0$ and $0\leqslant q\leqslant \rho\leqslant B$.
Fix $t>0$. For $-t\leqslant s\leqslant 0$ define
\[
 w_t(s)=u_-(\gamma_*(s))-U(\gamma_*(s),s+t).
\]
Then $0\leqslant w_t(s)\leqslant E(s+t)$ and $w_t(-t)=\delta$.

\vspace{0.5em}
For $-t<s_1<s_2\leq0$, the definition of solution semigroup and our choice of $\gamma_\ast$ give
\[
 \begin{aligned}
 U(\gamma_*(s_2),s_2+t)-U(\gamma_*(s_1),s_1+t)
 &\leqslant \int_{s_1}^{s_2}
 L\bigl(\gamma_*(s),\dot\gamma_*(s),
 U(\gamma_*(s),s+t)\bigr)\,ds.\\
 u_-(\gamma_{\ast}(s_2))-u_-(\gamma_{\ast}(s_1))&=\int_{s_1}^{s_2}
 L\bigl(\gamma_*(s),\dot\gamma_*(s),
 u_-(\gamma_*(s))\bigr)\,ds.
 \end{aligned}
 \]
Subtracting the first inequality from the second equality, we obtain that, for almost every $s\in(-t,0)$,
\[
 \begin{aligned}
 \dot w_t(s)
 &\geq L\bigl(\gamma_*(s),\dot\gamma_*(s),u_-(\gamma_*(s))\bigr)
       -L\bigl(\gamma_*(s),\dot\gamma_*(s),
           u_-(\gamma_*(s))-w_t(s)\bigr)\\
 &=w_t(s)\int_0^1\partial_u L\bigl(\gamma_*(s),\dot\gamma_*(s),
               u_-(\gamma_*(s))-\theta w_t(s)\bigr)\,d\theta\\
 &\geqslant \bigl[b(s)-\omega(E(s+t))\bigr]w_t(s).
 \end{aligned}
\]
The function $w_t$ is locally Lipschitz on $(-t,0]$ and
continuous at $-t$. Integrating from $-t+\varepsilon$ to $0$ and then sending $\varepsilon$ to $ 0$, using boundedness of the coefficients,
we obtain
\begin{equation}\label{eq:general-contact-direct-bound}
 \begin{aligned}
 \|T_t^-u_\delta-u_-\|_\infty
 &\geqslant u_-(x_\ast)-T_t^-u_\delta(x_\ast)=w_t(0) \geqslant \delta\exp\left\{
       \int_{-t}^0 b(s)\,ds
       -\int_0^t\omega(E(r))\,dr\right\}.
 \end{aligned}
\end{equation}
Since $\lim_{t\rightarrow+\infty}E(t)\to0$ and $\omega$ is bounded with
$ \lim_{\rho \to 0} \omega(\rho)=0$, we use dominated convergence theorem to conclude that
\begin{equation}\label{eq:int-omega}
 \lim_{t\to\infty}\frac1t\int_0^t\omega(E(r))\,dr=0.	
\end{equation}
We complete the proof by combining \eqref{eq:general-contact-birkhoff}, \eqref{eq:general-contact-direct-bound} and \eqref{eq:int-omega}.
\end{proof}

\begin{proof}[Proof of Proposition \ref{prop:JYZ-new}]
The assumptions on $\varphi$ and the compactness of $M$ imply
\[
 \delta:=\min_{x\in M}(u_-(x)-\varphi(x))>0.
\]
It follows that $u_--\Delta<\varphi\leq u_\delta$. Hence Proposition \ref{sg-1} (2) yields $T_t^-\varphi\leq T_t^-u_\delta\leq u_-$ and
\[
 \|T_t^-\varphi-u_-\|_\infty
 \geq\|T_t^-u_\delta-u_-\|_\infty.
\]
Now Lemma \ref{lem5} and Theorem
\ref{RWY-1.2} help us to complete the proof.
\end{proof}

\begin{remark}
The same proof applies to any initial datum satisfying
$\varphi<u_-$ on $M$ and $T_t^-\varphi\to u_-$ uniformly,
provided the upper rate in \eqref{exp-cov} is obtained by entering
the local neighborhood of $u_-$. Indeed, one can repeat the proof by comparing the evolution
of $\varphi$ directly along $\gamma_*$ and replacing the initial value $\delta$ by its positive lower bound
$\min_M(u_--\varphi)$ in the proof above.
\end{remark}

\begin{remark}
For the discounted Hamiltonian $H_c(x,p,u)=\lambda(x)u+h(x,p)-c$, we can offer a simpler proof for Proposition \ref{prop:JYZ-new}: let $l(x,v)$ be the convex dual of $h(x,p)$.  For
$\varphi\in C(M,\mathbb R)$, the solution semigroup has the
explicit representation
\begin{equation}\label{eq:D-explicit}
T_t^{c,-}\varphi(x)=
\inf_{\substack{\gamma\in\operatorname{AC}([-t,0],M)\\ \gamma(0)=x}}
\left\{
e^{-\int_{-t}^{0}\lambda(\gamma(\tau))\,d\tau}\varphi(\gamma(-t))
+\int_{-t}^{0}e^{-\int_s^0\lambda(\gamma(\tau))\,d\tau}
\bigl[l(\gamma(s),\dot\gamma(s))+c\bigr]\,ds
\right\}.
\end{equation}
Let $c>c_0$ and $u=u_-^c$ be the asymptotically stable solution of \eqref{HJs}. Hence $a:=\inf_{\mu\in\mathfrak M_{u^c_-}^c}\int_{T^*M\times\mathbb R}\lambda\,d\mu>0$ and the infimum is achieved by some Mather measure $\mu_\ast$. We choose a Birkhoff generic point $x^*\in\text{supp}\,\mu_{\ast}$ and a $u^c_-$-calibrated curve $\gamma_*:[-\infty,0]\rightarrow M$ with $\gamma_*(0)=x_\ast$. It follows from the Birkhoff ergodic theorem that
\begin{equation}
\lim_{t\rightarrow\infty}\frac{1}{t}\int_{-t}^{0}\lambda(\gamma(s))\,ds=a.
\end{equation}  
 We insert $\gamma_*$ into
\eqref{eq:D-explicit}, with $\varphi=u^c_-$ and $u^c_- -\delta$ to obtain
\begin{equation}\label{eq:D-direct-lower}
u^c_-(x_*)-T_t^{c,-}(u^c_--\delta)(x_*)
\geqslant
\delta\exp\left(-\int_{-t}^{0}\lambda(\gamma_*(s))\,ds\right).
\end{equation}
By Proposition \ref{mono}, $T_t^{c,-}(u^c_--\delta)\leqslant u^c_-$.  Hence
\eqref{eq:D-direct-lower} imply
\begin{align*}
&\liminf_{t\to\infty}\frac1t
\ln\|T_t^{c,-}(u^c_--\delta)-u^c_-\|_\infty
\geqslant\liminf_{t\to\infty}\frac1t
\ln\big[u^c_-(x^*)-T_t^{c,-}(u^c_--\delta)(x^*)\big]\\
\geqslant&\,\liminf_{t\to\infty}\frac1t
\ln\bigg[\delta\exp\left(-\int_{-t}^{0}\lambda(\gamma_*(s))\,ds\right)\bigg]= -a.
\end{align*}
The reverse inequality follows from \eqref{exp-cov}.  This proves that
for every $\delta\in(0,\Delta)$,
$$
\lim_{t\to\infty}\frac1t
\ln\|T_t^{c,-}(u^c_--\delta)-u^c_-\|_\infty=-a.
$$
\end{remark}

  \bigskip

\noindent   {\bf The authors declare that they have no conflicts of interest and that no data was used in this paper.}
  
  	\medskip
\section*{Acknowledgments}
L.Jin is partly supported by the National Key R\&D Program of China 2021YFA1001600. All of the author are supported in part by the National Natural Science Foundation of China (Grant No. 12171096). L.Jin is also supported in part by the NSFC (Grant No. 12371186, 12671222). J.Yan is also supported in part by the NSFC (Grant No. 12231010,12571197). K.Zhao is also supported by National Natural Science Foundation of China (Grant No. 12301233). The authors also thank Prof. K.Z. Wang, Prof. H. Tran, Dr.P.R. Ni and Dr.Y.Q.Ruan for interesting discussions on this topic.


\medskip

\begin{thebibliography}{99}\small
\addcontentsline{toc}{section}{References}
\renewcommand{\baselinestretch}{0.0}
\setlength\itemsep{-2pt}

\bibitem{AA}
S.Allais, M-C.Arnaud, {\it The dynamics of conformal Hamiltonian flows: dissipativity and conservativity}. Rev. Mat. Iberoam. 40 (2024), no.3, 987-1021.




\bibitem{AF}
M-C.Arnaud, J.Fejoz: {\it Invariant submanifolds of conformal symplectic dynamics}. J. \'{E}c. polytech. Math. 11 (2024), 159-185.








\bibitem{AHV}
M-C.Arnaud, V.Humili\`{e}re, C.Viterbo: {\it Higher dimensional Birkhoff attractors. With an appendix by M.Zavidovique}. arXiv:2404.00804v2.




\bibitem{AS}
M-C.Arnaud, X.F.Su: {\it On the $C^1$ and $C^2$-convergence to weak K.A.M. solutions}. Comm. Math. Phys. 392 (2022), no. 3, 825-861.



\bibitem{Be} P. Bernard: {\it Existence of $C^{1,1}$ critical sub-solutions of the Hamilton-Jacobi equation on compact manifolds}, Ann. Sci. \'Ecole Norm. Sup. 40 (2007), 445-452.




\bibitem{BLL}
C. Bertucci, J.-M. Lasry, P.-L. Lions: {\it The equilibrium price of bubble assets}. arXiv:2505.15578, 2025.





\bibitem{CCWY}
P. Cannarsa, W. Cheng, K. Wang and J. Yan: {\it Herglotz' generalized variational principle and  contact type Hamilton-Jacobi equations, Trends in Control Theory and Partial Differential Equations.} 39--67. Springer INdAM Ser., \textbf{32}, Springer, Cham, 2019.





\bibitem{CCJWY}
P. Cannarsa, W. Cheng, L. Jin, K. Wang and J. Yan: {\it Herglotz' variational principle and Lax-Oleinik evolution}. J. Math. Pures Appl. 141 (2020), 99-136.





\bibitem{CS}
P. Cannarsa, C. Sinestrari: {\it Semiconcave functions, Hamilton-Jacobi equations, and optimal control}. Progress in Nonlinear Differential Equations and their Applications,  58. Birkh\"{a}user Boston, Inc., Boston, MA,  xiv+304 pp, (2004)

 
\bibitem{CCIZ}
Q. Chen, W. Cheng, H. Ishii, and K. Zhao: {\it  Vanishing contact structure problem and convergence of the viscosity solutions}. Comm. Partial Differential Equations, 44 (2019),  801-Ð836.
 


\bibitem{CFZZ}
Q. Chen, A. Fathi, M. Zavidovique, J. Zhang: {\it Convergence of the solutions of the nonlinear discounted Hamilton-Jacobi equation: The central role of Mather measures}, J. Math. Pures. Appl, 181, (2024), 22-57.





\bibitem{CL}
M. Crandall,P.L. Lions: {\it Viscosity solutions of Hamilton-Jacobi equations}. Trans. Amer. Math. Soc.  277 ,(1983) 1-42.






\bibitem{Crandall_Evans_Lions1984}
M. Crandall, L. Evans, and P.-L. Lions: {\it Some properties of viscosity solutions of Hamilton-Jacobi equations}. Trans. Amer. Math. Soc.  282 , (1984) 487--502.





\bibitem{CIP}
G. Contreras, R. Iturriaga, G. P. Paternain, M. Paternain: {\it Lagrangian graphs, minimizing measures and Ma\~{n}\'{e}'s critical values}. Geometric and Functional Analysis, 8(5) (1998), 788-809.


\bibitem{DFIZ}
A. Davini, A. Fathi, R. Iturriaga, M. Zavidovique: {\it Convergence of the solutions of the discounted Hamilton-Jacobi equation: convergence of the discounted solutions}. Invent. Math.  206 (2016),29-55.


\bibitem{DNYZ}
 A. Davini, P. Ni, J. Yan, M. Zavidovique: {\it Convergence/divergence phenomena in the vanishing discount limit of Hamilton-Jacobi equations}, arXiv:2411.13780, 2024.



\bibitem{Fathi_book}
A. Fathi: {\it Weak KAM Theorem in Lagrangian Dynamics}. preliminary version 10, Lyon,
unpublished (2008).






\bibitem{FM_noncompact}
A. Fathi, E. Maderna: {\it Weak KAM theorem on non-compact manifolds}. Nonlinear Differential Equations and Applications Nodea,   14(1-2) (2007),1-27.




\bibitem{FS}
A. Fathi and A. Siconolfi:  {\it Existence of $C^1$ critical subsolutions of the Hamilton-Jacobi equation}, Invent. math. \textbf{155} (2004), 363--388.

 
\bibitem{GMT}
 D. A. Gomes, H. Mitake, and H. V. Tran:{\it  The selection problem for discounted HamiltonÐJacobi equations: some non-convex cases}, J. Math. Soc. Japan. 70 (2018),  345Ð-364.
 


\bibitem{Ishii_chapter}
H. Ishii: {\it A short introduction to viscosity solutions and the large time behavior of solutions of Hamilton-Jacobi equations}. Hamilton-Jacobi equations: approximations, numerical analysis and applications, 111-249, Lecture Notes in Math., 2074, Fond. CIME/CIME Found. Subser., Springer, Heidelberg (2013)

 
\bibitem{IJ}
H. Ishii and L. Jin:{\it  The vanishing discount problem for monotone systems of HamiltonÐJacobi equations. Part 2: nonlinear coupling} Calc. Var. Partial Differential Equations, 59 (2020), 1Ð-28.

\bibitem{IMT1}
H. Ishii, H. Mitake, and H. Tran: {\it  The vanishing discount problem and viscosity Mather measures. Part 1: The problem on a torus} J. Math. Pures Appl., 108 (2017), pp. 125Ð149.

\bibitem{IMT2}
 H. Ishii, H. Mitake, and H. Tran: {\it The vanishing discount problem and viscosity Mather measures. Part 2: Boundary value problems} J. Math. Pures Appl., 108 (2017), pp. 261Ð305.
 


\bibitem{JMT}
W. Jing, H. Mitake and H.V. Tran: {\it Generalized ergodic problems: existence and uniqueness structures of solutions}. J. Differential Equations  268, no. 6,(2020) 2886-2909.





\bibitem{JYZ}
L.Jin, J.Yan, K.Zhao: {\it Nonlinear semigroup approach to Hamilton-Jacobi equation$-$a toy model}. Minimax Theory Appl. 8 (2023), no. 1, 61-84.




\bibitem{JYZ2}
L.Jin, J.Yan, K.Zhao: {\it Variational construction of asymptotic orbits in contact Hamiltonian systems}. J. Differential Equations 462 (2026), Paper No. 114141, 37 pp.




\bibitem{MS}
S.Mar\`{o}, A.Sorrentino: {\it Aubry-Mather theory for conformally symplectic systems}. Comm. Math. Phys. 354 (2017), no. 2, 775¨C808.


 
\bibitem{MT}
H. Mitake and H. Tran: {\it Selection problems for a discount degenerate viscous HamiltonÐJacobi equation}. Adv. Math. 306 (2017), 684-Ð703.

\bibitem{NW}
P. Ni, L.Wang: {\it A nonlinear semigroup approach to Hamilton-Jacobi equations-revisited}. J. Differential Equations 403 (2024), 272-307.

\bibitem{NY}
P. Ni, J. Yan: {\it A PDE formulation of Lyapunov stability for contact-type Hamilton-Jacobi equations} arXiv.2604.24329, 2026

\bibitem{NYZ}
P. Ni, J. Yan, M. Zavidovique: {\it Static class-guided selection of elementary solutions in non-monotone
vanishing discount problems}, arXiv:2602.09697, 2026.


\bibitem{LPV_Hom}
P.L.Lions, G.Papanicolaou, S.R.S. Varadhan: {\it Homogenization of Hamilton-Jacobi equations},unpublished work (1987)





\bibitem{RWY}
Y.Q. Ruan, K. Wang, J. Yan: {\it Lyapunov stability and uniqueness problems for Hamilton-Jacobi Equations Without Monotonicity}. Commun. Math. Phys. 406:157(2025).





\bibitem{RWY2}
Y.Q. Ruan, K. Wang, J. Yan: {\it Lyapunov stability and uniqueness problems for Hamilton-Jacobi Equations Without Monotonicity II}.





\bibitem{Tran}
H.V.Tran, {\it Hamilton-Jacobi equations¡ªtheory and applications}. Grad. Stud. Math., 213. American Mathematical Society, Providence, RI, 2021, xiv+322 pp.




\bibitem{WWY1}
K. Wang, L. Wang and J. Yan: {\it Implicit variational principle for contact Hamiltonian systems}. Nonlinearity 30 (2017), 492--515.





\bibitem{WWY2}
K. Wang, L. Wang and J. Yan: {\it Variational principle for contact Hamiltonian systems and its applications}. J. Math. Pures Appl.  123 (2019), 167--200.





\bibitem{WWY3}
K. Wang, L. Wang and J. Yan: {\it Aubry-Mather theory for contact Hamiltonian systems}. Commun. Math. Phys.  366 (2019), 981--1023.





\bibitem{WY2021}
K. Wang and J. Yan: {\it Viscosity solutions of contact Hamilton-Jacobi equations without monotonicity assumptions}. arXiv: 2107.11554.



\bibitem{WYZ} Y. Wang, J. Yan, and J. Zhang:{\it Convergence of viscosity solutions of generalized contact HamiltonÐJacobi equations}. Arch. Rational Mech. Anal. 241 (2021),  1Ð-18.

\bibitem{Z}
M. Zavidovique: {\it Convergence of solutions for some degenerate discounted Hamilton--Jacobi equations}, Analysis \& PDE, 15 (2022),  1287--1311.

\bibitem{Zbook}
M. Zavidovique: {\it Discrete and Continuous Weak KAM Theory: an introduction through examples and its applications to twist maps}, Lecture Notes in Mathematics, Springer, 2025.
 
\end{thebibliography}

\end{document}